\documentclass[11pt, twoside]{article}
\usepackage{amsmath}
\usepackage{amsthm}
\usepackage{amssymb}
\usepackage{graphics}
\usepackage[letterpaper]{geometry}
\usepackage{url}
\numberwithin{equation}{section}

\newcommand{\matgrp}[3]{\ensuremath{#1^{#2 \times #3}}}
\newcommand{\matring}[2]{\matgrp{#1}{#2}{#2}}

\newcommand{\F}{\ensuremath{\textup{\textsf{F}}}}

\newcommand{\Z}{\ensuremath{\mathbb{Z}}}

\newcommand{\Q}{\ensuremath{\mathbb{Q}}}

\newcommand{\R}{\ensuremath{\mathbb{R}}}

\theoremstyle{plain}
\newtheorem{theorem}{Theorem}[section]
\newtheorem{lemma}[theorem]{Lemma}
\newtheorem{corollary}[theorem]{Corollary}

\theoremstyle{definition}

\allowdisplaybreaks
\begin{document}
\title{Black Box Linear Algebra: Extending Wiedemann's Analysis of a Sparse Matrix
  Preconditioner for Computations over Small Fields}
\author{Wayne Eberly \\
       Department of Computer Science\\
       University of Calgary}
\maketitle

\begin{abstract}
Wiedemann's paper, introducing his algorithm for sparse and structured matrix computations over
arbitrary fields, also presented
a pair of matrix preconditioners for computations over small fields. The analysis of the second of these is
extended in order to provide more explicit statements of the expected number of nonzero entries in the matrices
obtained as well as bounds on the probability that such matrices have maximal rank.

This is part of ongoing work to
establish that this matrix preconditioner can also be used to bound the number of nontrivial nilpotent blocks in the Jordan
normal form of a preconditioned matrix, in such a way that one can also sample uniformly from the null space
of the originally given matrix. If successful this will result in a black box algorithm for the type of matrix computation
required when using the number field sieve for integer factorization that is provably reliable and --- by
a small factor --- asymptotically more efficient than alternative techniques that make use of other matrix
preconditioners or require computations over field extensions.
\end{abstract}

\section{Introduction}

Suppose that $\F = \F_q$ is a finite field with size~$q$. Let $m$ and~$n$ be integers such that $0 \le m \le n$.
The paper that introduced Wiedemann's algorithm~\cite{wied86} also includes a proof of the following claim ---
which concerns an $n \times n$ matrix obtained by appending an additional set of row vectors to a matrix
$A \in \matgrp{\F}{m}{n}$ with maximal rank:\footnote{Wiedemann attributes much of the proof of this claim to
an anonymous referee who is thanked for allowing this work to be included.}

\begin{quote}
\textbf{Theorem~$\text{\textbf{1}}'$ [Wiedemann]}:
\emph{Numbers~$\epsilon > 0$ and~$c_1$ exist, both independent of~$q$, with
the following property: For any integers $n > m \ge 0$ a random procedure exists for generating $n-m$ row vectors
with length~$n$ such that if $A$ is an $m \times n$ matrix of rank~$m$, then with probability at least~$\epsilon$,
the resulting $n \times n$ matrix is nonsingular and the total Hamming weight of the generated rows is at most
$1 + c_1 n \log n$.}
\end{quote}

Unfortunately the unknown constants~$\epsilon$ and~$c_1$ are neither supplied nor estimated. Furthermore, it
seems that if the proof in~\cite{wied86} is applied without change in order to determine these values then
either~$c_1$ must be so large or $\epsilon$ so tiny that the result is of limited practical interest.

This is, somewhat, rectified in Section~\ref{sec:modified_proof}: While the outline of Wiedemann's argument
is maintained, along with the details of several steps, several other components are revised or replaced 
entirely in order to remove unnecessary bounds on various parameters and to simplify the estimation of
the unknown parameters~$\epsilon$ and~$c_1$. The bound~$\epsilon$ is also increased by adding 
another $\ell$ rows to the resulting matrix~$B \in \matgrp{\F}{(n+\ell)}{n}$ rows; here, $\ell$ depends on
the size of the field~$\F$. In particular, the following result is obtained.

\begin{theorem}
\label{thm:constants_for_conditioner}
Let $\F = \F_q$ be the finite field with size~$q$. Let $m$ and~$n$ be integers such that $n \ge m \ge 0$.
Let $\ell$ be a nonnegative integer and let $\sigma$, $\tau$ and $\upsilon$ be positive constants (depending on~$q$,
but independent of~$n$ and~$m$) as given in Table~\ref{fig:summary_table} on
page~\pageref{fig:summary_table}.

A random procedure exists for generating $n - m + \ell$ rows with  length~$n$ such that if $A$ is an $m \times n$
matrix of rank~$m$, then an additional $m + \ell$ rows (each with length~$n$) are produced, and the expected
number of nonzero entries in these rows is $\sigma n \ln n + \tau n$ if $n - m \ge \upsilon$, and at most $\frac{q-1}{q} (n-m+\ell) n
\le \frac{q-1}{q} (\upsilon + \ell) n$, otherwise.

If $q \le n^2$ then the matrix $B \in \F^{(n+\ell) \times n}$ obtained from~$A$ by adding these rows has maximal
rank~$n$ with probability at least~$\frac{9}{10}$. If $q > n^2$ then this matrix has maximal rank with probability
at least $\frac{9}{10} - \frac{9}{10n}$.
\end{theorem}

The probability bounds listed above are quite arbitrary. The parameter~$\sigma$ does not depend on this probability. Formulas
for~$\ell$, $\tau$ and~$\upsilon$, depending on the field size~$q$ and an arbitrarily small failure probability~$\epsilon$, are given
in Section~\ref{sec:modified_proof}.

\begin{figure}[t!]
\begin{center}
\begin{tabular}{ccccc|ccccc}
$q$ & $\ell$ & $\sigma$ & $\tau$ & $\upsilon$ & $q$ & $\ell$ & $\sigma$ & $\tau$ & $\upsilon$ \\[2pt] \hline
$2$\rule{0pt}{12pt} & $8$ & $\frac{43}{2}$ & $17$ & $41$  & $13$ & $3$  & $\frac{120}{13}$ & $6$ & $150$ \\
$3$\rule{0pt}{12pt} & $5$ & $16$ & $11$ & $55$  & $16$--$19$ & $2$ &  $\frac{9(q-1)}{q}$  & $5$ & $194$ \\
$4$\rule{0pt}{12pt}& $4$  & $\frac{225}{16}$ & $9$ & $65$ & $23$--$29$ & $2$  & $\frac{8(q-1)}{q}$ & $4$ & $285$ \\
$5$\rule{0pt}{12pt} & $4$ & $\frac{64}{5}$ & $8$ & $75$  & $31$--$43$ & $2$  & $\frac{15(q-1)}{2q}$ & $4$ & $381$ \\ 
$7$\rule{0pt}{12pt} & $3$ & $\frac{78}{7}$ & $7$ & $96$ & $47$---$59$ & $2$ & $\frac{7(q-1)}{q}$ & $4$ & $577$ \\ 
$8$\rule{0pt}{12pt} & $3$ & $\frac{21}{2}$ & $6$ & $108$ & $61$--$71$ & $2$ & $\frac{27(q-1)}{4q}$ & $4$ & $783$ \\
$9$\rule{0pt}{12pt} & $3$ & $\frac{88}{9}$ & $6$ & $124$  & $73$--$83$ & $2$  & $\frac{33(q-1)}{5q} n$ & $4$ & $996$  \\
$11$\rule{0pt}{12pt} & $3$ & $\frac{105}{11}$ & $6$ & $136$ & $\ge 89$ & $2$ & $\frac{13(q-1)}{2q}$ & $4$ & $1213$ 
\end{tabular}
\end{center}
\caption{Bounds Established in Section~\ref{sec:modified_proof}}
\label{fig:summary_table}
\end{figure}

A second result, which is also proved in Section~\ref{sec:modified_proof}, establishes that the constant~$\sigma$, mentioned above, can be
made arbitrarily close to~$6 \left(1 - \frac{1}{q}\right)$, provided that the minimum field size~$q$ and the
constant~$\upsilon$ are both increased --- at the cost of increasing the constants~$\ell$, $\tau$ and $\upsilon$
(but not~$\sigma$) that are listed.

\begin{theorem}
\label{thm:second_constants_for_conditioner}
Let $N$ be an integer such that $N \ge 18$. Let $\F_q$ be a finite field with size $q \ge 16N+9$ Let $m$ and~$n$
be integers such that $n \ge m \ge 0$. Let $\sigma = \left(1 - \frac{1}{q}\right)\ \cdot \left (6 + \frac{3}{N}\right)$,
$\tau = 1$, and $\upsilon = \lceil (2N+1) \ln (2N+1) + \frac{167}{5} (2N+1) \rceil$.

A random procedure exists for generating $n-m$ rows with length~$n$ such that if $A$ is an $m \times n$ matrix
of rank~$m$, then an additional $m$ rows (each with length~$n$) are produced, and the expected number of
nonzero entries in this row is $\sigma n \ln n + \tau n$ if $n - m \ge \upsilon$, and at most
$\frac{q-1}{q} (n - m) n \le \frac{q-1}{q}  \upsilon n$, otherwise.

If $q \le n^2$ then the matrix $B \in \matring{\F}{n}$ obtained from~$A$ by adding these rows is nonsingular
with probability at least $\frac{8}{9}$. If $q > n^2$ then this matrix is nonsingular with probability at least
\mbox{$\frac{8}{9} - \frac{8}{9n}$}.
\end{theorem}

Once again, the probability bounds here are quite arbitrary, and probability bounds that are closer to one can
be obtained by applications of the same techniques, at the cost of increasing the values of the constants~$\tau$
and~$\upsilon$.

This is work in progress. Future versions of this report will document progress in establishing that this
yields an efficient matrix preconditioner, to bound the number of nontrivial nilpotent blocks of a conditioned
matrix without lowering matrix rank, for matrices over small finite fields.
 

\section{A Modified Proof of Wiedemann's Result}
\label{sec:modified_proof}

This section describes modifications to Wiedemann's argument needed to establish
Theorems~\ref{thm:constants_for_conditioner} and~\ref{thm:second_constants_for_conditioner}.

\subsection{Getting Started --- and Improving Reliability by Adding Rows}
\label{ssec:getting_started}

Suppose that $m$ and~$n$ are positive integers such that
 $0 \le m \le n$ and $A \in \matgrp{\F}{m}{n}$ is a matrix with maximal rank~$m$, where $\F = \F_q$
is a finite field with size~$q$. Following Wiedemann's argument, let us begin by assuming that $q \le n^2$ and suppose,
as well, that
\begin{equation}
\label{eq:definition_of_k_and_c2}
n = m + k + c_2
\end{equation}
where $c_2$ will be defined later.

Suppose first that $k \le c_3 \ln n$, where $c_3$ is another constant to be chosen later,  and that the remaining
$n-m+\ell$ rows of an $(n + \ell) \times n$ matrix~$B$ are chosen uniformly and independently from~$\matgrp{\F}{1}{n}$.
The following lemma, which is easily proved, bounds the probability that the rank of~$B$ is less than~$n$ in this case.

\begin{lemma}
\label{lem:dense_selection_of_vectors}
Let $\ell$ be a nonnegative integer and let $B \in \matgrp{\F}{(n + \ell)}{n}$ be a matrix produced by appending another
$n-m+\ell$ rows, selected uniformly and independently from~$\matgrp{\F}{1}{n}$, to~$A$. Then the probability
that the rank of~$B$ is less than~$n$ is at most $q^{-\ell}$.

Furthermore, if $q \ge 3$ then the top $n \times n$ submatrix of~$B$ is nonsingular with probability at least
$\frac{1}{q-1}$.
\end{lemma}

\begin{proof}
Permuting the columns of~$A$ (and~$B$) as needed we may assume without loss of generality that the principal
$m \times m$ submatrix of~$A$ is nonsingular. Let us continue by choosing the entries in the leftmost $m$~columns
of the $n-m+\ell$ rows that are to be appended to~$A$. Regardless of the choice of these entries, this completes an
$(n+\ell) \times m$ submatrix of~$B$, including the leftmost $m$ columns, that must also have maximal rank~$m$.

The remaining entries of the top~$m$ rows of~$B$ are, of course, fixed: They are entries of~$A$. The entries in
the lower columns may now be chosen freely and (viewing the selection of these entries in column order, instead
of row order) a standard argument establishes that, for $1 \le i \le n-m$, if the leftmost $m+i-1$ columns of~$B$ are
linearly independent then the probability that the $m+i^{\text{th}}$ column of~$B$ is a linear combination of
these columns is at most $q^{m+i-n-\ell-1}$. It now follows that $B$ has rank less than~$n$ with probability
at most
\[
 \sum_{i=1}^{n-m} q^{m+i-n-\ell-1} < \sum_{j \ge 0} q^{-(\ell+1)-j} = \frac{q^{-(\ell+1)}}{1-q^{-1}} \le q^{-\ell}.
\]

Finally, if $q \ge 3$ then, setting $\ell = 0$, one can see the (top) $n \times n$ submatrix obtained by appending these
rows is singular with probability at most $\frac{q^{-1}}{1 - q^{-1}} = \frac{1}{q-1}$, as claimed.
\end{proof}

Since $n- m = k + c_2$, the expected number of nonzero entries in these rows is
\[
 \left( \frac{q-1}{q} \right) (k + c_2 + \ell)n < \left( \frac{q-1}{q} \right) \cdot \left( c_3 n \ln n + (c_2 + \ell) n \right)
\]
when $k \le c_3 \ln n$.

With that noted let us suppose, instead, that $k > c_3 \ln n$. Following Wiedemann once again, let
\begin{equation}
\label{eq:definition_of_z}
z = 1 - \frac{c_3 \ln n}{k} > 0.
\end{equation}

Suppose that, for the initial $k$~rows, each entry is set to zero with probability~$z$. The remaining unset entries
are then chosen uniformly and independently from~$\F$. The expected number of nonzero entries in these rows
is then
\[
 \left( \frac{q-1}{q} \right) c_3 n \ln n.
\]
If the entries of another $c_2 + \ell$ rows are chosen uniformly and independently from~$\F$ then the expected
number of nonzero entries in these rows is
\[
 \left( \frac{q-1}{q} \right) (c_2 + \ell) n.
\]
Consequently the expected number of nonzero entries in all these rows is less than
\[
 \left( \frac{q-1}{q} \right) \cdot \left( c_3 n \ln n + (c_2 + \ell)n \right)
\]
in this case as well.

The bulk of the rest of Wiedemann's argument concerns the derivation of an upper bound
for the probability that $B$ has rank less than~$n$ if $k > c_3 \ln n$ and the rows of~$B$ are chosen in this way.

Following Wiedemann, let $\rho$ be the probability that the rows of~$A$, together with the first $k$ (sparse)
rows generated using the above process, are linearly dependent --- that is, the probability that the space spanned
by these vectors has dimension less than $m+k$. Wiedemann shows that
\begin{equation}
\label{eq:bound_for_rho}
\rho \le \rho_0 + \rho_1
\end{equation}
where
\begin{equation}
\label{eq:definition_of_rho0}
\rho_0= \sum_{1 \le j \le k\beta_0} \binom{k}{j} (q-1)^j \left( q^{-1} + \frac{(q-1)}{q} (nq)^{-c_4 \beta} \right)^{n-m},
\end{equation}
and
\begin{equation}
\label{eq:definition_of_rho_1}
\rho_1 = \sum_{k \beta_0 < j \le k} \binom{k}{j} (q-1)^j \left( q^{-1} + \frac{(q-1)}{q} (nq)^{-c_4 \beta} \right)^{n-m},
\end{equation}
when
\begin{equation}
\label{eq:definition_of_c4}
c_4 = \frac{c_3}{3},
\end{equation}
and when
\begin{equation}
\label{eq:definition_of_beta}
\beta = \frac{j}{k}
\end{equation}
in the above expressions, and where $\beta_0$ is yet another constant to be defined later.

A useful bound for $\rho_1$ is next obtained: Assuming that
\begin{equation}
\label{eq:constraints_on_c2_and_c4}
c_2 \ge \frac{3}{2} \ge \log_2 e \qquad \text{and} \qquad c_4 \ge \frac{2}{\beta_0},
\end{equation}
Wiedemann establishes that
\begin{equation}
\label{eq:bound_for_rho1}
\rho_1 \le 2 q^{-c_2}.
\end{equation}

Wiedemann continues by observing that
\begin{equation}
\label{eq:first_bound_for_rho0}
\rho_0 \le \sum_{1 \le j < k \beta_0} \left( \frac{1}{\beta^{\beta} (1-\beta)^{1-\beta}} f(q) \right)^k,
\end{equation}
where
\begin{equation}
\label{eq:definition_of_f}
 f(x) = (x-1)^{\beta} \left( x^{-1} + (1-x^{-1}) (nx)^{-c_4 \beta} \right).
\end{equation}

\subsection{Getting to the Next Step by a Different Route}
\label{ssec:next_step}
 
Wiedemann continues by using the above to establish that
\begin{equation}
\label{eq:big_goal}
\rho_0 \le \sum_{1 \le j < k \beta_0} \beta^{k \beta}.
\end{equation}
Unfortunately, Wiedemann's  involves a Taylor series approximation that seems only to be accurate for a limited
range of values, and might suggest that either $\beta$ must be tiny or $c_4$ must be huge in order for it to be applicable.
The argument that follows is, therefore, quite different from given by Wiedemann~\cite{wied86}.

With that noted, consider the equation at lines~\eqref{eq:first_bound_for_rho0} and~\eqref{eq:definition_of_f}
once again. These imply that
\begin{equation}
\label{eq:what_is_known_1}
\begin{split}
 \rho_0 &\le \sum_{1 \le j < k \beta_0} \left(\frac{(q-1)^{\beta}}{\beta^{\beta} (1-\beta)^{1-\beta}}
   \left( \frac{1}{q} + \frac{q-1}{q} (nq)^{-c_4 \beta} \right) \right)^k \\
  &\le \sum_{1 \le j < k \beta_0}  \left(\frac{(q-1)^{\beta}}{\beta^{\beta} (1-\beta)^{1-\beta}}
   \left( \frac{1}{q} + \frac{q-1}{q} \left( \frac{\beta}{q} \right)^{c_4 \beta}  \right)\right)^k
\end{split}
\end{equation}
since $n^{-1} \le k^{-1} \le \frac{j}{k} = \beta$. It follows from this that
\begin{equation}
\label{eq:what_is_known_2}
 \rho_0 < \sum_{1 \le j < k \beta_0}  \left(\frac{1}{\beta^{\beta} (1-\beta)^{1-\beta}}
   \left( \frac{(q-1)^{\beta}}{q} + \frac{(q-1) q^{(1-c_4)\beta}}{q} \beta^{c_4 \beta} \right) \right)^k.
\end{equation}

\subsubsection{Bounding Terms in $\boldsymbol{\rho_0}$ When $\boldsymbol{\beta}$ is Extremely Small}
\label{ssec:beta_is_tiny}

Note that $\displaystyle{\lim_{\beta \rightarrow 0^{+}}} \beta^{\beta} = \displaystyle{\lim_{\beta
\rightarrow 0^{+}}} (1-\beta)^{1-\beta} = 1$, and $\beta^{\beta} < 1$ when $0 < \beta < 1$. The bounds
given at lines~\eqref{eq:what_is_known_1} and~\eqref{eq:what_is_known_2} can be simplified by
establishing a lemma like the following, allowing the factor $(1-\beta)^{-(1-\beta)}$ to be replaced
by a factor $\beta^{\gamma \beta}$, for a negative constant~$\gamma$, when $\beta$ is small.

\begin{lemma}
\label{lem:getting_rid_of_pesky_term}
Consider the relationship between $(1-x)^{-(1-x)}$ and $x^{\gamma x}$, for a negative constant~$\gamma$,
when $x$ is small and positive.
\begin{enumerate}
\renewcommand{\labelenumi}{\text{\textup{(\alph{enumi})}}}
\setlength{\itemsep}{0pt}
\item If $0 < x \le \frac{5}{43}$ then $(1-x)^{-(1-x)} \le x^{\gamma x}$, with $\gamma = -\frac{11}{25}$.
\item If $0 < x \le \frac{1}{5}$ then $(1-x)^{-(1-x)} \le x^{\gamma x}$, with $\gamma = -\frac{23}{40}$.
\end{enumerate}
\end{lemma}

\begin{proof} Consider the function 
\[
 g(x) = \gamma x \ln x + (1-x) \ln (1-x)
\]
when $\gamma$ is a negative
constant and $0 < x < 1$. Since $e^{g(x)} = \frac{x^{\gamma x}}{(1-x)^{-(1-x)}}$ and $(1-x)^{-(1-x)} > 0$,
$(1-x)^{-(1-x)} \le x^{\gamma x}$ (for $0 < x < 1$) if and only if $g(x) \ge 0$.

Now
\begin{align*}
\lim_{x \rightarrow 0^{+}} g(x)
 &= \lim_{x \rightarrow 0^{+}} \frac{\gamma \ln x}{x^{-1}} + \lim_{x \rightarrow 0^{+}} (1-x) \ln (1-x) \\
 &= \lim_{x \rightarrow 0^{+}} \frac{\gamma \ln x}{x^{-1}} \\
 &= \lim_{x \rightarrow 0^{+}} \frac{\gamma x^{-1}}{-x^{-2}} \tag{by l'H\^{o}pital's Rule} \\
 &= \lim_{x \rightarrow 0^{+}} -\gamma x = 0.
 \end{align*}
It is easily checked that $g'(x) = \gamma \ln x - \ln (1-x) + \gamma - 1$. Thus $\displaystyle{\lim_{x \rightarrow{0^{+}}}} g'(x)
= +\infty$, so that $g'(x) > 0$ when $x$ is small and positive. It follows from the above that $g(x) > 0$ when
$x$ is small and positive, as well.

Note next that $g''(x)= \frac{\gamma}{x} + \frac{1}{1-x}$, so that $\displaystyle{\lim_{x \rightarrow 0^{+}}} g''(x)
= -\infty$, since $\gamma < 0$, and $g''(x) < 0$ when $x$ is small and positive.
On the other hand, $g'''(x) = -\frac{\gamma}{x^2} + \frac{1}{(1-x)^2} > 0$ when $0 < x < 1$.

It now follows that $g(x)  \ge 0$ for $0 < x \le \delta$ if $g(\delta) \ge 0$ and $g'(\delta) < 0$: Since
$g'''(x) > 0$ for all $x$ such that $0 < x < 1$, this implies that either
\begin{enumerate}
\renewcommand{\labelenumi}{\text{\textup{\roman{enumi}.}}}
\setlength{\itemsep}{0pt}
\item $g(x) \ge 0$ for all $x$ such that $0 < x < 1$,
\item there exists a value $\Gamma$ such that $0 < \Gamma < 1$, $g'(x) \ge 0$ when
     $0 < x \le \Gamma$ and $g'(x) < 0$ when $\Gamma < x < 1$, or
\item there exist values~$\Gamma_1$ and~$\Gamma_2$ such that $0 < \Gamma_1 < \Gamma_2 < 1$,
     $g'(x) \ge 0$ when $0 < x \le \Gamma_1$,
     $g'(x) \le 0$ when $\Gamma_1 \le x \le \Gamma_2$, and $g'(x) > 0$ when $\Gamma_2 < x < 1$.
\end{enumerate}
The claim is certainly trivial in the first case. In the second case it necessarily follows that
$\Gamma < \delta$, while it follows that $\Gamma_1 < \delta < \Gamma_2$ in the third case. In each of the last two
cases, the function~$g$ is either nondecreasing or has a local maximum in the interval $0 < x \le \delta$. In either
case, it is minimized at one or the other of this interval's endpoints. Since $\displaystyle{\lim_{x \rightarrow 0^{+}}}
g(x) = 0$, it therefore suffices to confirm that $g(\delta) \ge 0$ in order to establish that $g(x) \ge 0$ when
$0 < x \le \delta$.

Part~(a) of the claim can now be established by choosing $\gamma = -\frac{11}{25}$ and $\delta = \frac{5}{43}$;
then $g'(\delta) = -\frac{11}{25} \ln \frac{5}{43} - \ln \frac{38}{43} - \frac{36}{25} < -0.3 < 0$ and
$g(\delta) = -\frac{11}{215} \ln \frac{5}{43} + \frac{38}{43} \ln \frac{38}{43} > 0.0008 > 0$, as required.

Part~(b) of the claim can be established by choosing $\gamma = -\frac{23}{40}$ and $\delta = \frac{1}{5}$; then
$g'(\delta) = \frac{23}{40} \ln 5 - \ln \frac{4}{5} - \frac{63}{40} < -0.4 < 0$ and
$g(\delta) = \frac{23}{200} \ln 5 + \frac{4}{5} \ln \frac{4}{5} > 0.006 > 0$.
\end{proof}

Suppose, now, that $\gamma \in \Q$ is a negative constant such that $(1 - \beta)^{-(1-\beta)} \le
\beta^{\gamma \beta}$
when $0 < \beta \le \Delta$ for a positive constant~$\Delta$; it follows by the above lemma that one might
choose $\gamma = -\frac{11}{25}$ if $\Delta = \frac{5}{43}$, and that one might choose $\gamma = -\frac{23}{40}$
if $\Delta = \frac{1}{5}$. It would follow from the inequality at
line~\eqref{eq:what_is_known_2} that
\begin{equation}
\label{eq:what_is_known_3}
 \rho_0 \le \sum_{1 \le j < k \beta_0} \left(\beta^{(\gamma-1)\beta} \left( \frac{(q-1)^{\beta}}{q}
   + \frac{(q-1)q^{(1-c_4)\beta}}{q} \beta^{c_4 \beta} \right)\right)^k
\end{equation}

This can be further simplified by bounding $(q-1)^{\beta}$ by $\beta^{\delta \beta}$ for a small
positive constant~$\delta$:

\begin{lemma}
\label{lem:getting_rid_of_power_of_q}
Consider the relationship between $\zeta^x$ and $x^{\delta x}$ when $\zeta$ is a positive constant.

If $\zeta = 2$ and  $0 < x \le \frac{1}{5}$ then $\zeta^x \le x^{\delta x}$ when $\delta = -\frac{9}{20}$.
\end{lemma}

\begin{proof}
Consider the function
\[
 h(x) = \delta x \ln x - x \ln \zeta
\]
when $\delta$ is a negative constant and $\zeta$ is a positive one.
Since $e^{h(x)} = \frac{x^{\delta x}}{\zeta^x}$ and $\zeta^x > 0$ when $x > 0$, $\zeta^x \le x^{\delta x}$
(for positive~$x$) if and only if $h(x) \ge 0$.

Now note that\\[-33pt]

\begin{align*}
\lim_{y \rightarrow 0^{+}} h(y)
  &= \lim_{y \rightarrow 0^{+}} \frac{\delta \ln y}{y^{-1}} - 0\\
  &= \lim_{y \rightarrow 0^{+}} \frac{\delta y^{-1}}{-y^{-2}} \tag{by l'H\^{o}pital's Rule} \\
  &= \lim_{y \rightarrow 0^{+}} -\delta y = 0.
 \end{align*}

It is easily checked that $h'(x) = \delta \ln x + \delta - \ln \zeta$. Now $\displaystyle{\lim_{y \rightarrow 0^{+}}} h'(y)= +\infty$,
since $\delta < 0$,
so that $h'(x) > 0$ for sufficiently small positive~$x$. It follows by the above that $h(x) > 0$ when $x$ is small and positive
as well.

Since $h''(x) = \frac{\delta}{x} < 0$ when $x  > 0$ it now follows, for any positive constant~$\rho$, that if
$h(\rho) \ge 0$ then $h(x) \ge 0$ as well when $0 < x \le \rho$ --- for it follows from that this that either
$h(x) \ge 0$ for all positive~$x$ --- in which case the claim is trivial ---
or there exists some constant $\Gamma$ such that $h(x) \ge 0$
when $0 < x \le \Gamma$, and such that $h(x) < 0$ when $x > \Gamma$. If $h(\frac{1}{5}) > \frac{1}{2}$,
then $\Gamma$ must be greater than~$\frac{1}{5}$ in this second case.

With that noted, the claim  can now be established by setting
$\zeta=2$,  $\delta = -\frac{9}{20}$, $\rho = \frac{1}{5}$,
and confirming that $h(\rho) = \frac{9}{100} \ln 5 - \frac{1}{5} \ln 2 > 0.006$.
\end{proof}

Suppose next that $\delta \in \Q$ is constant that is less than or equal to~$0$ such that $(q-1)^{\beta}
\le \beta^{\delta \beta}$  when $0 < x \le \Delta$; one can certainly choose $\delta = 0$ if $q = 2$,
and it follows by the above lemma that if $q=3$ then one can choose $\delta = -\frac{9}{20}$ when
$\Delta = \frac{1}{5}$. It is already necessary for the argument being developed that $c_4 \ge 1$,
so that (when $q \ge 2$) $q^{(1-c_4)\beta} \le 1$. It would therefore follow from the bound at
line~\eqref{eq:what_is_known_3} that
\begin{equation}
\label{eq:what_is_known_4}
 \rho_0 \le \sum_{1 \le j < k \beta_0} \left(\beta^{(\gamma -1)\beta} \left( \frac{\beta^{\delta \beta}}{q}
   + \frac{(q-1)}{q} \beta^{c_4 \beta}\right)\right)^k.
\end{equation}
Once again suppose that --- as in Wiedemann's original argument --- we wish to show that
\[
 \rho_0 \le \sum_{1 \le j < k \beta_0} \beta^{k\beta}.
\]
Then it follows from the above that it is sufficient to show that
\begin{equation}
\label{eq:what_is_needed_1}
 \beta^{(\gamma - 1) \beta} \left( \frac{\beta^{\delta \beta}}{q} + \frac{(q-1)}{q} \beta^{c_4 \beta} \right) 
 \le \beta^{\beta} \qquad \text{when $0 < \beta  = \frac{j}{k} \le \beta_0$,} 
\end{equation}
that is, that $f_1(\beta, q, c, \gamma, \delta) \ge 0$, when $0 < \beta \le \beta_0$ and $c = c_4$, for
\begin{equation}
\label{eq:first_general_approximation}
f_1(\beta, q, c, \gamma, \delta) = \beta^{\beta} - \beta^{\beta (\gamma-1)} \left( \frac{\beta^{\beta \delta}}{q}
  + \frac{(q-1)}{q} \beta^{\beta c} \right).
\end{equation}
Since ${\displaystyle{\lim_{\beta \rightarrow 0^{+}}}} \beta^{\beta} = 1$, 
\[
 \lim_{\beta \rightarrow 0^{+}} f_1(\beta, q, c, \gamma, \delta) = 1 - \left( {\textstyle{\frac{1}{q}}} + {\textstyle{\frac{q-1}{q}}}
 \right) = 0.
\]

\begin{lemma}
\label{lem:correctness_of_first_approximation}
Suppose that $\gamma = \frac{a}{d}$ and $\delta = \frac{b}{d}$ for non-positive integers~$a$ and~$b$ and
a positive integer~$d$. Suppose, as well, that
\begin{equation}
\label{eq:new_requirement_for_c}
c > \frac{2q}{q-1} - \frac{\delta}{q-1} - \frac{q \gamma}{q-1}.
\end{equation}
If $\Delta$ and~$\widehat{\Delta}$ are positive constants such that $0 < \Delta < \widehat{\Delta} \le \frac{1}{e}$,
$f_1(\Delta, q, c, \gamma, \delta) \ge 0$ and $f_1(\widehat{\Delta}, q, c, \gamma, \delta) < 0$, then
$f_1(\beta, q, c, \gamma, \delta) \ge 0$ for $0 < \beta \le \Delta$.
\end{lemma}

\begin{proof}
Let $z = \beta^{\beta}$, so that $f_1(\beta, q, c, \gamma, \delta) = p_0(z)$, where
\[
 p_0(z) = z - \frac{1}{q} z^{\gamma + \delta - 1} - \frac{q-1}{q} z^{c+\gamma - 1}.
\]
It follows that (differentiating with respect to~$z$)
\[
 p_0'(z) = 1 - \frac{1}{q}(\gamma + \delta - 1) z^{\gamma + \delta - 2}
   - \frac{(q-1)}{q}(c+\gamma -1) z^{c+ \gamma - 2},
\]
so that
\begin{align*}
 \lim_{z \rightarrow 1^{-}} p_0'(z)
  &= 1 - \frac{1}{q}(\gamma + \delta - 1) - \frac{(q-1)}{q} (c + \gamma - 1) \\
  &= 2 - \frac{\delta}{q} - \gamma - \frac{(q-1)}{q} c \\
  &= \frac{q-1}{q} \left( \frac{2q}{q-1} - \frac{\delta}{q-1} - \frac{q \gamma}{q-1} - c \right) \\
  &< 0,
\end{align*}
by the inequality at line~\eqref{eq:new_requirement_for_c}. Thus $p_0$ is a decreasing function as $z$ approaches~$1$
from below and, since ${\displaystyle{\lim_{z \leftarrow 1^{-}}}} p_0(z) = 0$, it follows that $p_0(z) > 0$ when $z$
is less than and sufficiently close to~$1$. Now, since $\beta^{\beta} < 1$ when $0 < \beta < 1$ and
${\displaystyle{\lim_{\beta \rightarrow 0^{+}}}} \beta^{\beta} = 1$, this implies that $f_1(\beta) > 0$ when $\beta$
is positive, and sufficiently small, as well.

Recall that $\gamma = \frac{a}{d}$ and $\delta = \frac{b}{d}$ where $a, b, d \in \Z$, $a \le 0$, $b \le 0$,
and $d > 0$. Thus $\beta^{(1-\gamma-\delta)\beta} f_1(\beta, q, c, \gamma, \delta)  = p_1(y)$,
where $y = \beta^{\beta/d}$ and
\[
 p_1(y) = -\frac{(q-1)}{q}  y^{cd - b} + y^{2d-a-b} -\frac{1}{q} \in \Q[y],
\]
so that (for $y$ as above) $f_1(\beta, q, c, \gamma, \delta) \ge 0$ if and only if $p_1(y) \ge 0$ --- and
(by the above) $p_1(y) > 0$
if $y$ is less than and sufficiently close to~$1$. Now --- regardless of the relationship between
$c$ and~$\gamma$ --- there are at most two changes in sign of the nonzero coefficients of this polynomial, when
listed by decreasing powers of~$y$. It follows that if $0 < \Delta < \widehat{\Delta} \le \frac{1}{e}$ (so that the function
$g(\beta) = \beta^{\beta/d}$ is decreasing, and injective, over the interval $0 < \beta \le \widehat{\Delta}$), 
$f_1(\Delta) \ge 0$ and $f_1(\widehat{\Delta}) < 0$, then $f_1(\beta) \ge 0$ for $0 < \beta \le \Delta$ ---
for, otherwise, the polynomial~$p_1$ would necessarily have at least four positive roots in the interval
$\widehat{\Delta}^{\widehat{\Delta}/d} \le y \le 1$, contradicting Descarte's rule of signs.
\end{proof}

It therefore suffices to check the condition at line~\eqref{eq:new_requirement_for_c} and to
confirm that $f_1(\Delta) \ge 0$ and $f_1(\widehat{\Delta}) < 0$, for $0 < \Delta < \widehat{\Delta} \le \frac{1}{e}$,
in order to establish that $f_1(\beta) \ge 0$ when $0 < \beta \le \Delta$ --- so that
\begin{equation}
\label{eq:first_target}
 \sum_{1 \le j \le k \Delta} \left( \frac{(q-1)^{\beta}}{\beta^{\beta} (1-\beta)^{1 - \beta}}
    \left( \frac{1}{q} + \frac{q-1}{q} (nq)^{-c_4 \beta} \right) \right)^k \le \sum_{1 \le j \le k \Delta} \beta^{k\beta}.
\end{equation}

\subsubsection{Bounding Terms in $\boldsymbol{\rho_0}$ When $\boldsymbol{\beta}$ is Larger}
\label{ssec:beta_is_moderate}

The process described in Subsection~\ref{ssec:beta_is_tiny} can only be used to establish the inequality
at line~\eqref{eq:first_target}, above,  for small positive constants~$\Delta$ that are generally much smaller than the
 desired bound $\beta_0 = \frac{2}{c_4}$. However, a complementary process --- which in turn, does not
 seem to be useful to establish the above inequalities when $\beta$ is extremely close to zero --- can (at least,
 sometimes) be used to establish that these inequalities hold for larger~$\Delta$ as well.
 
 Once again, recall that
 \[
  \frac{(q-1)^{\beta}}{\beta^{\beta} (1-\beta)^{1-\beta}} \left( \frac{1}{q} + \frac{q-1}{q}  (nq)^{-c_4 \beta} \right) \le
    \frac{(q-1)^{\beta}}{\beta^{\beta} (1-\beta)^{1-\beta}} \left( \frac{1}{q} +  \frac{q-1}{q} 
       \left( \frac{\beta}{q} \right)^{c_4 \beta} \right).
 \]
 
Suppose, now, that $\eta$ is a constant such that $0 < \eta < 1$.
 
 Then it certainly follows that
 \[
\frac{(q-1)^{\beta}}{\beta^{\beta} (1-\beta)^{1-\beta}} \left( \frac{1}{q} + \frac{q-1}{q} (nq)^{-c_r \beta}  \right) 
  \le \beta^{\beta}
 \]
 when $\Gamma_L \le \beta \le \Gamma_H$, for positive constants~$\Gamma_L$ and~$\Gamma_H$, if
 \begin{equation}
 \label{eq:first_new_relationship}
   \frac{(q-1)^{\beta}}{\beta^{\beta} (1-\beta)^{1-\beta}} \cdot \frac{1}{q} \le \eta \beta^{\beta} 
 \end{equation}
 and
 \begin{equation}
 \label{eq:second_new_relationship}
 \frac{(q-1)^{\beta}}{\beta^{\beta} (1-\beta)^{1-\beta}} \cdot 
   \frac{q-1}{q} \left(\frac{\beta}{q}\right)^{c_4 \beta} \le (1 - \eta) \beta^{\beta}
\end{equation}
when $\Gamma_L \le \beta \le \Gamma_H$ as well. 

Since $q > 1$ and $0 < \beta < 1$, the inequality at line~\eqref{eq:first_new_relationship} holds if and only if
\[
 \frac{\eta \beta^{2\beta} (1-\beta)^{1-\beta} q}{(q-1)^\beta} \ge 1.
\]
Considering logarithms, one can see that this is the case for $\Gamma_L \le \beta \le \Gamma_H$ if and only
if $F_1(\beta) \ge 0$ for $\Gamma_L \le \beta \le \Gamma_H$, where
\begin{equation}
\label{eq:defiinition_of_F1}
F_1(x) = 2x\ln x + (1-x) \ln (1-x) - x \ln (q-1) + \ln q + \ln \eta.
\end{equation}
Similarly, the inequality at line~\eqref{eq:second_new_relationship} is satisfied
if and only if
\[
 \frac{(1 - \eta) \beta^{(2 - c_4)\beta} (1-\beta)^{1-\beta} q^{c_4\beta + 1}}{(q-1)^{\beta+1}} \ge 1.
\]
Considering logarithms, one can see that this is the case for $\Gamma_L \le \beta \le \Gamma_H$
if and only if $F_2(\beta) \ge 0$ for $\Gamma_L \le \beta \le \Gamma_2$, where
\begin{equation}
\label{eq:definition_of_F2}
F_2(x) = (2-c_4)x\ln x + (1-x)\ln(1-x) + (c_4x+1) \ln q - (x+1)\ln(q-1) + \ln (1-\eta).
\end{equation}

Note next that
\[
 F_1'(x) = 1 + 2 \ln x - \ln (1-x) - \ln (q-1)
\]
--- which is independent of~$\eta$ --- and
\[
F_1''(x) = \frac{2}{x} + \frac{1}{1-x} 
\]
--- which is positive if $0 < x <  1$. Consequently it   it suffices to check that $F_1'(\beta_0) < 0$ in order to establish
that $F_1'(x) < 0$ for $0 < x \le \beta_0$, so that $F_1(x)$ is decreasing over the interval $\Delta < x \le \beta_0$. If
this is the case, and $\Delta \le \Gamma_L < \Gamma_H \le \beta_0$, then it suffices to check that $F_1(\Gamma_H)
\ge 0$ in order to confirm that $F_1(\beta) \ge 0$ for $\Gamma_L \le \beta \le \Gamma_H$ as well.

Note as well that
\[
 F_2'(x) = (1-c_4) + (2-c_4)\ln x - \ln (1-x) + c_4 \ln q - \ln (q-1)
\]
--- which is also independent of~$\eta$ --- and that
\[
F_2''(x) = \frac{2- c_4}{x} + \frac{1}{1-x}
\]
--- which is negative (for $c_4 > 2$) if $0 < x < \frac{2-c_4}{1-c_4} = 1 + \frac{1}{1-c_4}$, zero if $x = 1 + \frac{1}{1-c_4}$,
and positive if $x > 1 + \frac{1}{1-c_4}$. Consequently if $\beta_0 \le 1 + \frac{1}{1-c_4}$ then it suffices to check that
$F2'(\beta_0) \ge 0$ in order to establish that $F_2$ is increasing over the interval $\Delta \le x \le \beta_0$. If $\Delta
\le 1 + \frac{1}{1 - c_4} \le \beta_0$ then it suffices to check that $F_2'\left(1 + \frac{1}{1-c_4}\right) \ge 0$ in order to confirm
that $F_2$ is increasing over this interval. If $1 + \frac{1}{1 -c_4} \le \Delta$ then it suffices to check
that $F_2'(\Delta) \ge 0$ in
order to confirm this. In any case, if this has been confirmed and $\Delta \le \Gamma_L < \Gamma_H \le \beta_0$,
then it suffices to check that $F_2(\Gamma_L) \ge 0$ in order to confirm that $F_2(\beta) \ge 0$ for $\Gamma_L
\le \beta \le \Gamma_H$.

The desired inequality can now be established, for $\Delta \le \beta \le \beta_0$, by breaking this interval into one or
more subintervals, and using the above process with various choices of~$\eta$ to confirm that $F_1$ and~$F_2$
are both non-negative over each subinterval.

\subsubsection{Establishing That $f$ is a Decreasing Function}
\label{ssec:f_is_decreasing}

Once again, consider the function
\[
 f(x) = (x-1)^{\beta} (x^{-1} + (1 - x^{-1}) (nx)^{-c_4 \beta})
\]
as defined at line~\eqref{eq:definition_of_f}.
Wiedemann establishes that if $0 < \beta \le \beta_0 < {\textstyle{\frac{1}{4}}}$, $n \ge 1$, $q \ge 2$ and
$c_4 > {\textstyle{\frac{4}{\ln 2}}}$, then $f$ is a non-increasing function, so that $f(q) \le f(2)$ for $q \ge 2$ ---
as needed to establish that results like the above hold for larger finite fields as well.

Unfortunately, this argument requires both $c_4$ and~$c_3$ to assume larger values than are either necessary
or desirable. However, Wiedemann's argument can be modified in a straightforward way to establish the
following.

\begin{lemma}
\label{lem:f_is_decreasing}
Suppose that $n \ge 1$, $q \ge 16$, $0 < \beta \le \beta_0 \le \frac{12}{13}$, and
$c_4 \ge \frac{2}{\beta_0} \ge \frac{13}{6}$.
Then $f$ is a decreasing function of~$x$ over the interval $x \ge q$.
\end{lemma}

\begin{proof}
As Wiedemann notes, if $f$ is as defined at line~\eqref{eq:definition_of_f} then
\begin{multline*}
 f'(x) = \beta (x-1)^{\beta-1} (x^{-1} + (1-x^{-1}) (nx)^{-c_4 \beta}) \\ + (x-1)^{\beta}(-x^{-2} + x^{-2}(nx)^{-c_4\beta}
   -c_4 \beta(x^{-1}-x^{-2}) (nx)^{-c_4\beta})
\end{multline*}
so that 
\[
 x^2(x-1)^{-\beta}f'(x) = g(x) + h(x),
\]
where
\[
 g(x) = (\beta x - c_4 \beta x + c_4 \beta) (nx)^{-c_4 \beta}
\]
and
\[
 h(x) = \frac{\beta x}{x-1} - 1 + (nx)^{-c_4 \beta} = \beta + \beta(x-1)^{-1} - 1 + (nx)^{-c_4 \beta}.
\]
Consequently, for $x > 1$, if $g(x) < 0$ and $h(x) < 0$ then $f'(x) < 0$ as well.

Now, since $\beta (nx)^{-c_4 \beta} > 0$ when $n$, $x$, $c_4$ and~$\beta$ are all positive, it suffices to show
that $q - c_4 q + c_4 < 0$ in order to establish that $g(q) < 0$, and $q - c_4 q + c_4 < 0$ if and only if
$q > \frac{c_4}{c_4 - 1} = 1 + \frac{1}{c_4-1}$. Since $c_4 \ge \frac{13}{6} > 2$, $1 + \frac{1}{c_4 - 1} < 2$,
so that $g(q) < 0$ when $q \ge 16$, as desired.

Consider the function~$h$ when $n$, $q$, $\beta$, $\beta_0$ and~$c_4$ are as above.
This function is certainly decreasing with both~$x$ and~$c_4$. It therefore suffices to set $x = q = 16$ and
$c_4 = \frac{13}{6}$ and show that
\[
 H(\beta) = h(16) =  {\textstyle{\frac{16}{15}}} \beta - 1 + (16n)^{-\frac{13}{6} \beta} < 0
\]
when $0 < \beta \le \frac{12}{13}$ in order to establish that $h(x) < 0$, for $\beta$ in this range,
and for $q$ and~$n$ as above. Now it is easily checked that ${\displaystyle{\lim_{\beta \rightarrow 0^{+}}}} H(\beta)
= 0 - 1 + 1 = 0$. Considered as a function of~$\beta$ (and differentiating with respect to~$\beta$),
$H'(\beta) = {\textstyle{\frac{16}{15}}} - {\textstyle{\frac{13}{6}}} \ln (16n) \cdot (16n)^{-\frac{13}{6} \beta}$, so that
${\displaystyle{\lim_{\beta \rightarrow 0^{+}}}} H'(\beta) = \frac{16}{15} - \frac{13}{6} \ln(16n) \le \frac{16}{15} -
\frac{13}{6} \ln 16 < -4$. Thus both $H'(\beta)$ and~$H(\beta)$ are negative when $\beta$ is
positive and sufficiently small.

Note next that $H''(\beta) = \frac{169}{36} \ln(16n)^2 \cdot (16n)^{-\frac{13}{6} \beta} > 0$ whenever $\beta > 0$, so
that $H'(\beta)$ is a strictly increasing function of~$\beta$.   This admits (only) two possibilities: Either
$H(\beta) < 0$ for all $\beta > 0$ --- which certainly establishes the desired result --- or there exists a
positive value $\Delta$ such that $H(\beta) < 0$ when $0 < \beta < \Delta$, $H(\Delta) = 0$, and
$H(\beta) > 0$ when $\beta > \Delta$. In either case it now suffices to check that $H(\beta) < 0$ when
$\beta$ has the maximum value of interest, that is, when $\beta = \frac{12}{13}$. It therefore remains only
to note that
\[
  H\left(\frac{12}{13}\right) = \frac{64}{65} -1 + (16n)^{-2} \le -\frac{1}{65} + \frac{1}{256} < -0.01 < 0
\]
in order to complete the proof.
\end{proof}

\subsubsection{Application of These Processes}
\label{ssec:bounding_rho0}

The processes described in Subsections~\ref{ssec:beta_is_tiny} and~\ref{ssec:beta_is_moderate} and the
result established in Section~\ref{ssec:f_is_decreasing} can now be applied to establish the following.

\begin{lemma}
\label{lem:bound_when_q=2}
If $q = 2$, $0 < \beta \le \frac{6}{43}$ and $c_4 \ge \frac{43}{3}$ then
\[
 \frac{1}{\beta^{\beta} (1-\beta)^{1-\beta}} \left( {\textstyle{\frac{1}{2}}} + {\textstyle{\frac{1}{2}}}
 (2n)^{-c_4 \beta} \right) \le \beta^{\beta}.
\]
\end{lemma}

\begin{proof}
To begin, let us use the process described in Subsection~\ref{ssec:beta_is_tiny} to establish the above inequality
when $0 < \beta \le \frac{5}{43}$. It follows by part~(a) of Lemma~\ref{lem:getting_rid_of_pesky_term} that
$(1-x)^{-(1-x)} \le x^{\gamma x}$ when $0 < x \le \frac{5}{43}$ and $\gamma = -\frac{11}{25}$, so that $\gamma$ can
be set to have this value when this process is applied. Since $(q-1)^x = 1^x = 1$ when $0 < x \le \frac{5}{43}$, 
$(q-1)^x \le x^{\delta x}$ in this range when $\delta = 0$, so this value will be used for this constant. In this case
\[
 c_4 - \left( \frac{2q}{q-1} - \frac{\delta}{q-1} - \frac{q\gamma}{q-1} \right) = \frac{89}{15} > 0,
\]
so that the condition at line~\eqref{eq:new_requirement_for_c} is satisfied. Since $0 < \frac{5}{43} < \frac{7}{43}
\le \frac{1}{e}$, it now suffices to note that $f_1(\frac{5}{43}, 2, \frac{43}{3}, -\frac{11}{25}, 0) > 0.04$ and
$f_1(\frac{7}{43}, 2, \frac{43}{3}, -\frac{11}{25}, 0) < - 0.03$ --- for it then follows by
Lemma~\ref{lem:correctness_of_first_approximation} that the inequality in the claim is satisfied when
$0 < \beta \le \frac{5}{43}$.

The process described in Subsection~\ref{ssec:beta_is_moderate} can now be used to establish the above
inequality when $\frac{5}{43} \le \beta \le \frac{6}{43}$, completing the proof. Since $F_1'\left(\frac{6}{43}\right) <
-2.7$ the function~$F_1$ is decreasing over this interval, for every choice of~$\eta$. Since $\frac{6}{43} <
\frac{37}{40} = 1 + \frac{1}{1 - (43/3)}$ and $F_2'\left(\frac{6}{43}\right) > 21$, the function $F_2$ is increasing
over this interval for every choice of~$\eta$.

It now suffices to confirm that if $\eta = \frac{99}{100}$ then $F_1\left(\frac{6}{43}\right) > 0.004$ and
$F_2\left(\frac{5}{43}\right) > 0.2$, so that $F_1$ and $F_2$ are both non-negative over the interval
$\frac{5}{43} \le \beta \le \frac{6}{43}$, as desired.
\end{proof}

It now follows that
\begin{equation}
\label{eq:fourth_bound_for_rho0_when_q=2}
\rho_0 \le \sum_{1 \le j < k \beta_0} \beta^{k\beta}
\quad \text{if $q = 2$, $\beta_0 = \frac{6}{43}$, and $c_4 \ge \frac{43}{3}$.}
\end{equation}

In the above lemma the upper limit, $\beta_0 = \frac{6}{43}$ for~$\beta$, has been chosen so that
$\beta_0 = \frac{2}{c_4}$ --- in order to match the constraint between~$\beta_0$ and~$c_4$ as shown
at line~\eqref{eq:constraints_on_c2_and_c4}. A plot of the function
$\beta^{\beta} - \frac{1}{\beta^{\beta} (1-\beta)^{1-\beta}} \left(\frac{1}{2} + \frac{1}{2} \left(\frac{\beta}{2}\right)^{\frac{43}{3}\beta}\right)$,
for $0 < \beta \le \frac{6}{43}$, is shown in Figure~\ref{fig:plot1}.
\begin{figure}[t!]
\begin{center}
\scalebox{0.4}{\includegraphics{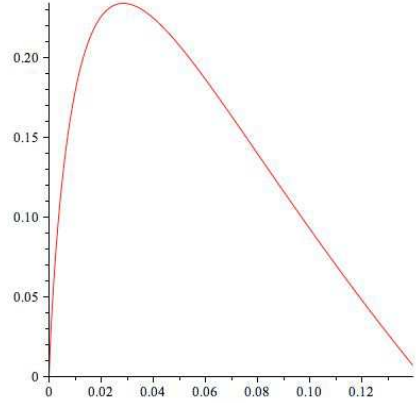}}
\end{center}
\caption{Plot of $\beta^{\beta} - \frac{1}{\beta^{\beta} (1-\beta)^{1-\beta}} \left( \frac{1}{2} - \frac{1}{2} \left(
 \frac{\beta}{2}\right)^{\frac{43}{3}\beta}\right)$
when $0 < \beta \le \frac{6}{43}$}
\label{fig:plot1}
\end{figure}
As this may suggest, the above result result can be improved slightly --- but not by very much: The inequality
at line~\eqref{eq:fourth_bound_for_rho0_when_q=2}, above, is not satisfied if $c_4$ is decreased to~$14$ and
$\beta_0$ increased to~$\frac{1}{7}$.

In order to see one more example of this process let us consider
 the case that $q \ge 3$. An application of the technique described above establishes the following.

\begin{lemma}
\label{lem:bound_when_q=3}
If $q = 3$, $0 < \beta \le \frac{1}{4}$ and $c_4 \ge 8$ then
\[
 \frac{2^{\beta}}{\beta^{\beta}(1-\beta)^{1-\beta}} \left( {\textstyle{\frac{1}{3}}} + {\textstyle{\frac{2}{3}}}
 (3n)^{-c_4\beta}\right) \le \beta^{\beta}.
\]
\end{lemma}

\begin{proof}
To begin, let us use the process described in Subsection~\ref{ssec:beta_is_tiny} to establish the above inequality
when $0 < \beta \le \frac{1}{5}$. It follows by part~(b) of Lemma~\ref{lem:getting_rid_of_pesky_term} that
$(1-x)^{-(1-x)} \le x^{\gamma x}$ when $0 < x \le \frac{1}{5}$ and $\gamma = -\frac{23}{40}$, so that $\gamma$ can
be set to have this value when this process is applied. It follows by Lemma~\ref{lem:getting_rid_of_power_of_q}
that if $q=3$ and $\delta = -\frac{9}{20}$ then  $(q-1)^x \le x^{\delta x}$ when $0 < x \le \frac{1}{5}$, so $\delta$ can
be set to be $-\frac{9}{20}$ in this argument. In this case
\[
 c_4 - \left( \frac{2q}{q-1} - \frac{\delta}{q-1} - \frac{q\gamma}{q-1} \right) = \frac{313}{80} > 0,
\]
so that the condition at line~\eqref{eq:new_requirement_for_c} is satisfied. Since $0 < \frac{1}{5} < \frac{1}{4}
\le \frac{1}{e}$, it now suffices to note that $f_1(\frac{1}{5}) > 0.0008$ and
$f_1(\frac{1}{4}) < - 0.03$ --- for it then follows by
Lemma~\ref{lem:correctness_of_first_approximation} that the inequality in the claim is satisfied when
$0 < \beta \le \frac{1}{5}$.

The process described in Subsection~\ref{ssec:beta_is_moderate} can now be used to establish the above
inequality when $\frac{1}{5} \le \beta \le \frac{1}{4}$, completing the proof. Since $F_1'\left(\frac{1}{4}\right) <
-2.1$ the function~$F_1$ is decreasing over this interval, for every choice of~$\eta$. Since $\frac{1}{4} <
\frac{6}{7} = 1 + \frac{1}{1 - 8}$ and $F_2'\left(\frac{1}{4}\right) > 9$, the function $F_2$ is increasing
over this interval for every choice of~$\eta$.

It now suffices to confirm that if $\eta = \frac{39}{40}$ then $F_1\left(\frac{6}{25}\right) > 0.01$ and
$F_2\left(\frac{1}{5}\right) > 0.08$, so that $F_1$ and $F_2$ are both non-negative over the interval
$\frac{1}{5} \le \beta \le \frac{6}{25}$.

It then suffices to confirm that if $\eta = \frac{123}{125}$ then $F_1\left(\frac{1}{4}\right) > 0.0002$ and
$F_2\left(\frac{6}{25}\right) > 0.05$, so that $F_1$ and $F_2$ are both non-negative over the interval
$\frac{6}{25} \le x \le \frac{1}{4}$, as needed to complete the proof.
\end{proof}

It now follows that
\begin{equation}
\label{final_bound_for_rho0_when_q=3}
\rho_0 \le \sum_{1 \le j < k \beta_0} \beta^{k\beta} \quad \text{if $q = 3$, $\beta_0 = \frac{1}{4}$, and
  $c_4 \ge 4$.}
\end{equation}
A plot of the function $\beta^{\beta} - \frac{2^{\beta}}{\beta^{\beta} (1-\beta)^{1-\beta}} \left( \frac{1}{3} +
\frac{2}{3} \left(\frac{\beta}{3}\right)^{8 \beta}\right)$, when $0 < \beta \le \frac{1}{4}$, is shown in
Figure~\ref{fig:plot2}.
\begin{figure}[t!]
\begin{center}
\scalebox{0.4}{\includegraphics{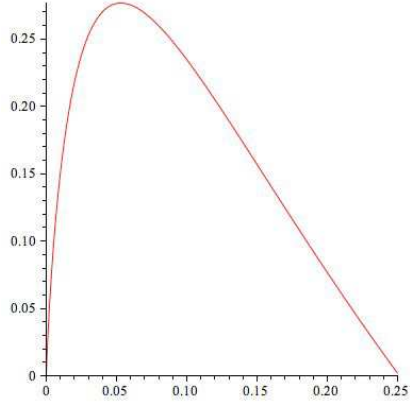}}
\end{center}
\caption{Plot of $\beta^{\beta} - \frac{1}{\beta^{\beta} (1-\beta)^{1-\beta}} \left(\frac{1}{3} + \frac{2}{3}
\left(\frac{\beta}{3}\right)^{8\beta}\right)$ when $0 < \beta \le \frac{1}{4}$}
\label{fig:plot2}
\end{figure}
Once again, this suggests that the above result cannot be improved by very much.

Appendix~\ref{sec:continuation_of_argument} include details of analyses for additional field sizes
as well --- as summarized in Figure~\ref{fig:choices_of_c4} on page~\pageref{fig:choices_of_c4}.
A Maple worksheet, that can be used to check
these details, is available online at
\begin{center}
\url{http://www.cpsc.ucalgary.ca/~eberly/Research/sparse_conditioner.mw}.
\end{center}
\begin{figure}[t!]
\begin{center}
\begin{tabular}{ccc|ccc|ccc}
$q$ & $c_4$ & $\beta_0$ & $q$ & $c_4$ & $\beta_0$  & $q$ & $c_4$ & $\beta_0$  \\[2pt] \hline
$2$ \rule{0pt}{12pt} & $\frac{43}{3}$ & $\frac{6}{43}$ & $9$ & $\frac{11}{3}$ & $\frac{6}{11}$ & $47$--$59$ & 
$\frac{7}{3}$ & $\frac{6}{7}$ \\
$3$ \rule{0pt}{12pt} & $8$& $\frac{1}{4}$ & $11$ & $\frac{7}{2}$ & $\frac{4}{7}$ & $61$--$71$ &
$\frac{9}{4}$ & $\frac{8}{9}$ \\
$4$ \rule{0pt}{12pt} & $\frac{25}{4}$ & $\frac{8}{25}$ & $13$ & $\frac{10}{3}$ & $\frac{3}{5}$ & $73$--$83$ &
$\frac{11}{5}$ & $\frac{10}{11}$ \\
$5$ \rule{0pt}{12pt} & $\frac{16}{3}$ & $\frac{3}{8}$ & $16$--$19$ & $3$ & $\frac{2}{3}$ & $\ge 89$ &
$\frac{13}{6}$ & $\frac{12}{13}$ \\
$7$ \rule{0pt}{12pt} & $\frac{13}{3}$ & $\frac{6}{13}$ & $23$--$29$ & $\frac{8}{3}$ &$\frac{3}{4}$ & \\
$8$ \rule{0pt}{12pt} & $4$ & $\frac{1}{2}$ & $31$--$43$ & $\frac{5}{2}$ & $\frac{4}{5}$ & & 
\end{tabular}
\end{center}
\caption{Choices of~$c_4$ and~$\beta_0$ for Various Field Sizes}
\label{fig:choices_of_c4}
\end{figure}

 It follows from this that
\[
 \rho_0 \le \sum_{1 \le j < k \beta_0} \beta^{k\beta}
\]
for each of the choices of~$q$, $\beta_0$ and~$c_4$ given in Figure~\ref{fig:choices_of_c4}.

\subsection{Asymptotic Results: Choice of Field Size}
\label{ssec:asymptotics_1}

The objective of this next subsection is to identify bounds on the sizes of primes allowing the inequality at line~\eqref{eq:big_goal}
to be established when $c_4$ is closer to~$2$. Suppose, in particular, that $N$ is an integer such that $N \ge 18$ and that
\begin{equation}
\label{eq:asymptotic_constraints_on_constants}
c_4 = 2 + {\textstyle{\frac{1}{N}}} \qquad \text{and} \qquad q \ge 16N + 9.
\end{equation}
Consider now the function
\begin{equation}
\label{eq:definition_of_g}
\begin{split}
  g(\beta)  &= \frac{\left( \frac{(q-1)^{\beta}}{\beta^{\beta} (1-\beta)^{1-\beta}} \right)
    \left( \frac{1}{q} + \left(\frac{q-1}{q}\right) \left( \frac{\beta}{q} \right)^{c\beta} \right)}{\beta^{\beta}} \\
   &= \frac{(q-1)^{\beta}}{\beta^{2\beta} (1-\beta)^{1-\beta}} \left( \frac{1}{q} + \left(\frac{q-1}{q}\right)  
     \left( \frac{\beta}{q} \right)^{c_4 \beta}  \right)
\end{split}
\end{equation}
noting --- by the inequality at line~\eqref{eq:what_is_known_1} --- that the inequality at line~\eqref{eq:big_goal} is satisfied if
$g(\beta) \le 1$ when $0 < \beta \le \beta_0 = \frac{2}{c_4}$.

Consider, as well the function
\begin{equation}
\label{eq:definition_of_h}
\begin{split}
h(\beta) &= \ln g(\beta) \\ &= \beta  \ln (q-1) - 2 \beta \ln \beta - (1 - \beta) \ln (1-\beta) \\
  &\hspace*{1 true in} + \ln \left( \frac{1}{q} + \left( \frac{q-1}{q} \right) \left(\frac{\beta}{q} \right)^{c_4 \beta}   \right)
\end{split}
\end{equation}
observing that
\begin{equation}
\label{eq:definition_of_hprime}
 h'(\beta) = \ln (q-1) - 2 \ln \beta + \ln (1-\beta) - 1 + \frac{H(\beta)}{H(\beta)+1} (c_4 + c_4 \ln \beta - c_4 \ln q)
\end{equation}
where
\begin{equation}
\label{eq:definition_of_H}
H(\beta) = (q-1) \left( \frac{\beta}{q}\right)^{c_4 \beta}.
\end{equation}

\begin{lemma}
\label{lem:first_asymptotic_interval}
If $0 < \beta \le \frac{1}{50 c_4 \ln q}$ then $g(\beta) \le 1$.
\end{lemma}

\begin{proof} It is easily checked that $\displaystyle{\lim_{\beta \rightarrow 0^+} g(\beta)} = 1$ and $\displaystyle{\lim_{\beta \rightarrow 0^+}
h(\beta)} = 0$. The claim can therefore be established by showing that $h'(\beta) < 0$ when $0 < \beta \le \frac{1}{50 c_4 \ln q}$.

Consider the above function $H(\beta)$, recalling as well that
\begin{equation}
\label{eq:log_approximation}
  \frac{x}{1+x} \le \ln (1+x) \le x,
\end{equation}
for any real number~$x$ such that $x > -1$. Replacing $x$ with $-x$ (where $x < 1$), one has that
\[
 - \frac{x}{1-x} \le \ln (1-x) \le -x
\]
as well, and replacing $x$ with $1-x$ (where, once again $x < 1$) one has that
\[
 1 - \frac{1}{x} \le \ln x \le x-1.
\]
It follows from this that (for $0 < x < 1$)
\[
 x - 1 \le x \ln x \le x^2 - x,
\]
so that
\[
 e^{x-1} \le x^x \le e^{x^2 - x} \le 1
\]
when $0 < x \le 1$. It follows from the definition of $H(\beta)$ at line~\eqref{eq:definition_of_H} that
\begin{align*}
H(\beta) &\ge (q-1) \frac{e^{(\beta-1) c_4}}{q^{c_4 \beta}} \\
 &\ge (q-1) \frac{e^{-c_4}}{q^{c_4 \beta}} \\
 &= (1 - q^{-1}) e^{-c_4} q^{1 - c_4 \beta} \\
 &\ge (1 - q^{-1}) e^{-\frac{37}{18}} q^{1- c_4 \beta} \tag{since $N \ge 18$, so that $c_4 \le \frac{37}{18}$} \\
 &\ge \frac{296}{297} \frac{e^{-\frac{37}{18}}}{q^{c_4 \beta}} q \tag{since $q = 16N + 9 \ge 297$}\\
 &\ge \frac{296}{297} \frac{e^{-\frac{37}{18}}}{q^{\frac{1}{50 \ln q}}} q & \tag{since $\beta \le \frac{1}{50 c_4 \ln q}$} \\
 &= \frac{296}{297} e^{-\frac{37}{18} - \frac{1}{50}} q\\
 &> \frac{q}{8} \\
 &\ge 2N+1 \tag*{(since $q \ge 16N + 9$).} \\
\end{align*}

It follows that
\[
\frac{H(\beta)}{H(\beta)+1} = 1 - (H(\beta)+1)^{-1} \ge 1 - (2N+1)^{-1} = \frac{2N}{2N+1}.
\]

Note, as well that $c_4 + c_4 \ln \beta - c_4 \ln q = c_4 \left( 1 + \ln \textstyle{\frac{\beta}{q}} \right) < c_4 \left( 1 + \ln \textstyle{\frac{1}{q}} \right)
 < 0$.
It now follows by the equation at line~\eqref{eq:definition_of_hprime} that
\begin{align*}
h'(\beta) &= \ln (q-1) - 2 \ln \beta + \ln (1-\beta) -1 + \frac{H(\beta)}{H(\beta) +1} (c_4 + c_4 \ln \beta - c_4 \ln q) \\
  &\le \ln (q-1) - 2 \ln \beta + \ln (1-\beta) - 1 + \textstyle{\frac{2N}{2N+1}} (c_4 + c_4 \ln \beta - c_4 \ln q) \\
  &= \ln (q-1) - 2 \ln \beta + \ln (1-\beta) - 1 + 2(1+ \ln \beta - \ln q) \tag{since $c_4 = 2 + \frac{1}{N} = \frac{2N+1}{N}$} \\
  &< - \ln q + \ln (1 - \beta) + 1 < 0,
\end{align*}
as required.
\end{proof}

A different approach is required for larger values of~$\beta$ because $h'(\beta)$ is eventually positive. Note that
$g(\beta) = g_1(\beta) + g_2(\beta)$, where
\begin{equation}
\label{eq:definition_of_g1}
 g_1(\beta) = \frac{(q-1)^{\beta}}{\beta^{2\beta} (1-\beta)^{1-\beta}} \cdot \frac{1}{q}
   = q^{-1} (q-1)^{\beta} \cdot \frac{(1-\beta)^{\beta-1}}{\beta^{2\beta}}
\end{equation}
and
\begin{equation}
\label{eq:definition_of_g2}
g_2(\beta) = \frac{(q-1)^{\beta}}{\beta^{2\beta} (1-\beta)^{1-\beta}} \cdot \left( \frac{q-1}{q} \right) \left( \frac{\beta}{q} \right)^{c_4 \beta}
  =\left( 1-\frac{1}{q}\right) \frac{\beta^{\beta/N}}{ \left( \frac{q^{c_4-1}}{1-\beta} \right)^{\beta} (1- \beta)}.
\end{equation}

\begin{lemma}
\label{lem:second_asymptotic_interval}
If $\frac{1}{50 c_4 \ln q} \le \beta \le \frac{7}{8}$ then $g(\beta) \le 1$.
\end{lemma}

\begin{proof} Consider first the function
\[
 k(x) = \frac{(1-x)^{x-1}}{x^{2x}}
\]
and
\[
 \ell(x) = \ln k(x) = (x-1) \ln (1-x) - 2x \ln x,
\]
noting that
\[
 \ell'(x) = \ln (1-x) - 2 \ln x - 1
\]
and
\[
 \ell''(x) = - \frac{1}{1-x} - \frac{2}{x},
\]
so that $\ell''(x) < 0$ when $0 < x < 1$. Note that $\ell'(x) = 0$ when
\[
 \frac{1-x}{x^2} = e,
\]
that is, when $e x^2 + x - 1 = 0$. Applying the quadratic equation, one can see that the functions $k(x)$ and~$\ell(x)$ are both
increasing when $0 < x < \frac{-1 + \sqrt{1+4e}}{2e}$, and decreasing when $\frac{1 + \sqrt{1+4e}}{2e} < x < 1$. In particular,
$k(x) \le k\left(\frac{-1 + \sqrt{1 + 4e}}{2e}\right) < 3$ for $0 < x < 1$.

This be used to obtain upper bounds for the function~$g_1$, shown at line~\eqref{eq:definition_of_g1}, above, over various intervals.
Suppose, in particular, that $k(\beta) \le \delta$ when $\Delta_L < \beta \le \Delta_H$ for constants $\delta$, $\Delta_L$ and~$\Delta_H$
such that $0 < \Delta_L < \Delta_H < 1$. It follows that if $\Delta_L < \beta \le \Delta_H$ then
\[
 g_1(\beta) \le \delta q^{-1} (q-1)^{\beta} \le {\textstyle{\frac{\delta}{297}}} 296^{\beta} \quad \text{when $\beta \le \textstyle{\frac{7}{8}}$}
\]
since $q \ge 297$; this upper bound for $g_2(\beta)$ is increasing with $\beta$.

Note, as well, that, since $\frac{x}{1+x} \le \ln (1+x) \le x$ for any real number~$x$ such that $x > -1$, $-\frac{x}{1-x} \le \ln (1-x) \le -x$
for any real number~$x$ such that $0 < x < 1$. Consequently $-x \le (1-x) \ln (1-x) \le -x + x^2$, so that
\[
 e^{-x} \le (1-x)^{1-x} \le e^{-x+x^2},
\]
and (replacing $x$ with $1-x$, and applying the bounds for $(1-x)^{1-x}$)
\[
 e^{1-x} \le x^x \le e^{-x+x^2}
\]
when $0 < x < 1$. Thus, if $g_2(\beta)$ is as given at line~\eqref{eq:definition_of_g2}, above, then
\begin{align*}
\label{eq:bound_for_g2}
 g_2(\beta) &= \left( 1 - \frac{1}{q} \right) \frac{\beta^{\beta/N}}{\left( \frac{q^{c_4-1}}{1-\beta}\right)^{\beta} (1-\beta)} \\
  &\le \left( 1 - \frac{1}{q} \right) \frac{1}{q^{(c_4 -1)\beta} (1-\beta)^{1-\beta}} \tag{since $\beta^{\beta} \le e^{-\beta + \beta^2} \le 1$} \\
  &\le \left( 1 - \frac{1}{q} \right) \left( \frac{e}{q^{c_4-1}} \right)^{\beta} & \tag{since $(1-\beta)^{1-\beta} \ge e^{-\beta}$} \\
  &< \left( 1 - \frac{1}{q} \right) \left( \frac{e}{q} \right)^{\beta} \tag{since $c_4 > 2$} \\
  &\le {\textstyle{\frac{296}{297}}} \left( {\textstyle{\frac{e}{297}}} \right)^{\beta} \tag*{(since $q \ge 297$).}
\end{align*}
Since $q > e$, this upper bound for $g_2(\beta)$ is certainly decreasing as $\beta$ increases.

It follows by Lemma~\ref{lem:first_asymptotic_interval} that $g(\beta) = g_1(\beta) + g_2(\beta) < 1$ when $\beta = \frac{1}{50 c_4 \ln q}$.
Now
\begin{align*}
 g_2 \left( {\textstyle{\frac{1}{50 c_4 \ln q}}} \right)
  &\le \left( {\textstyle{1 - \frac{1}{q}}} \right)  \left( {\textstyle{\frac{e}{q}}} \right)^{\frac{1}{50 c_4 \ln q}} \\
  &\le \left( {\textstyle{1 - \frac{1}{q}}} \right) \left( e^{\frac{1}{50 c_4}} \right)^{\frac{1}{\ln q} - 1} \\
  &\le \left( e^{\frac{1}{50 c_4}} \right)^{\frac{1}{\ln q}-1}. \\
\end{align*}
Since $e^{\frac{1}{50 c_4}} > 1$, $q \ge 297$, and the above exponent $\frac{1}{\ln q}- 1$ decreases as $q$ increases, it now follows
that
\[
 g_2\left({\textstyle{\frac{1}{50 c_4 \ln q}}} \right) \le \left( e^{\frac{1}{50 c_4}} \right)^{\frac{1}{\ln 297} - 1}.
\]
Now, since $N \ge 18$, $c_4 \le \frac{37}{18}$ and $\frac{1}{50 c_4} \ge \frac{9}{925}$. Since the exponent in the above expression
is negative, it now follows that
\[
   g_2\left({\textstyle{\frac{1}{50 c_4 \ln q}}} \right) \le \left( e^{\frac{9}{925}} \right)^{\frac{1}{\ln 297} - 1} <  \frac{397}{400}.
\]
Since $g_2$ is decreasing with $\beta$ and $g_1$ is increasing with~$\beta$, it now suffices to choose a value~$\gamma$
such that $\frac{1}{50 c_4 \ln q} < \gamma \le \frac{7}{8}$, and $g_1(\gamma) \le \frac{3}{400}$ in order to conclude that
$g(\beta) = g_1(\beta) + g_2(\beta) \le 1$ when $\frac{1}{50 c_4 \ln q} \le \beta \le \gamma$.

Suppose now that $\gamma \le \frac{1}{8}$ as well; then $k(\gamma) \le k\left(\frac{1}{8}\right) \le 2$ so that $g_1(\beta)
\le \frac{2}{297} 296^{\beta}$. It follows from that that if $\gamma = \frac{1}{75}$ then  $\frac{1}{50 c_4 \ln q} < \frac{1}{75} < \frac{1}{8}$
and $g_1(\gamma) < \frac{3}{400}$, as required
to conclude that $g(\beta) < 1$ when $\frac{1}{50 c_4 \ln q} \le \beta \le \frac{1}{75}$.

Note next that $g_2\left(\frac{1}{75}\right) \le \frac{296}{297} \left( \frac{e}{297} \right)^{\frac{1}{75}} < \frac{19}{20}$. Since $g_2$ is
decreasing with~$\beta$ and $g_1$ is increasing with $\beta$, it suffices to choose a value~$\widehat{\gamma}$ such that
$\frac{1}{75} < \widehat{\gamma} \le \frac{7}{8}$ and $g(\widehat{\gamma}) < \frac{1}{20}$ in order to conclude that $g(\beta) \le 1$
when $\frac{1}{75} \le \beta \le \widehat{\gamma}$ as well.

As noted above, $k(\beta) < 3$ when $0 < \beta < 1$, and it follows that $g_1(\beta) \le \frac{3}{297} 296^{\beta}$ for all such~$\beta$.
It follows from this that if $\widehat{\gamma} = \frac{1}{4}$ then $g_1(\gamma) < \frac{1}{20}$, as needed to conclude that $g(\beta) \le 1$
when $\frac{1}{75} \le \beta \le \frac{1}{4}$.

Now note that $g_2(\left(\frac{1}{4}\right) \le \frac{296}{97} \left( \frac{e}{297} \right)^{\frac{1}{4}} < \frac{1}{3}$. Since $g_2$ is decreasing
with~$\beta$ and $g_1$ is increasing with~$\beta$ it suffices to choose a value $\widetilde{\gamma}$ such that $\frac{1}{4} <
\widetilde{\gamma} \le \frac{7}{8}$ and $g(1)(\widetilde{\gamma}) \le \frac{2}{3}$ in order to conclude that $g(\beta)\le1$ when
$\frac{1}{4} \le \beta \le \widetilde{\gamma}$.

Once again, $g_1(\widetilde{\gamma}) \le \frac{3}{297} 296^{\widetilde{\gamma}}$, and this suffices to set $\widetilde{\gamma} = \frac{2}{3}$
in order to ensure that $g_1(\widetilde{\gamma}) < \frac{2}{3}$, as needed.

Now $g_2(\left(\widetilde{\beta}\right) = g_2\left(\frac{2}{3}\right) < \frac{1}{20}$, so it suffices to choose $\overline{\gamma}$ such that
$\frac{2}{3} < \overline{\gamma} \le \frac{7}{8}$ and $g_1(\overline{\gamma}) \le\frac{19}{20}$ in order to ensure that $g(\beta) \le 1$ when
$\frac{2}{3} \le \beta \overline{\gamma}$. Now, $k(x) \le \frac{5}{2}$ when $\frac{2}{3} \le x \le \frac{7}{8}$, so that $g_1(x) \le \frac{5}{2} \cdot
\frac{1}{297} \cdot 296^x$ for all $x$ in this range, and this can be used to establish that one can set $\overline{\beta} = \frac{4}{5}$ in order
to ensure that the desired conditions are met.

Finally, $g_2\left(\frac{4}{5}\right) < \frac{23}{1000}$ and $k(x) < \frac{99}{50}$ when $\frac{4}{5} \le x \le \frac{7}{8}$. This can be used to
establish that $g_1(\left(\frac{7}{8}\right) < \frac{977}{1000}$, as needed to establish that $g(\beta) \le 1$ when $\frac{4}{5} \le \beta
\le \frac{7}{8}$ and complete the proof of the claim.
\end{proof}

Once again the functions~$h(\beta)$, $h'(\beta)$ and~$H(\beta)$, shown at lines~\eqref{eq:definition_of_h}--\eqref{eq:definition_of_H},
are of use to prove the desired result for larger values of~$\beta$.

\begin{lemma}
\label{lem:third_asymptotic_interval}
If $\frac{7}{8} \le \beta \le \frac{2N}{2N+1}$ then $g(\beta) \le 1$.
\end{lemma}

\begin{proof}
Consider the functions~$h(\beta)$, $h'(\beta)$ and~$H(\beta)$.
Note first that if $\frac{7}{8} \le \beta \le \frac{2N}{2N+1}$ then
\begin{align*}
H(\beta)&= (q-1) \left({\textstyle{\frac{\beta}{q}}}\right)^{c_4 \beta} \\
 &\le (q-1) \left({\textstyle{\frac{1}{q}}}\right)^{c_4 \beta}\tag{since $\beta < 1$, $q >0$,and $c_4 \beta > 1$} \\
 &\le q^{1 - c_4 \beta} \\
 &\le q^{-\frac{3}{4}} \tag*{(since $c_4 \ge 2$ and $\beta \ge \frac{7}{8}$, so that $1 - c_4 \beta \le - \frac{3}{4}$).} 
\end{align*}
Since $H(\beta) \ge 0$ as well, it follows that
\[
\frac{H(\beta)}{H(\beta)+1} \le H(\beta) \le q^{-\frac{3}{4}}
\]
as well. Since $c_4 = 2 + \frac{1}{N} \le 2 + \frac{1}{18} = \frac{37}{18}$, 
\[
 c_4 \frac{H(\beta)}{H(\beta)+1} \le \frac{37}{18} \cdot q^{-\frac{3}{4}}.
\]

Since $1 + \ln \beta - \ln q < 0$ it now follows that
\begin{align*}
h'(\beta) &= \ln (q-1) - 2 \ln \beta + \ln (1-\beta) - 1 + \frac{H(\beta)}{H(\beta)+1} (c_4 + c_4 \ln \beta - c_4 \ln q) \\
  &\ge \ln (q-1) - 2 \ln \beta + \ln (1-\beta) - 1 + \frac{37}{18} \cdot q^{-\frac{3}{4}} (1 + \ln \beta + \ln q) \\
  &\ge \ln (16N+8) - 2 \ln \beta + \ln \left(\frac{1}{2N+1}\right) - 1 + \frac{37}{18} \cdot q^{-\frac{3}{4}}
    (1 + \ln \beta + \ln q) \tag{since $q-1 \ge 16N+8$ and $1 - \beta \ge \frac{1}{2N+1}$} \\
  &= \ln 8 - 2 \ln \beta - 1 + \frac{37}{18} \cdot q^{-\frac{3}{4}}(1 + \ln \beta + \ln q) \\
 &= (\ln 8 - 1 + {\textstyle{\frac{37}{18}}} \cdot q^{-\frac{3}{4}} (1 + \ln q))
   - (2 - {\textstyle{\frac{37}{18}}} \cdot q^{-\frac{3}{4}}) \ln \beta \\
  &\ge {\textstyle{\frac{9}{10}}}  - (2 - {\textstyle{\frac{37}{18}}} \cdot q^{-\frac{3}{4}}) \ln \beta
   \tag{since $\ln 8 - 1 + {\textstyle{\frac{37}{18}}} \cdot q^{-\frac{3}{4}} (1 + \ln q)
     \ge \ln 8 - 1 + {\textstyle{\frac{37}{18}}} \cdot 267^{-\frac{3}{4}} (1 + \ln 267) \le {\textstyle{\frac{9}{10}}}$} \\
  &\ge {\textstyle{\frac{9}{10}}} \tag{since $\ln \beta < 0$ and $2 - {\textstyle{\frac{37}{18}}} \cdot q^{-\frac{3}{4}}
    \ge 2 - {\textstyle{\frac{37}{18}}} \cdot 267^{-\frac{3}{4}} \ge {\textstyle{\frac{19}{10}}} > 0$} \\
  &> 0.
\end{align*}
Thus the function $h(\beta)$ is increasing over the interval $\frac{7}{8} \le \beta \le \frac{2N}{2N+1}$. Since
$h(\beta) = \ln g(\beta)$, the function $g(\beta)$ is increasing as well --- and it suffices to confirm that
$g\left(\frac{2N}{2N+1}\right) \le 1$ in order to establish the claim.

Now recall that $g\left(\frac{2N}{2N+1}\right) = g_1\left(\frac{2N}{2N+1}\right) + g_2\left(\frac{2N}{2N+1}\right)$,
for the functions~$g_1(\beta)$ and~$g_2(\beta)$ as defined at lines~\eqref{eq:definition_of_g1}
and~\eqref{eq:definition_of_g2} respectively. Applying these definitions one can see that
\begin{align*}
 g_1\left({\textstyle{\frac{2N}{2N+1}}}\right)
  &= q^{-1} (q-1)^{\frac{2N}{2N+1}} \frac{\left({\textstyle{\frac{1}{2N+1}}}\right)^{-\frac{1}{2N+1}}}
  {\left({\textstyle{\frac{2N}{2N+1}}}\right)^{\frac{4N}{2N+1}}} \\
  &= \left( 1 - {\textstyle{\frac{1}{q}}} \right)
   \frac{(q-1)^{-\frac{1}{2N+1}} \cdot \left( {\textstyle{\frac{1}{2N+1}}} \right)^{-\frac{1}{2N+1}}}
   {\left( {\textstyle{\frac{2N}{2N+1}}} \right)^{\frac{4N}{2N+1}}} \\
  &\le \frac{16N+8}{16N+9} \cdot
   \frac{\left((16N+8) \cdot \left( {\textstyle{\frac{1}{2N+1}}}\right)\right)^{-\frac{1}{2N+1}}}
    {\left( {\textstyle{\frac{2N}{2N+1}}} \right)^{\frac{4N}{2N+1}}} \tag{since $q \ge 16N+9$} \\
  &= \frac{16N+8}{16N+9} \cdot
   \frac{8^{-\frac{1}{2N+1}}}
    {\left( {\textstyle{\frac{2N}{2N+1}}} \right)^{\frac{4N}{2N+1}}} \\
  &\le \frac{16N+8}{16N+9} \cdot
   \frac{8^{-\frac{1}{2N+1}}}
    {e^{-\frac{2}{2N+1}}} \tag{since ${\textstyle{\left( \frac{2N}{2N+1} \right)}}^{\frac{4N}{2N+1}} \ge
      e^{-\frac{2}{2N+1}}$} \\
   &= \frac{16N+8}{16N+9} \cdot
    \left( \frac{e^2}{8} \right)^{\frac{1}{2N+1}} \\
   &\le \frac{16N+8}{16N+9} \tag{since $\textstyle{\frac{e^2}{8}} < 1$}
\end{align*}
and
\begin{align*}
 g_2\left({\textstyle{\frac{2N}{2N+1}}}\right)
  &= \left(1 - {\textstyle{\frac{1}{q}}}\right)
   \frac{\left( {\textstyle{\frac{2N}{2N+1}}} \right)^{\frac{2}{2N+1}}  }
    { \left(  \frac{q^{\frac{N+1}{N}}}{\frac{1}{2N+1}}     \right)^{\frac{2N}{2N+1}}  \left( {\textstyle{\frac{1}{2N+1}}} \right)   } \\
   &= \left(1 - {\textstyle{\frac{1}{q}}}\right)
   \frac{\left( {\textstyle{\frac{2N}{2N+1}}} \right)^{\frac{2}{2N+1}}  }
   {  \left( q^{\frac{N+1}{N}} \right)^{\frac{2N}{2N+1}} \left( {\textstyle{\frac{1}{2N+1}}} \right)^{\frac{1}{2N+1}}    } \\
    &= \left(1 - {\textstyle{\frac{1}{q}}}\right)
     \frac{\left( {\textstyle{\frac{(2N)^2}{2N+1}}} \right)^{\frac{1}{2N+1}} q^{-\frac{1}{2N+1}}   }
     {q} \\
    &\le  \left(1 - {\textstyle{\frac{1}{q}}}\right)
    \frac{\left( {\textstyle{\frac{(2N)^2}{(2N+1)^2}}} \right)^{\frac{1}{2N+1}} }{q} \tag{since $q \ge 16N+9 \ge 2N+1$} \\
    &\le \left(1 - {\textstyle{\frac{1}{q}}}\right) \cdot {\textstyle{\frac{1}{q}}} \\
    &\le \frac{1}{q} \\
    &\le \frac{1}{16N+9} \tag{since $q \ge 16N+9$}
 \end{align*}
as needed to establish that $g\left(\frac{2N}{2N+1}\right) \le 1$ and complete the proof.
\end{proof}

The following is now a straightforward consequence of
Lemmas~\ref{lem:first_asymptotic_interval}--\ref{lem:third_asymptotic_interval}.

\begin{corollary}
\label{cor:asymptotic_primes}
If $N \ge 18$, $c_4$ and~$q$ are as shown at line~\eqref{eq:asymptotic_constraints_on_constants},
and $\beta_0 = \frac{2}{c_4}  = \frac{2N}{2N+1}$, the the inequality at line~\eqref{eq:big_goal} is satisfied.
\end{corollary}

\subsection{Splitting the Sum to Get a Better Bound for $\boldsymbol{\rho_0}$}
\label{ssec:splitting_sum}

Let $\epsilon$ be a positive constant.
Suppose now that $\Delta$ is a positive integer whose depends on~$\epsilon$ and the field size~$q$. It follows from the
above (for appropriate choices of~$q$, $\beta_0$ and~$c_4$) that
\begin{align*}
\rho_0 
  &\le \sum_{1 \le j < k \beta_0} \beta^{k \beta}  \\
  &= \sum_{1 \le j < k \beta_0} \left(\frac{j}{k}\right)^j & \tag{since $\beta = \frac{j}{k}$} \\
  &= \zeta+ \theta,
\end{align*}
where
\begin{equation}
\label{eq:defn_of_zeta}
 \zeta = \sum_{1 \le j \le \Delta} \left( \frac{j}{k} \right)^j
\end{equation}
and
\begin{equation}
\label{eq:defn_of_theta}
\theta = \sum_{\Delta < j < k \beta_0} \left( \frac{j}{k} \right)^j.
\end{equation}

It follows from the above that
\begin{align*}
\theta &\le \sum_{\Delta < j \le k \beta_0} \left( \frac{j}{k} \right) ^j\\
 &\le \sum_{\Delta < j \le k \beta_0} \beta_0^j \\
 &\le \beta_0^{\Delta + 1} \sum_{j \ge 0} \beta_0^j \\
 &= \frac{\beta_0^{\Delta+1}}{1 - \beta_0} \\
 &\le {\textstyle{\frac{1}{10}}} \epsilon
\end{align*}
 provided that
 \[
  \Delta \ge \frac{\ln (\epsilon^{-1}) + \ln 10 - \ln (1 - \beta_0)}{\ln (1/\beta_0)} - 1.
 \]
 Choices of~$\Delta$ that satisfy this inequality for various field sizes (and corresponding choices of~$\beta_0$)
 are shown in Figure~\ref{fig:choices_of_Delta} on page~\pageref{fig:choices_of_Delta}.
 \begin{figure}[t!]
 \begin{center}
 \begin{tabular}{ccc|ccc|ccc}
 $q$ & $\beta_0$ & $\Delta$ & $q$ & $\beta_0$ & $\Delta$ & $q$ & $\beta_0$ & $\Delta$ \\[2pt] \hline
 $2$ \rule{0pt}{12pt} & $\frac{6}{43}$ & $\lceil \frac{51}{100} \ln (\epsilon^{-1}) + \frac{1}{4} \rceil$ 
  & $9$ & $\frac{6}{11}$ & $\lceil \frac{33}{20} \ln (\epsilon^{-1}) + \frac{41}{10} \rceil$
  & $47$--$59$ & $\frac{6}{7}$ & $\lceil \frac{13}{2} \ln (\epsilon^{-1}) + \frac{133}{5} \rceil$ \\
 $3$ \rule{0pt}{12pt} & $\frac{1}{4}$ & $\lceil \frac{73}{100} \ln (\epsilon^{-1}) + \frac{9}{10} \rceil$
  & $11$ & $\frac{4}{7}$ & $\lceil \frac{179}{100} \ln (\epsilon^{-1}) + \frac{47}{10} \rceil$
  & $61$--$71$ & $\frac{8}{9}$ & $\lceil \frac{17}{2} \ln (\epsilon^{-1}) + \frac{149}{4} \rceil$ \\
 $4$ \rule{0pt}{12pt} & $\frac{8}{25}$ & $\lceil \frac{22}{25} \ln (\epsilon^{-1}) + \frac{7}{5} \rceil$
  & $13$ & $\frac{3}{5}$ & $\lceil \frac{49}{25} \ln (\epsilon^{-1}) + \frac{107}{20} \rceil$
  & $73$--$83$ & $\frac{10}{11}$ & $\lceil \frac{21}{2} \ln (\epsilon^{-1}) + \frac{242}{5} \rceil$ \\
 $5$ \rule{0pt}{12pt} & $\frac{3}{8}$ & $\lceil \frac{51}{50} \ln (\epsilon^{-1}) + \frac{19}{10} \rceil$
  & $16$--$19$ & $\frac{2}{3}$ & $\lceil \frac{99}{40} \ln (\epsilon^{-1}) + \frac{37}{5} \rceil$
  & $\ge 89$ & $\frac{12}{13}$ & $\lceil \frac{25}{2} \ln (\epsilon^{-1}) + \frac{599}{10} \rceil$ \\
 $7$ \rule{0pt}{12pt} & $\frac{6}{13} $ & $\lceil \frac{13}{10} \ln (\epsilon^{-1}) + \frac{14}{5} \rceil$
  & $23$--$29$ & $\frac{3}{4}$ & $\lceil \frac{87}{25} \ln (\epsilon^{-1}) + \frac{119}{10} \rceil$ & \\
 $8$ \rule{0pt}{12pt} & $\frac{1}{2}$ & $\lceil \frac{29}{20} \ln (\epsilon^{-1}) + \frac{17}{5} \rceil$
  & $31$--$43$ & $\frac{4}{5}$ & $\lceil \frac{9}{2} \ln (\epsilon^{-1}) + \frac{50}{3} \rceil$ & \\
 \end{tabular}
 \end{center}
 \caption{Choices of $\Delta$ for Various Field Sizes}
 \label{fig:choices_of_Delta}
 \end{figure}
 When $N \ge 18$ and $\beta_0 = \frac{2N}{2N+1}$ (as in Subsection~\ref{ssec:asymptotics_1})
 \begin{align*}
 \frac{\ln (\epsilon^{-1}) + \ln 10 - \ln(1 - \beta_0)}{\ln(1/\beta_0} - 1
  &= \frac{\ln (\epsilon^{-1} + \ln 10 - \ln \left( {\textstyle{\frac{1}{2N+1}}}\right)}{\ln \left(1 + {\textstyle{\frac{1}{2N}}}\right)} - 1 \\
  &\le \frac{\ln (\epsilon^{-1}) + \ln 10 + \ln (2N+1)}{{\textstyle{\frac{1}{2N+1}}}}  - 1
     \tag{since $\ln \left(1 + \frac{1}{2N}\right) \ge \frac{1}{2N+1}$} \\
  &= (2N+1) \ln (\epsilon^{-1}) + (2N+1) \ln (2N+1) + (2N+1) \ln 10 - 1.
 \end{align*}
 It therefore suffices to ensure that
 \[
  \Delta \ge \lceil (2N+1) \ln (\epsilon^{-1}) + (2N+1) \ln (2N+1) + (2N+1) \ln 10 - 1 \rceil
 \]
 in this case.

 It also follows from the above that
 \begin{align*}
 \zeta &\le \sum_{1 \le j \le \Delta} \left( \frac{j}{k} \right)^j \\
  &\le \sum_{1 \le j \le \Delta} \left( \frac{\Delta}{k} \right)^j \\
  &< \sum_{j \ge 1} \left( \frac{\Delta}{k} \right)^j \\
  &= \frac{\Delta/k}{1-\Delta/k} \\
  &\le \textstyle{\frac{4}{5}} \epsilon
 \end{align*}
 provided that $k \ge \lceil (\frac{5}{4} \epsilon^{-1} + 1) \Delta \rceil$, and this is the case if
 $k \ge \lceil (\frac{5}{4} \epsilon^{-1} + 1) (\widehat{\Delta} + 1) \rceil$, where $\Delta = \lceil \widehat{\Delta} \rceil$.
 Suitable choices of~$k$ for small field sizes are as shown in Figure~\ref{fig:choices_of_k} on page~\pageref{fig:choices_of_k}.
 \begin{figure}[t!]
 \begin{center}
 \begin{tabular}{c|c|c}
$q$  & Lower Bound for $k$ & Bound when $\epsilon = \frac{1}{10}$ \\[2pt] \hline
$2$ \rule{0pt}{12pt} &
 $\lceil \frac{51}{80} \epsilon^{-1} \ln (\epsilon^{-1}) + \frac{25}{16} \epsilon^{-1} + \frac{51}{100} \ln (\epsilon^{-1}) + \frac{5}{4} \rceil$  &
 $33$ \\
$3$ \rule{0pt}{12pt} &
 $\lceil \frac{73}{80} \epsilon^{-1} \ln (\epsilon^{-1}) + \frac{19}{8} \epsilon^{-1} + \frac{73}{100} \ln (\epsilon^{-1}) + \frac{19}{10} \rceil$ &
 $49$ \\
$4$ \rule{0pt}{12pt} &
$\lceil \frac{11}{10} \epsilon^{-1} \ln (\epsilon^{-1}) + 3 \epsilon^{-1} + \frac{22}{25} \ln (\epsilon^{-1}) + \frac{12}{5} \rceil$ &
$60$ \\
$5$ \rule{0pt}{12pt} &
$\lceil \frac{51}{40} \epsilon^{-1} \ln (\epsilon^{-1}) + \frac{29}{8} \epsilon^{-1} + \frac{51}{50} \ln (\epsilon^{-1}) + \frac{29}{10} \rceil$ &
$71$ \\
$7$ \rule{0pt}{12pt} &
$\lceil \frac{13}{8} \epsilon^{-1} \ln (\epsilon^{-1}) + \frac{19}{4} \epsilon^{-1} + \frac{13}{10} \ln (\epsilon^{-1}) + \frac{19}{5} \rceil$ &
$92$ \\
$8$ \rule{0pt}{12pt} &
$\lceil \frac{29}{16} \epsilon^{-1} \ln (\epsilon^{-1}) + \frac{11}{2} \epsilon^{-1} + \frac{29}{20} \ln (\epsilon^{-1}) + \frac{22}{5} \rceil$ &
$105$ \\
$9$ \rule{0pt}{12pt} &
$\lceil \frac{33}{16} \epsilon^{-1} \ln (\epsilon^{-1}) + \frac{51}{8} \epsilon^{-1} + \frac{33}{20} \ln (\epsilon^{-1}) + \frac{51}{10} \rceil$ &
$121$ \\
$11$ \rule{0pt}{12pt} &
$\lceil \frac{179}{80} \epsilon^{-1} \ln (\epsilon^{-1}) + \frac{57}{8} \epsilon^{-1} + \frac{179}{100} \ln (\epsilon^{-1}) + \frac{57}{10} \rceil$ &
$133$ \\
$13$ \rule{0pt}{12pt} &
$\lceil \frac{49}{20} \epsilon^{-1} \ln (\epsilon^{-1}) + \frac{127}{16} \epsilon^{-1} + \frac{49}{25} \ln (\epsilon^{-1}) + \frac{127}{20} \rceil$&
$147$  \\
 $16$--$19$ \rule{0pt}{12pt} &
 $\lceil \frac{99}{32} \epsilon^{-1} \ln (\epsilon^{-1}) + \frac{21}{2} \epsilon^{-1} + \frac{99}{40} \ln (\epsilon^{-1}) + \frac{42}{5} \rceil$&
 $191$  \\
 $23$--$29$ \rule{0pt}{12pt} &
 $\lceil \frac{87}{20} \epsilon^{-1} \ln (\epsilon^{-1}) + \frac{129}{8} \epsilon^{-1} + \frac{87}{25} \ln (\epsilon^{-1}) + \frac{129}{10} \rceil$ &
 $283$ \\
 $31$--$43$ \rule{0pt}{12pt} &
 $\lceil \frac{45}{8} \epsilon^{-1} \ln (\epsilon^{-1}) + \frac{265}{12} \epsilon^{-1} + \frac{9}{2} \ln (\epsilon^{-1}) + \frac{53}{3} \rceil$ &
 $379$ \\
 $47$--$59$ \rule{0pt}{12pt} &
 $\lceil \frac{65}{8} \epsilon^{-1} \ln (\epsilon^{-1}) + \frac{69}{2} \epsilon^{-1} + \frac{13}{2} \ln (\epsilon^{-1}) + \frac{138}{5} \rceil$ &
 $575$ \\
 $61$--$71$ \rule{0pt}{12pt} & 
 $\lceil \frac{85}{8} \epsilon^{-1} \ln (\epsilon^{-1}) + \frac{765}{16} \epsilon^{-1} + \frac{17}{2} \ln (\epsilon^{-1}) + \frac{153}{4} \rceil $ &
 $781$  \\
 $73$--$83$ \rule{0pt}{12pt} &
 $\lceil \frac{105}{8} \epsilon^{-1} \ln (\epsilon^{-1}) + \frac{247}{4} \epsilon^{-1} + \frac{21}{2} \ln (\epsilon^{-1}) + \frac{247}{5} \rceil$ &
 $994$ \\
 $\ge 89$ \rule{0pt}{12pt} &
 $\lceil \frac{125}{8} \epsilon^{-1} \ln (\epsilon^{-1}) + \frac{609}{8} \epsilon^{-1} + \frac{25}{2} \ln (\epsilon^{-1}) + \frac{609}{10} \rceil$&
 $1211$
 \end{tabular}
 \end{center}
 \caption{Choices of~$k$ for Various Field Sizes}
 \label{fig:choices_of_k}
 \end{figure}
 
 When $N \ge18$, $c_4 = 2 + \frac{1}{N}$, and $\beta_0 =\frac{2}{c_4} = \frac{2N}{2N+1}$, it suffices that
 \begin{multline}
 \label{eq:asymptotic_bound_for_k}
  k \ge \left\lceil \left({\textstyle{\frac{5}{4}}} \epsilon^{-1} + 1\right)   ((2N+1) \ln (\epsilon^{-1}) + (2N+1) \ln (2N+1)
   + (2N+1) \ln 10  \right\rceil \\ \in \Theta(\epsilon^{-1} N(\ln (\epsilon^{-1}) + \ln N)).
 \end{multline}
 In particular, when $\epsilon = \frac{1}{10}$, it suffices to ensure that
 \begin{equation}
 \label{eq:asymptotic_bound_for_k_with_specific_failure_probability}
 k \ge \left\lceil (2N+1) \ln (2N+1) = {\textstyle{\frac{167}{5}}} \ln (2N+1) \right\rceil.
\end{equation}
 
 It now follows that $\rho_0 \le \frac{1}{10} \epsilon + \frac{4}{5} \epsilon = \frac{9}{10} \epsilon$ provided that this constraint on~$k$ can
 be satisfied.
 
 \subsection{Completion of the Analysis for the Case $\boldsymbol{q \le n^2}$}
 
 Suppose next that one wishes to ensure that $\rho_1 \le \frac{1}{20}\epsilon$, so that $\rho \le \frac{19}{20} \epsilon$.
 If the constraint on~$k$, described above, can be satisfied, then it follows
 by the inequality at line~\eqref{eq:bound_for_rho1} that it suffices to choose~$c_2$ such that $2q^{-c_2} \le 
 \frac{\epsilon}{20}$, that is,
 \[
  c_2 \ge \log_q (40 \epsilon^{-1}) = \frac{\ln (\epsilon^{-1}) + \ln 40}{\ln q}.
 \]
 It remains to choose $c_2 + \ell$ rows of the $(n + \ell) \times n$ matrix~$B$ uniformly and independently
 from $\matgrp{\F}{1}{n}$. Suppose that that the set of rows of the original matrix~$A$, and the $k$ ``sparse''
 rows selected as described above, are linearly independent. Then it follows by 
 Lemma~\ref{lem:dense_selection_of_vectors} that (once the remaining rows of~$B$ are chosen uniformly
 and independently) the rank of~$B$ is less than~$n$ with probability at most~$q^{-\ell}$. Furthermore, if
 $q \ge 3$ then the top $n$~rows of~$B$ are linearly independent --- so that we can set $\ell = 0$ --- with
 probability at least $1 - \frac{1}{q-1}$. Thus, if we wish to ensure
 that this probability is at most $\frac{1}{20}\epsilon$ then it suffices to ensure that either
 \[
  \ell \ge \log_q (20 \epsilon^{-1}) = \textstyle{\frac{\ln (\epsilon^{-1}) + \ln 20}{\ln q}} \quad
   \text{or} \quad q \ge 20 \epsilon^{-1} + 1 \text{ (and $\ell = 0$).}
 \]
 It now follows that if $c_4$, $c_2$ and~$\ell$ have all been chosen as described above then the probability
 that $B$ has rank less than~$n$ is at most
 \[
  \rho + \frac{\epsilon}{20} (1-\rho) \le \rho + \frac{\epsilon}{20} \le \frac{19 \epsilon}{20} + \frac{\epsilon}{20} = \epsilon.
 \]
 Choices of~$c_2$ and~$\ell$ satisfying the above constraints, along with the constraints at
 line~\eqref{eq:constraints_on_c2_and_c4}, are shown in Figure~\ref{fig:choices_of_c2_and_ell}, for the case that
 $\epsilon = \frac{1}{10}$.
 \begin{figure}[t!]
 \begin{center}
 \begin{tabular}{ccc|ccc|ccc}
 $q$ & $c_2$ & $\ell$ & $q$ & $c_2$ & $\ell$ & $q$ & $c_2$ & $\ell$ \\ \hline
 $2$ \rule{0pt}{12pt} & $9$ & $8$ & $5$ & $4$ & $4$ &  $16$--$19$ & $3$ & $2$   \\
 $3$ \rule{0pt}{12pt} & $6$ & $5$ & $7$ & $4$ & $3$ &   $23$--$89$ & $2$ & $2$   \\
 $4$ \rule{0pt}{12pt} & $5$ & $4$ &  $8$--$13$ & $3$ & $3$ &  &
 \end{tabular}
 \end{center}
 \caption{Choices of~$c_2$ and~$\ell$ for Various Field Sizes when $\epsilon = \frac{1}{10}$}
 \label{fig:choices_of_c2_and_ell}
 \end{figure}
 
 It now suffices to set $\sigma = c_3 \left( \frac{q-1}{q} \right)= 3 c_4 \left( \frac{q-1}{q} \right)$, for $c_4$ as given
 in Figure~\ref{fig:choices_of_c4} and
 $\tau = c_2 + \ell$, for $c_2$ and~$\ell$ as given in Figure~\ref{fig:choices_of_c2_and_ell}, in order to establish
 the claim in Theorem~\ref{thm:constants_for_conditioner} when $k = n - m - c_2$ is greater than or equal to the lower
 bound given in Subsection~\ref{ssec:splitting_sum} and $q \le n^2$, for $c_2$ as given above. Setting $\upsilon$ to be
 the sum of~$\ell$ and the lower bound for~$k$, described in Subsection~\ref{ssec:splitting_sum},
 suffices to establish these claims when $n - m - c_2$ is less than the lower
 bound for~$k$ and $q \le n^2$ as well,  for $c_2$ as above.
 The values shown in Figure~\ref{fig:summary_table} have been obtained using these equations.

 Similarly, Theorem~\ref{thm:second_constants_for_conditioner} can be obtained by setting $c_4 = 2 + \frac{1}{N}$
 for $N \ge 18$, $c_3 = 3 c_4 = 6 + \frac{3}{N}$, $\sigma =  c_3 (1 - \frac{1}{q})$,
 $c_2 = \left\lceil \frac{\ln (40 \epsilon^{-1})}{\ln q} \right\rceil$, $\ell = \left\lceil \frac{\ln 20 \epsilon^{-1}}{ln q} \right\rceil$
 if $\ell \le 20 \epsilon^{-1} + 1$, setting $\ell = 0$ otherwise, setting $\tau = c_2 + \ell$, and setting $\upsilon$ to the sum
 of~$\ell$ and the lower bound for~$k$ described in Subsection~\ref{ssec:splitting_sum}, above.
 
 \subsection{Analysis for the Case $\boldsymbol{q > n^2}$}
 
A slight variant of the argument from Wiedemann~\cite{wied86} can be applied when $q > n^2$. Since
dense linear algebra is certainly adequate for computations on small matrices it will be assumed that
$n \ge 7$, so that $q \ge 49$. 

Suppose, once again, that $A$ is an $m \times n$ matrix over a field~$\F$ that is either infinite or has size $q > n^2$.
Let $\widehat{q}$ be the largest power of a prime that is less than or equal to~$n^2$.  It suffices to apply the
above construction, using the choices of $c_4$, $c_3$, $\beta_0$, $c_2$ and~$\ell$ appropriate for a field with
size~$\widehat{q}$ (so that $\ell = 0$) --- except that, after choosing entries of rows that might be nonzero, the
remaining entries of the matrix~$B$, to be filled in, should be chosen uniformly and independently from a finite
subset~$S$ of~$\F$ with size at least $n^2$, rather than from the finite field with size~$\widehat{q}$.

In order to see that this process is reliable, consider yet another matrix --- namely, a matrix~$\widehat{B}$
obtained by placing a distinct indeterminate into each row entry that is assigned a value from~$S$, above,
instead of~$0$. Let us denote the indeterminate placed into the $i^{\text{th}}$ new row, in column~$j$,
by $z_{i, j}$. Since $\ell = 0$ this results in an $n \times n$ matrix whose entries are elements of~$\F$
and indeterminates. Let $\widehat{f}$ be the determinant of this matrix --- a multivariate polynomial with
total degree at most $n-m$, since only $n-m$ rows include indeterminates, and each entry of such a row
has total degree at most one.

Since the rows of the $m \times n$ matrix~$A$ are linearly independent, there exists a sequence of
integers $i_1, i_2, \dots, i_m$ such that
\[
 1 \le i_1 < i_2 < \dots < i_m \le n
\]
and such that the $m \times m$ submatrix, including columns $i_1, i_2, \dots, i_m$, is nonsingular.
Permuting rows of~$A$ as needed, we may assume without loss of generality that the entry of
the $j^{\text{th}}$ row of~$A$ in column~$i_j$ is nonzero, for $1 \le j \le m$. Consequently, if the
$j^{\text{th}}$ row of~$A$ was replaced by a row whose $i_j^{\text{th}}$ entry is~$1$ and whose
other entries are~$0$, this would result in an $n \times m$ matrix~$\widehat{A}$ whose rows are
linearly independent as well. Indeed, the $m \times m$ submatrix including the entries in columns
$i_1, i_2,\dots, i_m$ would have determinant~$1$.

Similarly, the rows of this matrix are linearly independent when the entries are viewed as elements
of the finite field~$\F_{\widehat{q}}$ with size~$\widehat{q}$, instead of as elements of~$\F$. Let
us call this matrix (an $m \times n$ matrix with entries in~$\F_{\widehat{q}}$) $\widetilde{A}$. 
Note that the process, described above, to produce new rows to obtain~$B$ from~$A$, is independent
of the entries in the rows of~$A$ --- it only depends on the number~$m$ of rows and~$n$ of columns
of~$A$. With that noted, let us consider yet another $n \times n$ matrix, namely a matrix with entries
in~$F_{\widehat{q}}$ whose first $m$ rows are the rows of~$\widetilde{A}$ and whose remaining rows
are produced by initially deciding to set the same entries of rows to~$0$ as for the new rows of~$B$,
and whose remaining entries are chosen uniformly and independently from~$\F_{\widehat{q}}$. It follows
by the analysis for the case $q \le n^2$ that this matrix is nonsingular with some probability
$\sigma \ge \frac{8}{9}$.

Let us suppose that this is the case. Then there must exist a set of column indices $i_{m+1}, i_{m+2},
\dots, i_n$ such that
\[
 1 \le i_{m_1} < i_{m+2}, \dots, i_n \le n,
 \quad
 \{ i_1, i_2, \dots, i_m \} \cup \{i_{m+1}, i_{m+2}, \dots, i_n \} = \{1, 2, \dots, n \},
\]
and the entries of~$\widetilde{A}$ in row~$j$ and column~$i_j$ are all nonzero, for $1 \le j \le n$.
Consequently if $m+1 \le j \le n$ then the entry of~$\widehat{B}$ in row~$j$ and column~$i_j$ is an
indeterminate, $z_{j-m, i_j}$, rather than zero. It now follows that the above polynomial~$\widehat{f}$
is not identically zero (in this case): For if one sets the value of each indeterminate $z_{j-m, i_j}$ to
be~$1$ and one sets the value of all other indeterminates to be~$0$, then the value of this polynomial
is the product of $\pm 1$ and the determinant of the submatrix of~$A$ including the entries in columns
$i_1, i_2, \dots, i_m$ --- which is nonsingular, as noted above. It now follows by an application of the
Schwartz-Zippel lemma~\cite{schwar80, zipp79} that the above matrix~$B$ is singular, in this particular case,
with probability at most $\frac{n-m}{|S|} \le \frac{n}{|S|} \le \frac{1}{n}$.

It follows that the overall probability that $B$ is singular is at most
\[
 (1 - \sigma) + \textstyle{\frac{\sigma}{n}} \le \textstyle{\frac{1}{9} + \frac{8}{9n}},
\]
as needed to complete the proofs of Theorems~\ref{thm:constants_for_conditioner}
and~\ref{thm:second_constants_for_conditioner}.
  
\bibliographystyle{plain}

\begin{thebibliography}{1}

\bibitem{schwar80}
J.~T. Schwartz.
\newblock Fast probabilistic algorithms for verification of polynomial
  identities.
\newblock {\em Journal of the Association of Computing Machinery}, 27:701--717,
  1980.

\bibitem{wied86}
D.~H. Wiedemann.
\newblock Solving sparse linear equations over finite fields.
\newblock {\em IEEE Transactions on Information Theory}, 32:54--62, 1986.

\bibitem{zipp79}
R.~Zippel.
\newblock Probabilistic algorithms for sparse polynomials.
\newblock In {\em EUROSAM '79}, volume~72 of {\em Lecture Notes in Computer
  Science}, pages 216--226. Springer-Verlag, 1979.

\end{thebibliography}

\begin{thebibliography}{4}

\bibitem[4]{macslo77}
J.~J. MacWilliams and N.~J.~A. Sloan.
\newblock {\em The Theory of Error-Correcting Codes}.
\newblock North-Holland, 1977.
\end{thebibliography}

\appendix
\section{The Rest of Wiedemann's Argument}

\newcommand{\macsloane}{4}

This appendix includes additional details of Wiedemann's proof of his Theorem~$1'$. While notation has
been changed to agree with the rest of this report, and a few more details have been included, this part
of the proof is essentially as given by Wiedemann~\cite{wied86}. The bulk of this is the beginning of the derivation of an
upper bound for the probability that $B$ has rank less than~$n$ when $k > c_3 \ln n$ and the additional
rows of~$B$ are chosen as described in Section~\ref{sec:modified_proof}.

\subsection{Getting Started}

As noted by Wiedemann~\cite{wied86}, the number of nonzero entries in a vector is called its
\textbf{\emph{Hamming weight}}.
Wiedemann's argument begins with a consideration of a subspace~$C$ of~$\matgrp{\F}{1}{n}$ with dimension~$m$ ---
specifically, the row space of the matrix~$A$ introduced at the beginning of Section~\ref{sec:modified_proof}.
For $0 \le j \le n$ the number of elements of~$C$ with Hamming
weight~$j$ is denoted by~$a[j]$, and the \textbf{\emph{weight enumerator polynomial}} for the vector space~$C$ is
defined to be the polynomial
\[
 a(r) = \sum_{j=0}^n a[j] r^j \in \F[r].
\]

Wiedemann begins by establishing the following claims, which concern the weight enumerator polynomial of
a subspace~$C$ of~$\matgrp{\F}{1}{n}$ with dimension~$m$. Short, readable proof of each of the following
can be found in Wiedemann~\cite{wied86}.

\textbf{Proposition 1} (Wiedemann \cite{wied86}:
Let $C$ be any subspace of~$\matgrp{\F}{1}{n}$ with dimension~$m$.
Let $a[j]$ denote the number of elements of~$C$ with Hamming weight~$j$. Then for each integer~$i$ such that
$0 \le i \le n$,
\[
 \sum_{j=0}^i a[j] \le \sum_{j=0}^i \binom{m}{j} (q-1)^j.
\]

\textbf{Proposition 2} (Wiedemann \cite{wied86}):
If $C$ is a subspace of~$\matgrp{\F}{1}{n}$ with dimension~$m$,
and $a \in \F[r]$ is the weight enumerator polynomial for~$C$,
then, for $0 \le r \le 1$, $a(r) \le (1 + (q-1)r)^m$.

\subsection{Getting to Equations~\eqref{eq:bound_for_rho}--\eqref{eq:definition_of_beta}}

Wiedemann continues by considering the probability that a specific $\F$-linear combination
of $j$ of the first $k$ generated rows (with all multipliers nonzero) yields a particular vector that includes
$i$ nonzero entries. The probability that a fixed entry of this vector is zero is
\begin{align*}
 z^j + \frac{1}{q} \left( \sum_{h=0}^{j-1} \binom{j}{h} z^h (1-z)^{j-h} \right) 
   &= \left( \frac{q-1}{q} \right) z^j + \frac{1}{q} \left( \sum_{h=0}^j \binom{j}{h} z^h (1-z)^{j-h} \right) \\
   &= \left(\frac{q-1}{q}\right) z^j + \frac{1}{q} \tag{by the Binomial Theorem} \\
   &= z^j + \frac{1}{q}(1-z^j).
\end{align*}
On the other hand, the probability that a fixed entry of this vector has a specific nonzero value in~$\F$ is
\begin{align*}
\frac{1}{q} \sum_{h=0}^{j-1} \binom{j}{h} z^h (1-z)^{j-1}
 &= \frac{1}{q} \sum_{h=0}^j \binom{j}{h} z^h (1-z)^{j-h} - \frac{1}{q} z^j \\
 &= \frac{1}{q} (1 - z^j) \tag*{(by the Binomial Theorem, once again).}
\end{align*}
It therefore follows that a specific $\F$-linear combination of $j$ of the first $k$ generated rows (with all
multipliers nonzero) yields a particular vector including $i$ nonzero entries is
\[
 \left( z^j + \frac{1}{q}\left(1-z^j\right)\right)^{n-i} \left(\frac{1}{q} \left(1-z^j\right)\right)^i,
\]
as claimed at line~(5) in Wiedeman~\cite{wied86}.

Continuing to follow Wiedemann's argument, let $a \in \F[r]$ be the weight enumerator polynomial for the
row space of~$A$, and let $\rho$ be the probability that the rows of~$A$ and
the first $k$ (sparse) vectors that have been generated are linearly dependent.

For $0 \le i \le n$ there are (by definition) $a[i]$ vectors in the row space of~$A$ that have exactly $i$ nonzero
entries. For $1 \le j \le k$ there are $\binom{k}{j} (q-1)^j$ ways to choose $\F$-linear combinations of the $k$
generated rows, for which exactly $j$ of the multipliers are nonzero.  As noted, again, by Wiedemann, it now
follows that
\begin{align*}
\rho &\le \sum_{i=0}^n a[i] \sum_{j=1}^k \binom{k}{j} (q-1)^j \left(z^j + \frac{1}{q}(1-z^j)\right)^{n-i} 
  \left(\frac{1}{q}(1-z^j)\right)^i \\
   &= \sum_{j=1}^k \binom{k}{j} (q-1)^j \left(z^j + \frac{1}{q} (1-z^j)\right)^n
     \sum_{i=0}^n a[i] \left( \frac{q^{-1} (1-z^j)}{z^j + q^{-1} (1-z^j)} \right)^i \\
   &= \sum_{j=1}^k \binom{k}{j} (q-1)^j \left( z^j + \frac{1}{q} (1-z^j\right)^n
     a\left(\frac{q^{-1} (1-z^j)}{z^j + q^{-1}(1-z^j)}\right).
\end{align*}

Now, by Wiedemann's Proposition~2 (since the row space of~$A$ has dimension~$m$)
\begin{align*}
 a\left(\frac{q^{-1}(1-z^j)}{z^j + q^{-1}(1-z^j)}\right)
   &\le \left(1 + (q-1) \left(\frac{q^{-1}(1-z^j)}{z^j + q^{-1}(1-z^j)}\right)\right)^m \\
   &= \left( \frac{z^j + q^{-1}(1-z^j) + (1-z^j) - q^{-1}(1-z^j)}{z^j + q^{-1}(1-z^j)} \right)^m \\
   &= \left( z^j + q^{-1} (1-z)^j \right)^{-m}
\end{align*}
so that
\begin{equation}
\label{eq:first_inequality_for_rho}
\begin{split}
\rho &\le \sum_{j=1}^k \binom{k}{j} (q-1)^j \left( z^j + q^{-1} (1-z^j) \right)^{n-m} \\
  &= \sum_{j=1}^k \binom{k}{j} (q-1)^j \left( q^{-1} + \frac{(q-1)}{q} z^j \right)^{n-m}.
\end{split}
\end{equation}

Recall that $k > c_3 \ln n$ and that $z = 1 - \frac{c_3 \ln n}{k} \ge 0$. Then, since $e^x \ge 1 + x$ for all $x \in \R$,
\[
 e^{-\frac{c_3 \ln n}{k}} \ge 1 - \frac{c_3 \ln n}{k} = z.
\]
Consequently, for $j \ge 1$,
\[
 z^j \le \left( e^{-\frac{c_3 \ln n}{k}} \right)^j = e^{-\frac{c_3 j \ln n}{k}}.
\]

Recall that $\beta=\frac{j}{k}$ and that $c_4 = \frac{c_3}{3}$ (as shown at lines~\eqref{eq:definition_of_beta}
and~\eqref{eq:definition_of_c4}). By assumption, $q \le n^2$ so that $q^{-1} \ge n^{-2}$. Now, since $c_4 \beta \ge 0$,
\[
 (nq)^{-c_4 \beta} \ge n^{-3 c_4 \beta} = n^{-c_3 \beta} = e^{-\frac{c_3 j \ln n}{k}} \ge z^j.
\]
It now follows by the equation at line~\eqref{eq:first_inequality_for_rho}, above, that
\[
 \rho \le \sum_{j=1}^k \binom{k}{j} (q-1)^j \left(q^{-1} + \frac{(q-1)}{q} (nq)^{-c_4 \beta}\right)^{n-m},
\]
as shown at line~(6) of Wiedemann~\cite{wied86}. Splitting the above sum at $k \beta_0$, the equations
at lines~\eqref{eq:bound_for_rho}--\eqref{eq:definition_of_beta} are now obtained.

\subsection{Getting to Equation~\eqref{eq:bound_for_rho1}}

Suppose (as shown at line~\eqref{eq:constraints_on_c2_and_c4}, above) that $c_2 \ge \frac{3}{2} \ge \log_2 e$
and $c_4 \ge \frac{2}{\beta_0}$. Then $\ln (n-m) \ge \ln c_2 \ge \ln \log_2 e$, so that
\[
 c_4 \beta_0 \ge 2 \ge \frac{\ln (q-1) + \ln (n-m) + \ln \log_2 e}{\ln q + \ln n}.
\]
Recall that $\beta = \frac{j}{k}$, so that
\begin{align*}
\rho_1 &= \sum_{k \beta_0 \le j \le k} \binom{k}{j} (q-1)^j \left( q^{-1} + \frac{(q-1)}{q} (nq)^{-c_4 \beta}\right)^{n-m} \\
 &\le \sum_{k \beta_0 \le j \le k}  \binom{k}{j} (q-1)^j  \left( q^{-1} + \frac{(q-1)}{q} (nq)^{-c_4 \beta_0} \right)^{n-m} \\
   \tag{since $-c_4 \beta \le - c_4 \beta_0$ if $j \ge k \beta_0$} \\
 &\le \sum_{k \beta_0 \le j \le k} \binom{k}{j} (q-1)^j \left( q^{-1}
   + \frac{(q-1)}{q} (nq)^{-\frac{\ln (q-1) + \ln (n-m) + \ln \log_2 e}{\ln q + \ln n}} \right)^{n-m}
     \tag{by the above inequality} \\
&= \sum_{k \beta_0 \le j \le k} \binom{k}{j} (q-1)^j \left( q^{-1} + \frac{(q-1)}{q}
  e^{-(ln(q-1) + \ln (n-m) + \ln \log_2 e)} \right)^{n-m} \\
&= \sum_{k \beta_0 \le j \le k} \binom{k}{j} (q-1)^j  \left( q^{-1} + \frac{(q-1)}{q} \frac{1}{(q-1)(n-m)\log_2 e} \right)^{n-m} \\
&= \sum_{k \beta_0 \le j \le k} \binom{k}{j} (q-1)^j \left( \frac{1}{q} \left( 1+ \frac{\ln 2}{(n-m)} \right) \right)^{n-m} \\
&\le \sum_{k \beta_0 \le j \le k} \binom{k}{j} (q-1)^j \left( \frac{1}{q} e^{\frac{ln 2}{(n-m)}} \right)^{n-m}
  \tag{again, since $e^x \ge 1+x$ for all $x \in \R$} \\
&= \sum_{k \beta_0 \le j \le k} \binom{k}{j} (q-1)^j \left( \frac{2^{\frac{1}{n-m}}}{q} \right)^{n-m} \\
&= 2 q^{m-n} \sum_{k \beta_0 \le j \le k} \binom{k}{j} (q-1)^j \\
&= 2 q^{-(c_2 + k)} \sum_{k \beta_0 \le j \le k} \binom{k}{j} (q-1)^j \tag{since $n = m + c_2 + k$} \\
&\le 2 q^{-(c_2+k)} \sum_{j=0}^k \binom{k}{j} (q-1)^j \\
&= 2 q^{-(c_2+k)} \cdot q^k \tag{by the Binomial Theorem} \\
&= 2 q^{-c_2},
\end{align*}
as claimed at line~\eqref{eq:bound_for_rho1}, above.

\subsection{Getting to Equations~\eqref{eq:first_bound_for_rho0} and~\eqref{eq:definition_of_f}}

Wiedemann continues by applying an inequality for $\binom{k}{j}$, citing Lemma~10.7 of MacWilliams
and Sloane~[\macsloane]. However, this lemma seems to establish a slightly different inequality. Furthermore,
while a proof of a different bound is provided, the proof of this one is left as an exercise for the reader. With that
noted, the inequality used here by Wiedemann is as follows.

\begin{lemma}
\label{lem:macslo}
If $1 \le j \le k-1$ and $\beta = \frac{j}{k}$ then
\[
 \binom{k}{j} \le \left(\beta^{-\beta} (1-\beta)^{-(1-\beta)} \right)^k.
\]
\end{lemma}

\begin{proof}
The inequality is easily verified when $k=2$ (so that $j=1$ and $\beta = \frac{1}{2}$), so it is sufficient to consider the case
that $k \ge 3$.

As asserted by MacWilliams and Sloane, one (lesser known and more precise) form of Stirling's approximation
asserts that if $\ell \ge 1$ then
\[
\sqrt{2\pi} \ell^{\ell + \frac{1}{2}} e^{-\ell + \frac{1}{12 \ell} - \frac{1}{360 \ell^3}} < \ell ! <
 \sqrt{2 \pi} \ell^{\ell + \frac{1}{2}} e^{-\ell + \frac{1}{12 \ell}}.
\]
It follows from this that if $1 \le j \le k-1$ then
\begin{align*}
\binom{k}{j} &\le \frac{\sqrt{2\pi} k^{k+\frac{1}{2}} e^{-k + \frac{1}{12k}}}{\sqrt{2\pi} j^{j+\frac{1}{2}}
  e^{-j + \frac{1}{12j} - \frac{1}{360 j^3}} \sqrt{2\pi} (k-j)^{k-j+\frac{1}{2}} e^{-(k-j) + \frac{1}{12(k-j)} - \frac{1}{360(k-j)^3}}} \\
 &= \frac{1}{\sqrt{2\pi \beta (1-\beta) k}} \left(\frac{1}{\beta^{\beta}(1-\beta)^{(1-\beta)}} \right)^k
  e^{\frac{1}{12k} - \frac{1}{12j} - \frac{1}{12(k-j)} + \frac{1}{360 j^3} + \frac{1}{360 (k-j)^3}} \\
 &= \frac{1}{\sqrt{2\pi \beta (1-\beta) k}} \left( \beta^{-\beta} (1-\beta)^{-(1-\beta)} \right)^k
    e^{\frac{1}{12k}  \left( 1 - \frac{1}{\beta} - \frac{1}{1-\beta} + \frac{1}{30 k^2 \beta^3}
      + \frac{1}{30 k^2 (1-\beta)^3}\right)} .\\
\end{align*}

Now suppose that either $\beta \le \frac{1}{3}$ or $\beta \ge \frac{2}{3}$. Consider the function
$g(x) = x^{-1} + (1-x)^{-1}$. It is easily checked that $g'(x) = -x^{-2} + (1-x)^{-2}$, so that $g(1/2)  = 0$, and that
$g''(x) = 2x^{-3} + 2(1-x)^{-3} > 0$ when $0 < x < 1$,  so that $g'(x) < 0$ when $0 < x < \frac{1}{2}$ and
$g'(x) > 0$ when $\frac{1}{x} < x < 1$. Consequently, if $0 < \beta \le \frac{1}{3}$ then $g(\beta) \ge g(1/3) = 6$,
and if $\frac{2}{3} \le \beta < 1$ then $g(\beta) \ge g(2/3) = 6$ as well.

Next consider the function $h(x) = x(1-x) = x - x^2$; $h'(x) = 1 -2x$ and $h''(x) = -2 < 0$, so that $h'(x) > 0$ when
$0 < x < \frac{1}{2}$, $h'(1/2) = 0$, and $h'(x) < 0$ when $\frac{1}{2} < x < 1$. Consequently if
$\frac{1}{k} \le x \le 1 - \frac{1}{k}$ then, since $h(1/k) = h(1 - 1/k) = \frac{1}{k} - \frac{1}{k^2}$,
$h(x) \ge \frac{1}{k} - \frac{1}{k^2}$.

Thus if $0 < \beta \le \frac{1}{3}$ or $\frac{2}{3} \le \beta < 1$ then (since $\frac{1}{k} \le \beta \le 1 - \frac{1}{k}$ as well, so
that $k \ge 3$),
\begin{align*}
\binom{k}{j} &\le \frac{1}{\sqrt{2\pi \beta (1-\beta) k}} \left(\beta^{-\beta} (1-\beta)^{-(1-\beta)} \right)^k
    e^{\frac{1}{12k}  \left( 1 - \frac{1}{\beta} - \frac{1}{1-\beta} + \frac{1}{30 k^2 \beta^3} + \frac{1}{30 k^2 (1-\beta)^3}\right)} \\
  &\le \frac{1}{\sqrt{2\pi \beta (1-\beta) k}} \left(\beta^{-\beta} (1-\beta)^{-(1-\beta)} \right)^k
    e^{\frac{1}{12k}  \left( 1 - \frac{1}{\beta} - \frac{1}{1-\beta} + \frac{1}{30 \beta} + \frac{1}{30 (1-\beta)}\right)} 
      \tag{since $k^2 \beta^2 \ge 1$ and $k^2 (1-\beta)^2 \ge 1$} \\
 &= \frac{1}{\sqrt{2\pi h(\beta) k}} \left(\beta^{-\beta} (1-\beta)^{-(1-\beta)} \right)^k
    e^{\frac{1}{12k}  \left( 1 - \frac{29}{30} g(\beta) \right)}  \\
 &\le \frac{1}{\sqrt{2\pi (1/k - 1/k^2)  k}} \left(\beta^{-\beta} (1-\beta)^{-(1-\beta)} \right)^k
    e^{\frac{1}{12k}  \left( 1 - \frac{29}{5} \right)}  
   \tag{since $g(\beta) \ge 6$ and $h(\beta) \ge \frac{1}{k} - \frac{1}{k^2}$} \\
 &= \frac{1}{\sqrt{2\pi (1-1/k)}} \left(\beta^{-\beta} (1-\beta)^{-(1-\beta)} \right)^k e^{-\frac{2}{5k}}  \\
 &\le \frac{1}{\sqrt{4\pi/3}} \left(\beta^{-\beta} (1-\beta)^{-(1-\beta}\right)^k e^{-\frac{2}{5k}} 
     \tag{since $1 - \frac{1}{k} \ge \frac{2}{3}$} \\
 &< \left(\beta^{-\beta} (1-\beta)^{-(1-\beta)}\right)^k,
\end{align*}
since  $1/\sqrt{4\pi/3} < 1$ and $e^{-\frac{2}{5k}} < 1$ as well.

Suppose next that $\frac{1}{3} \le \beta \le \frac{2}{3}$. In this case $g(\beta) \ge g(1/2) = 4$, and $h(\beta) \ge h(1/3) = 
h(2/3) = \frac{2}{9}$. A consideration of the function $j(x) = x^{-3} + (1-x)^{-3}$ along with its first and second derivatives
confirms that if $\frac{1}{3} \le \beta \le \frac{2}{3}$ then $j(\beta) \le j(1/3) = j(2/3) = 54$, so that
\[
 \frac{1}{30k^2 \beta^3} + \frac{1}{30 k^2 (1-\beta)^3}
   = \frac{1}{30 k^2} j(\beta) \le \frac{1}{30 k^2} j(1/3) = \frac{9}{5k^2}.
\]
It now follows that
\begin{align*}
\binom{k}{j}
 &\le \frac{1}{\sqrt{2\pi \beta(1-\beta) k}} \left( \beta^{-\beta} (1-\beta)^{-(1-\beta)} \right)^k
   e^{\frac{1}{12k} \left( 1 - \frac{1}{\beta} - \frac{1}{1-\beta} + \frac{1}{30 k^2 \beta^3} + \frac{1}{k^2 (1-\beta)^3}\right)} \\
 &\le \frac{1}{\sqrt{2\pi h(\beta) k}} \left( \beta^{-\beta} (1-\beta)^{-(1-\beta)} \right)^k
   e^{\frac{1}{12k} \left( 1 - g(\beta) + \frac{9}{5k^2}\right)} \\
 &\le \frac{1}{\sqrt{4\pi k/9}} \left(\beta^{-\beta} (1-\beta)^{-(1-\beta)} \right)^k
  e^{\frac{1}{12k} \left(-3 + \frac{9}{5k^2} \right)}
   \tag{since $g(\beta) \ge 4$ and $h(\beta) \ge \frac{2}{9}$} \\
 &\le \frac{1}{\sqrt{\pi/3}} \left(\beta^{-\beta} (1-\beta)^{-(1-\beta)}\right)^k e^{-\frac{7}{30k}}
   \tag{since $k \ge 3$, so that $-3 + \frac{9}{5k^2} \le -\frac{14}{5}$} \\
 &< \left( \beta^{-\beta} (1-\beta)^{-(1-\beta)}\right)^k,
\end{align*}
as required, since $\frac{1}{\sqrt{\pi/3}} < 1$ and $e^{-\frac{7}{30k}} < 1$ as well.
\end{proof}

It now follows that, since $nq > 1$, $-c_4\beta < 0$, and $n-m \ge k$, so that $\frac{1}{q} + \frac{(q-1)}{q} (nq)^{-c_4 \beta} < 1$,
\begin{align*}
\rho_0
 &= \sum_{1 \le j < k \beta_0} \binom{k}{j} (q-1)^j \left( q^{-1} + \frac{(q-1)}{q} (nq)^{-c_4 \beta} \right)^{n-m} \\
 &< \sum_{1 \le j < k \beta_0} \binom{k}{j} (q-1)^j \left( q^{-1} + \frac{(q-1)}{q} (nq)^{-c_4 \beta} \right)^k \\
 &\le \sum_{1 \le j < k \beta_0} \left( \beta^{-\beta} (1-\beta)^{-(1-\beta)} \right)^k (q-1)^{\beta k}
   \left(q^{-1} + \frac{(q-1)}{q} (nq)^{-c_4 \beta}\right)^k
     \tag{recalling that $\beta = \frac{j}{k}$, and applying Lemma~\ref{lem:macslo}} \\
 &= \sum_{1 \le j < k \beta_0} \left( \frac{(q-1)^\beta}{\beta^{\beta} (1-\beta)^{1-\beta}}
   \left(\frac{1}{q} + \frac{q-1}{q} (nq)^{-c_4 \beta}\right) \right)^k,
\end{align*}
as shown at line~(7) in Wiedemann~\cite{wied86}. Thus
\[
 \rho_0 \le \sum_{1 \le j \le k \beta_0}\left( \frac{1}{\beta^{\beta} (1-\beta)^{1-\beta}} f(q) \right)^k,
\]
as claimed at line~\eqref{eq:first_bound_for_rho0}, above, when $f$ is as defined at line~\eqref{eq:definition_of_f}.

\section{Derivation of Additional Constants Shown in Figure~\ref{fig:choices_of_c4}}
\label{sec:continuation_of_argument}

This appendix provides a derivations of additional constants, shown in Figure~\ref{fig:choices_of_c4}, that
were not derived in Subsection~\ref{ssec:bounding_rho0}.

To begin, one can extend Lemma~\ref{lem:getting_rid_of_pesky_term} as follows.

\begin{lemma}
\label{lem:continuing_to_get_rid_of_pesky_term}
Once again, consider the relationship between~$(1-x)^{-(1-x)}$ and $x^{\gamma x}$, for a negative
constant~$\gamma$, when $x$ is small and positive.
\begin{enumerate}
\renewcommand{\labelenumi}{\text{\textup{(\alph{enumi})}}}
\setlength{\itemsep}{0pt}
\item If $0 < x \le \frac{9}{50}$ then $(1-x)^{-(1-x)} \le x^{\gamma x}$, when $\gamma = -\frac{11}{20}$.
\item If $0 < x \le \frac{19}{100}$ then $(1-x)^{-(1-x)} \le x^{\gamma x}$, when $\gamma = -\frac{9}{16}$.
\item If $0 < x \le \frac{7}{50}$ then $(1-x)^{-(1-x)} \le x^{\gamma x}$, when $\gamma = -\frac{19}{40}$.
\item If $0 < x \le \frac{2}{25}$ then $(1-x)^{-(1-x)} \le x^{\gamma x}$, when $\gamma = -\frac{2}{5}$.
\item If $0 < x \le \frac{1}{20}$ then $(1 - x)^{-(1-x)} \le x^{\gamma x}$, when $\gamma = -\frac{1}{3}$.
\item If $0 < x \le \frac{1}{53}$ then $(1-x)^{(1-x)} \le x^{\gamma x}$, when $\gamma = -\frac{1}{4}$.
\item If $0 < x \le \frac{1}{200}$ then $(1-x)^{(1-x)} \le x^{\gamma x}$, when $\gamma = -\frac{19}{100}$.
\item If $0 < x \le \frac{1}{750}$ then $(1-x)^{(1-x)} \le x^{\gamma x}$, when $\gamma = -\frac{23}{150}$.
\item If $0 < x \le \frac{1}{2000}$ then $(1-x)^{(1-x)} \le x^{\gamma x}$, when $\gamma = -\frac{33}{250}$.
\end{enumerate}
\end{lemma}

\begin{proof}
Consider the function $g(x) = \gamma x \ln x + (1-x) \ln (1-x)$ when $\gamma$ is a negative constant
and $0 < x < 1$. As noted in the proof of Lemma~\ref{lem:getting_rid_of_pesky_term}, it suffices
(for $0 < \delta < 1$) to confirm that $g'(\delta) < 0$ and $g(\delta) \ge 0$ in order to confirm that
$g(\delta) \ge 0$ for $0 \le x \le \delta$. It then follows that $(1-x)^{-(1-x)} \le x^{\gamma x}$ for
$0 \le x \le \delta$ as well.

Part~(a) of the claim can now be established by choosing $\gamma = -\frac{11}{20}$ and $\delta = \frac{9}{50}$;
then $g'(\delta) = -\frac{11}{20} \ln \frac{9}{50} - \ln \frac{41}{50} - \frac{31}{20} < -0.4 < 0$ and
$g(\delta) = -\frac{99}{1000} \ln \frac{9}{50} + \frac{41}{50} \ln \frac{41}{50} > 0.007 > 0$, as required.

Part~(b) of the claim can be established by choosing $\gamma = -\frac{9}{16}$ and $\delta = \frac{19}{100}$; then
$g'(\delta) = -\frac{9}{16} \ln \frac{19}{100} - \ln \frac{81}{100} - \frac{25}{16} < - 0.4 < 0$ and
$g(\delta) = -\frac{171}{1600} \ln \frac{19}{100} + \frac{81}{100} \ln \frac{81}{100} > 0.0006 > 0$, as needed.

Part~(c) of the claim can be established by choosing $\gamma = -\frac{19}{40}$ and $\delta = \frac{7}{50}$; then
$g'(\delta) = -\frac{19}{40} \ln \frac{7}{50} - \ln \frac{43}{50} - \frac{59}{40} < - 0.3 < 0$ and
$g(\delta) = -\frac{133}{2000} \ln \frac{7}{50} + \frac{43}{50} \ln \frac{43}{50} > 0.001 > 0$, as needed.

Part~(d) of the claim can be established by choosing $\gamma = -\frac{2}{5}$ and $\delta = \frac{2}{25}$; then
$g'(\delta) = -\frac{2}{5} \ln \frac{2}{25} - \ln \frac{23}{25} - \frac{7}{5} < - 0.3 < 0$ and
$g(\delta) = -\frac{4}{125} \ln \frac{2}{25} + \frac{23}{25} \ln \frac{23}{25} > 0.004 > 0$, as needed.

Part~(e) of the claim can be established by choosing $\gamma = -\frac{1}{3}$ and $\delta = \frac{1}{20}$; then
$g'(\delta) = \frac{1}{3} \ln 20 - \ln \frac{19}{20} - \frac{4}{3} < - 0.28 < 0$ and
$g(\delta) = \frac{1}{60} \ln 20 + \frac{19}{20} \ln \frac{19}{20} > 0.001 > 0$, as needed.

Part~(f) of the claim can be established by choosing $\gamma = -\frac{1}{4}$ and $\delta = \frac{1}{53}$; then
$g'(\delta) = \frac{1}{4} \ln 53 - \ln \frac{52}{53} - \frac{5}{4} < - 0.23 < 0$ and
$g(\delta) = \frac{1}{212} \ln 53 + \frac{52}{53} \ln \frac{52}{53} > 0.00003 > 0$, as needed.

Part~(g) of the claim can be established by choosing $\gamma = -\frac{19}{100}$ and $\delta = \frac{1}{200}$; then
$g'(\delta) = \frac{19}{100} \ln 200 - \ln \frac{199}{200} - \frac{199}{200} < - 0.17 < 0$ and
$g(\delta) = \frac{19}{8000} \ln 80 + \frac{79}{80} \ln \frac{79}{80} > 0.00004 > 0$, as needed.

Part~(h) of the claim can be established by choosing $\gamma = -\frac{23}{150}$ and $\delta = \frac{1}{750}$; then
$g'(\delta) = \frac{23}{150} \ln 750 - \ln \frac{749}{750} - \frac{173}{150} < - 0.13 < 0$ and
$g(\delta) = \frac{23}{75000} \ln 500 + \frac{499}{500} \ln \frac{499}{500} > 0.00002 > 0$, as needed.

Part~(i) of the claim can be established by choosing $\gamma = -\frac{33}{250}$ and $\delta = \frac{1}{2000}$;
then $g'(\delta) = \frac{33}{250} \ln 2000 - \ln \frac{1999}{2000} - \frac{283}{250} < - 0.12 < 0$ and
$g(\delta) = \frac{33}{500000} \ln 2000 + \frac{1999}{2000} \ln \frac{1999}{2000} > 0.000001 > 0$, as needed.
\end{proof}

Similarly, one can extend Lemma~\ref{lem:getting_rid_of_power_of_q} as follows.

\begin{lemma}
\label{lem:continuing_to_get_rid_of_power_of_q}
Once again, consider the relationship between $\zeta^x$ and $x^{\delta x}$ when $\zeta$ and~$\delta$ are
positive constants.
\begin{enumerate}
\renewcommand{\labelenumi}{\text{\textup{(\alph{enumi})}}}
\setlength{\itemsep}{0pt}
\item If $0 < x \le \frac{1}{5}$ then $3^x \le x^{\delta x}$ when $\delta = -\frac{7}{10}$.
\item If $0 < x \le \frac{1}{5}$ then $4^x \le x^{\delta x}$ when $\delta = -\frac{9}{10}$.
\item If $0 < x \le \frac{1}{5}$ then $6^x \le x^{\delta x}$ when $\delta = -\frac{6}{5}$.
\item If $0 < x \le \frac{1}{5}$ then $7^x \le x^{\delta x}$ when $\delta = -\frac{61}{50}$.
\item If $0 < x \le \frac{9}{50}$ then $8^x \le x^{\delta x}$ when $\delta = -\frac{5}{4}$.
\item If $0 < x \le \frac{19}{100}$ then $10^x \le x^{\delta x}$ when $\delta = -\frac{7}{5}$.
\item If $0 < x \le \frac{9}{50}$ then $12^x \le x^{\delta x}$ when $\delta = -\frac{3}{2}$.
\item If $0 < x \le \frac{7}{50}$ then $15^x \le x^{\delta x}$ when $\delta = -\frac{7}{5}$.
\item If $0 < x \le \frac{2}{25}$ then $22^x \le x^{\delta x}$ when $\delta = -\frac{5}{4}$.
\item If $0 < x \le \frac{1}{20}$ then $30^x \le x^{\delta x}$ when $\delta = -\frac{8}{7}$.
\item If $0 < x \le \frac{1}{53}$ then $46^x \le x^{\delta x}$ when $\delta = -\frac{49}{50}$.
\item If $0 < x \le \frac{1}{200}$ then $60^x \le x^{\delta x}$ when $\delta = -\frac{4}{5}$.
\item If $0 < x \le \frac{1}{750}$ then $72^x \le x^{\delta x}$ when $\delta = -\frac{7}{10}$.
\item If $0 < x \le \frac{1}{2000}$ then $88^x \le x^{\delta x}$ when $\delta = - \frac{3}{5}$.
\end{enumerate}
\end{lemma}

\begin{proof}
As explained in the proof of Lemma~\ref{lem:getting_rid_of_power_of_q}, it suffices to
consider the function $h(x) = \delta x \ln x - x \ln \zeta$ when $\delta$ is a negative constant and $\zeta$
is a positive one. As explained in that proof, if $h(\rho) \ge 0$ for another positive value~$\rho$ then
$h(x) \ge 0$ for $0 \le x \le \rho$ as well, and it follows that $\zeta^x \le x^{\delta x}$ for $0 < x \le \rho$ as
well.

Part~(a) of the claim can now be established by setting $\zeta = 3$, $\delta = -\frac{7}{10}$ and
$\rho = \frac{1}{5}$ and confirming that $h(\rho) = \frac{7}{50} \ln 5 - \frac{1}{5} \ln 3 > 0.005$.

Part~(b) of the claim can be established by setting $\zeta = 4$, $\delta = -\frac{9}{10}$, and 
$\rho = \frac{1}{5}$, and confirming that $h(\rho) = \frac{9}{50} \ln 5 - \frac{2}{5} \ln 2 > 0.01$.

Part~(c) of the claim can be established by setting $\zeta = 6$, $\delta = -\frac{6}{5}$, and
$\rho = \frac{1}{5}$, and confirming that $h(\rho) = \frac{6}{25} \ln 5 - \frac{1}{5} \ln 6 > 0.02$.

Part~(d) of the claim can be established by setting $\zeta = 7$, $\delta = -\frac{61}{50}$, and
$\rho = \frac{1}{5}$, and confirming that $h(\rho) = \frac{61}{250} \ln 5 - \frac{1}{5} \ln 7 > 0.003$.

Part~(e) of the claim can be established by setting $\zeta = 8$, $\delta = -\frac{5}{4}$, and
$\rho = \frac{9}{50}$, and confirming that $h(\rho) = -\frac{9}{40} \ln \frac{9}{50} - \frac{27}{50} \ln 2 > 0.01$.

Part~(f) of the claim can be established by setting $\zeta = 10$, $\delta = -\frac{7}{5}$, and
$\rho = \frac{19}{100}$, and confirming that $h(\rho) = -\frac{133}{500} \ln \frac{19}{100} - \frac{19}{100} \ln 10
> 0.004$.

Part~(g) of the claim can be established by setting $\zeta = 12$, $\delta = -\frac{3}{2}$, and
$\rho = \frac{9}{50}$, and confirming that $h(\rho) = -\frac{27}{100} \ln \frac{9}{50} - \frac{9}{50} \ln 12 > 0.01$.

Part~(h) of the claim can be established by setting $\zeta = 15$, $\delta = -\frac{7}{5}$, and
$\rho = \frac{7}{50}$, and confirming that $h(\rho) = -\frac{49}{250} \ln \frac{7}{50} - \frac{7}{50} \ln 15 > 0.006$.

Part~(i) of the claim can be established by setting $\zeta = 22$, $\delta = -\frac{5}{4}$, and
$\rho = \frac{2}{25}$, and confirming that $h(\rho) = -\frac{1}{10} \ln \frac{2}{25} - \frac{2}{25} \ln 22 > 0.005$.

Part~(j) of the claim can be established by setting $\zeta = 30$, $\delta = -\frac{8}{7}$, and
$\rho = \frac{1}{20}$, and confirming that $h(\rho) = \frac{2}{35} \ln 20 - \frac{1}{20} \ln 30 > 0.001$.

Part~(k) of the claim can be established by setting $\zeta = 46$, $\delta = -\frac{49}{50}$, and
$\rho = \frac{1}{53}$, and confirming that $h(\rho) = \frac{49}{2650} \ln 53 - \frac{1}{53} \ln 46 > 0.013$.

Part~(l) of the claim can be established by setting $\zeta = 60$, $\delta = -\frac{4}{5}$, and
$\rho = \frac{1}{200}$, and confirming that $h(\rho) = \frac{1}{250} \ln 200 - \frac{1}{200} \ln 60 > 0.0007$.

Part~(m) of the claim can be established by setting $\zeta = 72$, $\delta = -\frac{7}{10}$, and
$\rho = \frac{1}{750}$, and confirming that $h(\rho) = \frac{7}{7500} \ln 750 - \frac{1}{750} \ln 72 > 0.0004$.

Part~(n) of the claim can be established by setting $\zeta = 88$, $\delta = -\frac{3}{5}$, and
$\rho = \frac{1}{2000}$, and confirming that $h(\rho) = \frac{3}{1000} \ln 2000 - \frac{1}{2000} \ln 88 > 0.00004$.
\end{proof}

\subsection{Analysis for the Case $q=4$}

In this case it follows by the inequality at line~\eqref{eq:first_bound_for_rho0} that
\[
\rho_0 \le \sum_{1 \le j < k \beta_0} \frac{1}{\beta^{\beta} (1-\beta)^{1-\beta}} f(4)
  = \frac{3^{\beta}}{\beta^{\beta} (1-\beta)^{1-\beta}} \left( {\textstyle{\frac{1}{4}}} + \textstyle{\frac{3}{4}} (4n)^{-c_4 \beta} \right)
  \quad \text{when $q = 4$.}
\]

It will now be shown that
\begin{equation}
\label{eq:bound_for_4}
  \sum_{1 \le j < k \beta_0} \frac{3^{\beta}}{\beta^{\beta} (1-\beta)^{1-\beta}} \left( \frac{1}{4} +
   \frac{3}{4} (4n)^{-c_4 \beta} \right) \le \beta^{\beta}
\end{equation}
when $c_4 = \frac{25}{4}$ and $0 < \beta \le \frac{8}{25}$.

To begin, let us use the process described in Subsection~\ref{ssec:beta_is_tiny} to establish the above
inequality when $0 < \beta \le \frac{1}{5}$. It follows by part~(b) of
Lemma~\ref{lem:getting_rid_of_pesky_term} that $(1-x)^{-(1-x)} \le x^{\gamma x}$ when
$0 < x \le \frac{1}{5}$ and $\gamma = -\frac{23}{40}$, so that this value may be used for~$\gamma$.
It also follows  by part~(a) of Lemma~\ref{lem:continuing_to_get_rid_of_power_of_q} that $3^x \le x^{\delta x}$
when $0 < x \le \frac{1}{5}$ and $\delta = -\frac{7}{10}$, so that this value can be chosen for~$\delta$. If
$c_4 = \frac{25}{4}$ then
\[
 c - \left( \frac{2q}{q-1} - \frac{\delta}{q-1} - \frac{\gamma q}{q-1} \right) = \frac{31}{12} > 2.5,
\]
so that the process described in Subsection~\ref{ssec:beta_is_tiny} can be
applied with these values. Since $0 < \frac{1}{5} < \frac{1}{3} \le \frac{1}{e}$, it now suffices to note that
$f_1\left(\frac{1}{5}\right) > 0.03$ and $f_1\left(\frac{1}{3}\right) <-0.01$ --- for it then follows by
Lemma~\ref{lem:correctness_of_first_approximation} that the inequality at
line~\eqref{eq:bound_for_4} is satisfied when $0 < \beta \le \frac{1}{5}$.

The process described in Subsection~\ref{ssec:beta_is_moderate} can now be used to establish the above
inequality when $\frac{1}{5} \le \beta \le \frac{8}{25}$, completing the analysis for this case. Since $F_1'\left(\frac{8}{25}\right)
< -1.9$ the function $F_1$ is decreasing over this interval, for every choice of~$\eta$. Since
$\frac{8}{25} < \frac{17}{21} = 1 + \frac{1}{1 - (25/4)}$ and $F_2'\left(\frac{8}{25}\right) > 7.5$, the function~$F_2$
is increasing over this interval for every choice of~$\eta$.

It now suffices to confirm  that if $\eta = \frac{18}{19}$  then $F_1\left(\frac{31}{100}\right) > 0.009$ and
$F_2\left(\frac{1}{5}\right) > 0.04$, so that $F_1$ and~$F_2$ are both non-negative over the interval
$\frac{1}{5} \le \beta \le \frac{31}{100}$.

It then suffices to confirm  that if $\eta = \frac{49}{50}$  then $F_1\left(\frac{8}{25}\right) > 0.02$ and
$F_2\left(\frac{31}{100}\right) > 0.008$, so that $F_1$ and~$F_2$ are both non-negative over the interval
$\frac{31}{100} \le \beta \le \frac{8}{25}$, as needed to complete the proof.

\subsection{Analysis for the Case $q=5$}

In this case it follows by the inequality at line~\eqref{eq:first_bound_for_rho0} that
\[
\rho_0 \le \sum_{1 \le j < k \beta_0} \frac{1}{\beta^{\beta} (1-\beta)^{1-\beta}} f(5)
  = \frac{4^{\beta}}{\beta^{\beta} (1-\beta)^{1-\beta}} \left( {\textstyle{\frac{1}{5}}} + \textstyle{\frac{4}{5}} (5n)^{-c_4 \beta} \right)
  \quad \text{when $q = 5$.}
\]

It will now be shown that
\begin{equation}
\label{eq:bound_for_5}
  \sum_{1 \le j < k \beta_0} \frac{4^{\beta}}{\beta^{\beta} (1-\beta)^{1-\beta}} \left( \frac{1}{5} +
   \frac{4}{5} (5n)^{-c_4 \beta} \right) \le \beta^{\beta}
\end{equation}
when $c_4 = \frac{16}{3}$ and $0 < \beta \le \frac{3}{8}$.

The process described in Subsection~\ref{ssec:beta_is_tiny} will be first be used to establish the
above inequality when $0 < \beta \le \frac{1}{5}$.
Once again, it follows by part~(b) of
Lemma~\ref{lem:getting_rid_of_pesky_term} that $(1-x)^{-(1-x)} \le x^{\gamma x}$ when
$0 < x \le \frac{1}{5}$ and $\gamma = -\frac{23}{40}$, so that this value may be used for~$\gamma$.

It also follows  by part~(b) of Lemma~\ref{lem:continuing_to_get_rid_of_power_of_q} that $4^x \le x^{\delta x}$
when $0 < x \le \frac{1}{5}$ and $\delta = -\frac{9}{10}$, so that this value can be chosen for~$\delta$. If
$c_4 = \frac{16}{3}$ then
\[
 c - 2 + {\textstyle{\frac{\delta}{q-1} + \frac{\gamma q}{q-1}}} ={\textstyle{\frac{907}{480}}} > 1.8,
\]
so that the first process to verify the above relationship, described in Section~\ref{sec:modified_proof}, can be
applied with these values. Since $\frac{1}{5} < \frac{1}{3}$ it suffices to note that
$f_1\left(\frac{1}{5}\right) > 0.04$ and $f_1\left(\frac{1}{3}\right) < -0.003$ --- for it then follows
by Lemma~\ref{lem:correctness_of_first_approximation} that the inequality at
line~\eqref{eq:bound_for_5} is satisfied when $0 < \beta \le \frac{1}{5}$.

The process described in Subsection~\ref{ssec:beta_is_moderate} can now be used to establish
the above inequality when $\frac{1}{5} < \beta \le \frac{3}{8}$, as needed to establish the claimed result.
Since $F_1'\left(\frac{3}{8}\right) < -1.8$, the function $F_1$ is decreasing over this interval, for
every choice of~$\eta$. Since $\frac{3}{8} < \frac{10}{13} = 1 + \frac{1}{1-(16/3)}$ and
$F_2'\left(\frac{3}{8}\right) > 6.6$, the function~$F_2$ is increasing over this interval for every choice
of~$\eta$.

It now suffices to confirm that if $\eta = \frac{11}{12}$ then $F_1\left(\frac{9}{25}\right) > 0.002$ and
$F_2\left(\frac{1}{5}\right) > 0.07$, so that $F_1$ and~$F_2$ are both non-negative over the interval
$\frac{1}{5} \le \beta \le \frac{9}{25}$.

It then suffices to confirm that if $\eta = \frac{19}{20}$ then $F_1\left(\frac{3}{8}\right) > 0.008$ and
$F_2\left(\frac{9}{25}\right) > 0.75$, so that $F_1$ and~$F_2$ are both non-negative over the
interval $\frac{9}{25} \le \beta \le \frac{3}{8}$, as needed to establish the claim when $q=5$.

\subsection{Analysis for the Case $q=7$}

In this case it follows by the inequality at line~\eqref{eq:first_bound_for_rho0} that
\[
\rho_0 \le \sum_{1 \le j < k \beta_0} \frac{1}{\beta^{\beta} (1-\beta)^{1-\beta}} f(7)
  = \frac{6^{\beta}}{\beta^{\beta} (1-\beta)^{1-\beta}} \left( {\textstyle{\frac{1}{7}}} + \textstyle{\frac{6}{7}} (7n)^{-c_4 \beta} \right)
  \quad \text{when $q = 7$.}
\]

It will now be shown that
\begin{equation}
\label{eq:bound_for_7}
  \sum_{1 \le j < k \beta_0} \frac{6^{\beta}}{\beta^{\beta} (1-\beta)^{1-\beta}} \left( \frac{1}{7} +
   \frac{6}{7} (7n)^{-c_4 \beta} \right) \le \beta^{\beta}
\end{equation}
when $c_4 = \frac{13}{3}$ and $0 < \beta \le \frac{6}{13}$.

The process described in Subsection~\ref{ssec:beta_is_tiny} will be first be used to establish the
above inequality when $0 < \beta \le \frac{1}{5}$.
Once again, it follows by part~(b) of
Lemma~\ref{lem:getting_rid_of_pesky_term} that $(1-x)^{-(1-x)} \le x^{\gamma x}$ when
$0 < x \le \frac{1}{5}$ and $\gamma = -\frac{23}{40}$, so that this value may be used for~$\gamma$.

It also follows  by part~(c) of Lemma~\ref{lem:continuing_to_get_rid_of_power_of_q} that $6^x \le x^{\delta x}$
when $0 < x \le \frac{1}{5}$ and $\delta = -\frac{6}{5}$, so that this value can be chosen for~$\delta$. If
$c_4 = \frac{13}{3}$ then
\[
 c - 2 + {\textstyle{\frac{\delta}{q-1} + \frac{\gamma q}{q-1}}} ={\textstyle{\frac{271}{240}}} > 1.1.
\]
so that the first process to verify the above relationship, described in Section~\ref{sec:modified_proof}, can be
applied with these values. Since $\frac{1}{5} < \frac{1}{3}$ it suffices to note that
$f_1\left(\frac{1}{5}\right) > 0.02$ and $f_1\left(\frac{1}{3}\right) < -0.01$ --- for it then follows
by Lemma~\ref{lem:correctness_of_first_approximation} that the inequality at
line~\eqref{eq:bound_for_7} is satisfied when $0 < \beta \le \frac{1}{5}$.

The process described in Subsection~\ref{ssec:beta_is_moderate} can now be used to establish
the above inequality when $\frac{1}{5} < \beta \le \frac{6}{13}$, as needed to establish the claimed result.
Since $F_1'\left(\frac{6}{13}\right) < -1.7$, the function $F_1$ is decreasing over this interval, for
every choice of~$\eta$. Since $\frac{6}{13} < \frac{7}{10} = 1 + \frac{1}{1-(13/3)}$ and
$F_2'\left(\frac{6}{13}\right) > 5.7$, the function~$F_2$ is increasing over this interval for every choice
of~$\eta$.

It now suffices to confirm that if $\eta = \frac{6}{7}$ then $F_1\left(\frac{2}{5}\right) > 0.03$ and
$F_2\left(\frac{1}{5}\right) > 0.01$, so that $F_1$ and~$F_2$ are both non-negative over the interval
$\frac{1}{5} \le \beta \le \frac{2}{5}$.

It then suffices to confirm that if $\eta = \frac{24}{25}$ then $F_1\left(\frac{6}{13}\right) > 0.03$ and
$F_2\left(\frac{2}{5}\right) > 0.14$, so that $F_1$ and~$F_2$ are both non-negative over the
interval $\frac{2}{5} \le \beta \le \frac{6}{13}$, as needed to establish the claim when $q=7$.

\subsection{Analysis for the Case $q=8$}

In this case it follows by the inequality at line~\eqref{eq:first_bound_for_rho0} that
\[
\rho_0 \le \sum_{1 \le j < k \beta_0} \frac{1}{\beta^{\beta} (1-\beta)^{1-\beta}} f(8)
  = \frac{7^{\beta}}{\beta^{\beta} (1-\beta)^{1-\beta}} \left( {\textstyle{\frac{1}{8}}} + \textstyle{\frac{7}{8}} (8n)^{-c_4 \beta} \right)
  \quad \text{when $q = 8$.}
\]

It will now be shown that
\begin{equation}
\label{eq:bound_for_8}
  \sum_{1 \le j < k \beta_0} \frac{7^{\beta}}{\beta^{\beta} (1-\beta)^{1-\beta}} \left( \frac{1}{8} +
   \frac{7}{8} (8n)^{-c_4 \beta} \right) \le \beta^{\beta}
\end{equation}
when $c_4 = 4$ and $0 < x \le \frac{1}{2}$.

The process described in Subsection~\ref{ssec:beta_is_tiny} will be first be used to establish the
above inequality when $0 < \beta \le \frac{1}{5}$. Once again, it follows by part~(a) of
Lemma~\ref{lem:getting_rid_of_pesky_term} that $(1-x)^{-(1-x)} \le x^{\gamma x}$ when
$0 < x \le \frac{1}{5}$ and $\gamma = -\frac{23}{40}$, so that this value may be used for~$\gamma$.

It also follows  by part~(d) of Lemma~\ref{lem:continuing_to_get_rid_of_power_of_q} that $7^x \le x^{\delta x}$
when $0 < x \le \frac{1}{5}$ and $\delta = -\frac{61}{50}$, so that this value can be chosen for~$\delta$. If
$c_4 = \frac{1}{2}$ then
\[
 c - 2 + {\textstyle{\frac{\delta}{q-1} + \frac{\gamma q}{q-1}}} ={\textstyle{\frac{309}{350}}} > 0.88,
\]
so that the first process to verify the above relationship, described in Section~\ref{sec:modified_proof}, can be
applied with these values. Since $\frac{1}{5} < \frac{1}{3}$ it suffices to note that
$f_1\left(\frac{1}{5}\right) > 0.01$ and $f_1\left(\frac{1}{3}\right) < -0.01$ --- for it then follows
by Lemma~\ref{lem:correctness_of_first_approximation} that the inequality at
line~\eqref{eq:bound_for_8} is satisfied when $0 < \beta \le \frac{1}{5}$.

The process described in Subsection~\ref{ssec:beta_is_moderate} can now be used to establish
the above inequality when $\frac{1}{5} < \beta \le \frac{1}{2}$, as needed to establish the claimed result.
Since $F_1'\left(\frac{1}{2}\right) < -1.6$, the function $F_1$ is decreasing over this interval, for
every choice of~$\eta$. Since $\frac{1}{2} < \frac{2}{3} = 1 + \frac{1}{1-4}$ and
$F_2'\left(\frac{1}{2}\right) > 5.4$, the function~$F_2$ is increasing over this interval for every choice
of~$\eta$.

It now suffices to confirm that if $\eta = \frac{5}{6}$ then $F_1\left(\frac{21}{50}\right) > 0.03$ and
$F_2\left(\frac{1}{5}\right) > 0.08$, so that $F_1$ and~$F_2$ are both non-negative over the interval
$\frac{1}{5} \le \beta \le \frac{21}{50}$.

It then suffices to confirm that if $\eta = \frac{24}{25}$ then $F_1\left(\frac{1}{2}\right) > 0.02$ and
$F_2\left(\frac{21}{50}\right) > 0.003$, so that $F_1$ and~$F_2$ are both non-negative over the
interval $\frac{21}{50} \le \beta \le \frac{1}{2}$, as needed to establish the claim when $q=8$.

\subsection{Analysis for the Case $q=9$}

In this case it follows by the inequality at line~\eqref{eq:first_bound_for_rho0} that
\[
\rho_0 \le \sum_{1 \le j < k \beta_0} \frac{1}{\beta^{\beta} (1-\beta)^{1-\beta}} f(9)
  = \frac{8^{\beta}}{\beta^{\beta} (1-\beta)^{1-\beta}} \left( {\textstyle{\frac{1}{9}}} + \textstyle{\frac{8}{9}} (9n)^{-c_4 \beta} \right)
  \quad \text{when $q = 9$.}
\]

It will now be shown that
\begin{equation}
\label{eq:bound_for_9}
  \sum_{1 \le j < k \beta_0} \frac{8^{\beta}}{\beta^{\beta} (1-\beta)^{1-\beta}} \left( \frac{1}{9} +
   \frac{8}{9} (9n)^{-c_4 \beta} \right) \le \beta^{\beta}
\end{equation}
when $c_4 = \frac{11}{3}$ and $0 < x \le \frac{6}{11}$.

The process described in Subsection~\ref{ssec:beta_is_tiny} will be first be used to establish the
above inequality when $0 < \beta \le \frac{9}{50}$. It follows by part~(a) of
Lemma~\ref{lem:continuing_to_get_rid_of_pesky_term} that $(1-x)^{-(1-x)} \le x^{\gamma x}$ when
$0 < x \le \frac{9}{50}$ and $\gamma = -\frac{11}{20}$, so that this value may be used for~$\gamma$.

It also follows  by part~(e) of Lemma~\ref{lem:continuing_to_get_rid_of_power_of_q} that $8^x \le x^{\delta x}$
when $0 < x \le \frac{9}{50}$ and $\delta = -\frac{5}{40}$, so that this value can be chosen for~$\delta$. If
$c_4 = \frac{11}{3}$ then
\[
 c - 2 + {\textstyle{\frac{\delta}{q-1} + \frac{\gamma q}{q-1}}} ={\textstyle{\frac{77}{120}}} > 0.64,
\]
so that the first process to verify the above relationship, described in Section~\ref{sec:modified_proof}, can be
applied with these values. Since $\frac{9}{50} < \frac{1}{4}$ it suffices to note that
$f_1\left(\frac{9}{50}\right) > 0.008$ and $f_1\left(\frac{1}{4}\right) < -0.01$ --- for it then follows
by Lemma~\ref{lem:correctness_of_first_approximation} that the inequality at
line~\eqref{eq:bound_for_9} is satisfied when $0 < \beta \le \frac{9}{50}$.

The process described in Subsection~\ref{ssec:beta_is_moderate} can now be used to establish
the above inequality when $\frac{9}{50} < \beta \le \frac{6}{11}$, as needed to establish the claimed result.
Since $F_1'\left(\frac{6}{11}\right) < -1.5$, the function $F_1$ is decreasing over this interval, for
every choice of~$\eta$. Since $\frac{6}{11} < \frac{5}{8} = 1 + \frac{1}{1-(11/3)}$ and
$F_2'\left(\frac{6}{11}\right) > 5.1$, the function~$F_2$ is increasing over this interval for every choice
of~$\eta$.

It now suffices to confirm that if $\eta = \frac{39}{50}$ then $F_1\left(\frac{21}{50}\right) > 0.03$ and
$F_2\left(\frac{9}{50}\right) > 0.03$, so that $F_1$ and~$F_2$ are both non-negative over the interval
$\frac{9}{50} \le \beta \le \frac{21}{50}$.

It then suffices to confirm that if $\eta = \frac{17}{18}$ then $F_1\left(\frac{53}{100}\right) > 0.01$ and
$F_2\left(\frac{21}{50}\right) > 0.02$, so that $F_1$ and~$F_2$ are both non-negative over the
interval $\frac{21}{50} \le \beta \le \frac{53}{100}$.

Finally, it suffices to confirm that if $\eta = \frac{31}{32}$ then $F_1\left(\frac{6}{11}\right) > 0.01$ and
$F_2\left(\frac{53}{100}\right) > 0.02$, so that $F_1$ and~$F_2$ are both non-negative over the
interval $\frac{53}{100} \le x \le \frac{6}{11}$, as needed to establish the claim when $q=9$.

\subsection{Analysis for the Case $q=11$}

In this case it follows by the inequality at line~\eqref{eq:first_bound_for_rho0} that
\[
\rho_0 \le \sum_{1 \le j < k \beta_0} \frac{1}{\beta^{\beta} (1-\beta)^{1-\beta}} f(11)
  = \frac{10^{\beta}}{\beta^{\beta} (1-\beta)^{1-\beta}} \left( {\textstyle{\frac{1}{11}}}
    + \textstyle{\frac{10}{11}} (11n)^{-c_4 \beta} \right)
  \quad \text{when $q = 11$.}
\]

It will now be shown that
\begin{equation}
\label{eq:bound_for_11}
  \sum_{1 \le j < k \beta_0} \frac{10^{\beta}}{\beta^{\beta} (1-\beta)^{1-\beta}} \left( \frac{1}{11} +
   \frac{10}{11} (4n)^{-c_4 \beta} \right) \le \beta^{\beta}
\end{equation}
when $c_4 = \frac{7}{2}$ and $0 < x \le \frac{4}{7}$.

The process described in Subsection~\ref{ssec:beta_is_tiny} will be first be used to establish the
above inequality when $0 < \beta \le \frac{19}{100}$. It follows by part~(b) of
Lemma~\ref{lem:continuing_to_get_rid_of_pesky_term} that $(1-x)^{-(1-x)} \le x^{\gamma x}$ when
$0 < x \le \frac{19}{100}$ and $\gamma = -\frac{9}{16}$, so that this value may be used for~$\gamma$.

It also follows  by part~(f) of Lemma~\ref{lem:continuing_to_get_rid_of_power_of_q} that $10^x \le x^{\delta x}$
when $0 < x \le \frac{19}{100}$ and $\delta = -\frac{7}{5}$, so that this value can be chosen for~$\delta$. If
$c_4 = \frac{7}{2}$ then
\[
 c - 2 + {\textstyle{\frac{\delta}{q-1} + \frac{\gamma q}{q-1}}} ={\textstyle{\frac{433}{800}}} > 0.54,
\]
so that the first process to verify the above relationship, described in Section~\ref{sec:modified_proof}, can be
applied with these values. Since $\frac{19}{100} < \frac{1}{4}$ it suffices to note that
$f_1\left(\frac{19}{100}\right) > 0.004$ and $f_1\left(\frac{1}{4}\right) < -0.01$ --- for it then follows
by Lemma~\ref{lem:correctness_of_first_approximation} that the inequality at
line~\eqref{eq:bound_for_11} is satisfied when $0 < \beta \le \frac{19}{100}$.

The process described in Subsection~\ref{ssec:beta_is_moderate} can now be used to establish
the above inequality when $\frac{19}{100} < \beta \le \frac{4}{7}$, as needed to establish the claimed result.
Since $F_1'\left(\frac{4}{7}\right) < -1.5$, the function $F_1$ is decreasing over this interval, for
every choice of~$\eta$. Since $\frac{4}{7} < \frac{3}{5} = 1 + \frac{1}{1-(7/2)}$ and
$F_2'\left(\frac{4}{7}\right) > 5.2$, the function~$F_2$ is increasing over this interval for every choice
of~$\eta$.

It now suffices to confirm that if $\eta = \frac{39}{50}$ then $F_1\left(\frac{19}{40}\right) > 0.01$ and
$F_2\left(\frac{19}{100}\right) > 0.04$, so that $F_1$ and~$F_2$ are both non-negative over the interval
$\frac{19}{100} \le \beta \le \frac{19}{40}$.

Finally, it suffices to confirm that if $\eta = \frac{23}{24}$ then $F_1\left(\frac{4}{7}\right) > 0.03$ and
$F_2\left(\frac{19}{40}\right) > 0.002$, so that $F_1$ and~$F_2$ are both non-negative over the
interval $\frac{19}{40} \le x \le \frac{4}{7}$, as needed to establish the claim when $q=11$.

\subsection{Analysis for the Case $q = 13$}

In this case it follows by the inequality at line~\eqref{eq:first_bound_for_rho0} that
\[
\rho_0 \le \sum_{1 \le j < k \beta_0} \frac{1}{\beta^{\beta} (1-\beta)^{1-\beta}} f(13)
  = \frac{12^{\beta}}{\beta^{\beta} (1-\beta)^{1-\beta}} \left( {\textstyle{\frac{1}{13}}}
    + \textstyle{\frac{12}{13}} (13n)^{-c_4 \beta} \right)
  \quad \text{when $q = 13$.}
\]

It will now be shown that
\begin{equation}
\label{eq:bound_for_13}
  \sum_{1 \le j < k \beta_0} \frac{12^{\beta}}{\beta^{\beta} (1-\beta)^{1-\beta}} \left( \frac{1}{13} +
   \frac{12}{13} (13n)^{-c_4 \beta} \right) \le \beta^{\beta}
\end{equation}
when $c_4 = \frac{10}{3}$ and $0 < x \le \frac{3}{5}$.

The process described in Subsection~\ref{ssec:beta_is_tiny} will be first be used to establish the
above inequality when $0 < \beta \le \frac{9}{500}$. It follows by part~(a) of
Lemma~\ref{lem:continuing_to_get_rid_of_pesky_term} that $(1-x)^{-(1-x)} \le x^{\gamma x}$ when
$0 < x \le \frac{9}{50}$ and $\gamma = -\frac{11}{20}$, so that this value may be used for~$\gamma$.

It also follows  by part~(g) of Lemma~\ref{lem:continuing_to_get_rid_of_power_of_q} that $12^x \le x^{\delta x}$
when $0 < x \le \frac{9}{50}$ and $\delta = -\frac{3}{2}$, so that this value can be chosen for~$\delta$. If
$c_4 = \frac{10}{3}$ then
\[
 c - 2 + {\textstyle{\frac{\delta}{q-1} + \frac{\gamma q}{q-1}}} ={\textstyle{\frac{107}{240}}} > 0.44,
\]
so that the first process to verify the above relationship, described in Section~\ref{sec:modified_proof}, can be
applied with these values. Since $\frac{9}{50} < \frac{1}{5}$ it suffices to note that
$f_1\left(\frac{9}{50}\right) > 0.04$ and $f_1\left(\frac{1}{5}\right) < -0.0004$ --- for it then follows
by Lemma~\ref{lem:correctness_of_first_approximation} that the inequality at
line~\eqref{eq:bound_for_13} is satisfied when $0 < \beta \le \frac{9}{50}$.

The process described in Subsection~\ref{ssec:beta_is_moderate} can now be used to establish
the above inequality when $\frac{9}{50} < \beta \le \frac{3}{5}$, as needed to establish the claimed result.
Since $F_1'\left(\frac{3}{5}\right) < -1.5$, the function $F_1$ is decreasing over this interval, for
every choice of~$\eta$. Since $\frac{3}{5} > \frac{4}{7} = 1 + \frac{1}{1-(10/3)}$ and
$F_2'\left(\frac{4}{7}\right) > 5.3$, the function~$F_2$ is increasing over this interval for every choice
of~$\eta$.

It now suffices to confirm that if $\eta = \frac{37}{50}$ then $F_1\left(\frac{24}{50}\right) > 0.02$ and
$F_2\left(\frac{24}{50}\right) > 0.07$, so that $F_1$ and~$F_2$ are both non-negative over the interval
$\frac{9}{50} \le \beta \le \frac{24}{50}$.

It next suffices to confirm that if $\eta = \frac{23}{25}$ then $F_1\left(\frac{3}{5}\right) > 0.01$ and
$F_2\left(\frac{9}{50}\right) > 0.5$, so that $F_1$ and~$F_2$ are both non-negative over the interval
$\frac{24}{50} \le \beta \le \frac{3}{5}$, as needed to establish the claim when $q=13$.

\subsection{Analysis for the Case $16 \le q \le 19$}

Suppose, first, that $q=16$.
In this case it follows by the inequality at line~\eqref{eq:first_bound_for_rho0} that
\[
\rho_0 \le \sum_{1 \le j < k \beta_0} \frac{1}{\beta^{\beta} (1-\beta)^{1-\beta}} f(16)
  = \frac{15^{\beta}}{\beta^{\beta} (1-\beta)^{1-\beta}} \left( {\textstyle{\frac{1}{16}}}
     + \textstyle{\frac{15}{4}} (16n)^{-c_4 \beta} \right)
  \quad \text{when $q = 16$.}
\]

It will now be shown that
\begin{equation}
\label{eq:bound_for_16}
  \sum_{1 \le j < k \beta_0} \frac{15^{\beta}}{\beta^{\beta} (1-\beta)^{1-\beta}} \left( \frac{1}{16} +
   \frac{15}{16} (16n)^{-c_4 \beta} \right) \le \beta^{\beta}
\end{equation}
when $c_4 = 3$ and $0 < x \le \beta_0 = \frac{2}{3}$. Since $c_4 = 3 \ge \frac{13}{6}$ and
$\beta_0 = \frac{2}{3} \le \frac{12}{13}$, it follows by Lemma~\ref{lem:f_is_decreasing} that
\[
 \rho_0 \le \sum_{1 \le j \le k \beta_0} \frac{1}{\beta^{\beta} (1-\beta)^{1-\beta}} f(q) \le \beta^{\beta}
\]
for $c_4 = 3$ and $0 < x \le \beta_0 = \frac{2}{3}$ when $17 \le q \le 19$ as well.

The process described in Subsection~\ref{ssec:beta_is_tiny} will be first be used to establish the
above inequality when $0 < \beta \le \frac{7}{50}$. It follows by part~(c) of
Lemma~\ref{lem:continuing_to_get_rid_of_pesky_term} that $(1-x)^{-(1-x)} \le x^{\gamma x}$ when
$0 < x \le \frac{7}{50}$ and $\gamma = -\frac{19}{40}$, so that this value may be used for~$\gamma$.

It also follows  by part~(h) of Lemma~\ref{lem:continuing_to_get_rid_of_power_of_q} that $15^x \le x^{\delta x}$
when $0 < x \le \frac{7}{50}$ and $\delta = -\frac{7}{5}$, so that this value can be chosen for~$\delta$. If
$c_4 = 3$ then
\[
 c - 2 + {\textstyle{\frac{\delta}{q-1} + \frac{\gamma q}{q-1}}} ={\textstyle{\frac{4}{15}}} > 0.26,
\]
so that the first process to verify the above relationship, described in Section~\ref{sec:modified_proof}, can be
applied with these values. Since $\frac{7}{50} < \frac{1}{5}$ it suffices to note that
$f_1\left(\frac{7}{50}\right) > 0.005$ and $f_1\left(\frac{1}{5}\right) < -0.006$ --- for it then follows
by Lemma~\ref{lem:correctness_of_first_approximation} that the inequality at
line~\eqref{eq:bound_for_16} is satisfied when $0 < \beta \le \frac{7}{50}$.

The process described in Subsection~\ref{ssec:beta_is_moderate} can now be used to establish
the above inequality when $\frac{7}{50} < \beta \le \frac{2}{3}$, as needed to establish the claimed result.
Since $F_1'\left(\frac{2}{3}\right) < -1.4$, the function $F_1$ is decreasing over this interval, for
every choice of~$\eta$. Since $\frac{2}{3} > \frac{1}{2} = 1 + \frac{1}{1- 3}$ and
$F_2'\left(\frac{1}{2}\right) > 4.9$, the function~$F_2$ is increasing over this interval for every choice
of~$\eta$.

It now suffices to confirm that if $\eta = \frac{3}{5}$ then $F_1\left(\frac{17}{40}\right) > 0.06$ and
$F_2\left(\frac{7}{50}\right) > 0.07$, so that $F_1$ and~$F_2$ are both non-negative over the interval
$\frac{7}{50} \le \beta \le \frac{17}{40}$.

It next suffices to confirm that if $\eta = \frac{4}{5}$ then $F_1\left(\frac{14}{25}\right) > 0.02$ and
$F_2\left(\frac{17}{40}\right) > 0.8$, so that $F_1$ and~$F_2$ are both non-negative over the interval
$\frac{17}{40} \le \beta \le \frac{14}{25}$.

Finally, it suffices to confirm that if $\eta = \frac{19}{20}$ then $F_1\left(\frac{2}{3}\right) > 0.009$ and
$F_2\left(\frac{14}{25}\right) > 0.17$, so that $F_1$ and~$F_2$ are both non-negative over the interval
$\frac{14}{25} \le \beta \le \frac{2}{3}$, as needed to establish the claim when $16 \le q \le 19$.

\subsection{Analysis for the Case $23 \le q \le 29$}

Suppose, first, that $q=23$.
In this case it follows by the inequality at line~\eqref{eq:first_bound_for_rho0} that
\[
\rho_0 \le \sum_{1 \le j < k \beta_0} \frac{1}{\beta^{\beta} (1-\beta)^{1-\beta}} f(23)
  = \frac{22^{\beta}}{\beta^{\beta} (1-\beta)^{1-\beta}} \left( {\textstyle{\frac{1}{23}}}
    + \textstyle{\frac{22}{23}} (23n)^{-c_4 \beta} \right)
  \quad \text{when $q = 23$.}
\]

It will now be shown that
\begin{equation}
\label{eq:bound_for_23}
  \sum_{1 \le j < k \beta_0} \frac{22^{\beta}}{\beta^{\beta} (1-\beta)^{1-\beta}} \left( \frac{1}{23} +
   \frac{22}{23} (23n)^{-c_4 \beta} \right) \le \beta^{\beta}
\end{equation}
when $c_4 = \frac{8}{3}$ and $0 < x \le \beta_0 = \frac{3}{4}$. Since $c_4 = \frac{8}{3} \ge \frac{13}{6}$ and
$\beta_0 = \frac{3}{4} \le \frac{12}{13}$, it follows by Lemma~\ref{lem:f_is_decreasing} that
\[
 \rho_0 \le \sum_{1 \le j \le k \beta_0} \frac{1}{\beta^{\beta} (1-\beta)^{1-\beta}} f(q) \le \beta^{\beta}
\]
for $c_4 = \frac{8}{3}$ and $0 < x \le \beta_0 = \frac{3}{4}$ when $25 \le q \le 29$ as well.

The process described in Subsection~\ref{ssec:beta_is_tiny} will be first be used to establish the
above inequality when $0 < \beta \le \frac{2}{25}$. It follows by part~(d) of
Lemma~\ref{lem:continuing_to_get_rid_of_pesky_term} that $(1-x)^{-(1-x)} \le x^{\gamma x}$ when
$0 < x \le \frac{2}{25}$ and $\gamma = -\frac{2}{5}$, so that this value may be used for~$\gamma$.

It also follows  by part~(i) of Lemma~\ref{lem:continuing_to_get_rid_of_power_of_q} that $22^x \le x^{\delta x}$
when $0 < x \le \frac{2}{25}$ and $\delta = -\frac{5}{4}$, so that this value can be chosen for~$\delta$. If
$c_4 = \frac{8}{3}$ then
\[
 c - 2 + {\textstyle{\frac{\delta}{q-1} + \frac{\gamma q}{q-1}}} ={\textstyle{\frac{133}{1320}}} > 0.1,
\]
so that the first process to verify the above relationship, described in Section~\ref{sec:modified_proof}, can be
applied with these values. Since $\frac{2}{25} < \frac{1}{10}$ it suffices to note that
$f_1\left(\frac{2}{25}\right) > 0.002$ and $f_1\left(\frac{1}{10}\right) < -0.0002$ --- for it then follows
by Lemma~\ref{lem:correctness_of_first_approximation} that the inequality at
line~\eqref{eq:bound_for_23} is satisfied when $0 < \beta \le \frac{2}{25}$.

The process described in Subsection~\ref{ssec:beta_is_moderate} can now be used to establish
the above inequality when $\frac{2}{25} < \beta \le \frac{3}{4}$, as needed to establish the claimed result.
Since $F_1'\left(\frac{3}{4}\right) < -1.2$, the function $F_1$ is decreasing over this interval, for
every choice of~$\eta$. Since $\frac{3}{4} > \frac{2}{5} = 1 + \frac{1}{1- (8/3)}$ and
$F_2'\left(\frac{2}{5}\right) > 4.7$, the function~$F_2$ is increasing over this interval for every choice
of~$\eta$.

It now suffices to confirm that if $\eta = \frac{2}{5}$ then $F_1\left(\frac{7}{20}\right) > 0.1$ and
$F_2\left(\frac{2}{25}\right) > 0.01$, so that $F_1$ and~$F_2$ are both non-negative over the interval
$\frac{2}{25} \le \beta \le \frac{7}{20}$.

It next suffices to confirm that if $\eta = \frac{21}{25}$ then $F_1\left(\frac{13}{20}\right) > 0.02$ and
$F_2\left(\frac{7}{20}\right) > 0.02$, so that $F_1$ and~$F_2$ are both non-negative over the interval
$\frac{7}{20} \le \beta \le \frac{13}{20}$.

It next suffices to confirm that if $\eta = \frac{24}{25}$ then $F_1\left(\frac{37}{50}\right) > 0.01$ and
$F_2\left(\frac{13}{20}\right) > 0.07$, so that $F_1$ and~$F_2$ are both non-negative over the interval
$\frac{13}{20} \le \beta \le \frac{37}{50}$.

Finally, it suffices to confirm that if $\eta = \frac{39}{40}$ then $F_1\left(\frac{3}{4}\right) > 0.01$ and
$F_2\left(\frac{37}{50}\right) > 0.05$, so that $F_1$ and~$F_2$ are both non-negative over the interval
$\frac{37}{50} \le \beta \le \frac{3}{4}$, as needed to establish the claim when $23 \le q \le 29$.

\subsection{Analysis for the Case $31 \le q \le 43$}

Suppose, first, that $q=31$.
In this case it follows by the inequality at line~\eqref{eq:first_bound_for_rho0} that
\[
\rho_0 \le \sum_{1 \le j < k \beta_0} \frac{1}{\beta^{\beta} (1-\beta)^{1-\beta}} f(31)
  = \frac{30^{\beta}}{\beta^{\beta} (1-\beta)^{1-\beta}} \left( {\textstyle{\frac{1}{31}}}
    + \textstyle{\frac{30}{31}} (31n)^{-c_4 \beta} \right)
  \quad \text{when $q = 31$.}
\]

It will now be shown that
\begin{equation}
\label{eq:bound_for_31}
  \sum_{1 \le j < k \beta_0} \frac{30^{\beta}}{\beta^{\beta} (1-\beta)^{1-\beta}} \left( \frac{1}{31} +
   \frac{30}{31} (31n)^{-c_4 \beta} \right) \le \beta^{\beta}
\end{equation}
when $c_4 = \frac{5}{2}$ and $0 < x \le \beta_0 = \frac{4}{5}$. Since $c_4 = \frac{5}{2} \ge \frac{13}{6}$ and
$\beta_0 = \frac{4}{5} \le \frac{12}{13}$, it follows by Lemma~\ref{lem:f_is_decreasing} that
\[
 \rho_0 \le \sum_{1 \le j \le k \beta_0} \frac{1}{\beta^{\beta} (1-\beta)^{1-\beta}} f(q) \le \beta^{\beta}
\]
for $c_4 = \frac{5}{2}$ and $0 < x \le \beta_0 = \frac{4}{5}$ when $32 \le q \le 43$ as well.

The process described in Subsection~\ref{ssec:beta_is_tiny} will be first be used to establish the
above inequality when $0 < \beta \le \frac{1}{20}$. It follows by part~(e) of
Lemma~\ref{lem:continuing_to_get_rid_of_pesky_term} that $(1-x)^{-(1-x)} \le x^{\gamma x}$ when
$0 < x \le \frac{1}{20}$ and $\gamma = -\frac{1}{3}$, so that this value may be used for~$\gamma$.

It also follows  by part~(j) of Lemma~\ref{lem:continuing_to_get_rid_of_power_of_q} that $30^x \le x^{\delta x}$
when $0 < x \le \frac{1}{20}$ and $\delta = -\frac{8}{7}$, so that this value can be chosen for~$\delta$. If
$c_4 = \frac{5}{2}$ then
\[
 c - 2 + {\textstyle{\frac{\delta}{q-1} + \frac{\gamma q}{q-1}}} ={\textstyle{\frac{16}{315}}} > 0.05,
\]
so that the first process to verify the above relationship, described in Section~\ref{sec:modified_proof}, can be
applied with these values. Since $\frac{1}{20} < \frac{1}{12}$ it suffices to note that
$f_1\left(\frac{1}{20}\right) > 0.001$ and $f_1\left(\frac{1}{12}\right) < -0.0009$ --- for it then follows
by Lemma~\ref{lem:correctness_of_first_approximation} that the inequality at
line~\eqref{eq:bound_for_31} is satisfied when $0 < \beta \le \frac{1}{20}$.

The process described in Subsection~\ref{ssec:beta_is_moderate} can now be used to establish
the above inequality when $\frac{1}{20} < \beta \le \frac{4}{5}$, as needed to establish the claimed result.
Since $F_1'\left(\frac{4}{5}\right) < -1.2$, the function $F_1$ is decreasing over this interval, for
every choice of~$\eta$. Since $\frac{4}{5} > \frac{1}{3} = 1 + \frac{1}{1- (5/2)}$ and
$F_2'\left(\frac{1}{3}\right) > 4.6$, the function~$F_2$ is increasing over this interval for every choice
of~$\eta$.

It now suffices to confirm that if $\eta = \frac{1}{5}$ then $F_1\left(\frac{1}{4}\right) > 0.06$ and
$F_2\left(\frac{1}{20}\right) > 0.09$, so that $F_1$ and~$F_2$ are both non-negative over the interval
$\frac{1}{20} \le \beta \le \frac{1}{4}$.

It next suffices to confirm that if $\eta = \frac{7}{10}$ then $F_1\left(\frac{3}{5}\right) > 0.05$ and
$F_2\left(\frac{1}{4}\right) > 0.08$, so that $F_1$ and~$F_2$ are both non-negative over the interval
$\frac{1}{4} \le \beta \le \frac{3}{5}$.

It next suffices to confirm that if $\eta = \frac{11}{12}$ then $F_1\left(\frac{3}{4}\right) > 0.01$ and
$F_2\left(\frac{3}{5}\right) > 0.4$, so that $F_1$ and~$F_2$ are both non-negative over the interval
$\frac{3}{5} \le \beta \le \frac{3}{4}$.

Finally, it suffices to confirm that if $\eta = \frac{29}{30}$ then $F_1\left(\frac{4}{5}\right) > 0.0002$ and
$F_2\left(\frac{3}{4}\right) > 0.2$, so that $F_1$ and~$F_2$ are both non-negative over the interval
$\frac{3}{4} \le \beta \le \frac{4}{5}$, as needed to establish the claim when $31 \le q \le 43$.

\subsection{Analysis for the Case $47 \le q \le 59$}

Suppose, first, that $q=47$.
In this case it follows by the inequality at line~\eqref{eq:first_bound_for_rho0} that
\[
\rho_0 \le \sum_{1 \le j < k \beta_0} \frac{1}{\beta^{\beta} (1-\beta)^{1-\beta}} f(47)
  = \frac{46^{\beta}}{\beta^{\beta} (1-\beta)^{1-\beta}} \left( {\textstyle{\frac{1}{47}}}
    + \textstyle{\frac{46}{47}} (47n)^{-c_4 \beta} \right)
  \quad \text{when $q = 47$.}
\]

It will now be shown that
\begin{equation}
\label{eq:bound_for_47}
  \sum_{1 \le j < k \beta_0} \frac{46^{\beta}}{\beta^{\beta} (1-\beta)^{1-\beta}} \left( \frac{1}{47} +
   \frac{46}{47} (47n)^{-c_4 \beta} \right) \le \beta^{\beta}
\end{equation}
when $c_4 = \frac{7}{3}$ and $0 < x \le \beta_0 = \frac{6}{7}$. Since $c_4 = \frac{7}{3} \ge \frac{13}{6}$ and
$\beta_0 = \frac{6}{7} \le \frac{12}{13}$, it follows by Lemma~\ref{lem:f_is_decreasing} that
\[
 \rho_0 \le \sum_{1 \le j \le k \beta_0} \frac{1}{\beta^{\beta} (1-\beta)^{1-\beta}} f(q) \le \beta^{\beta}
\]
for $c_4 = \frac{7}{3}$ and $0 < x \le \beta_0 = \frac{6}{7}$ when $49 \le q \le 59$ as well.

The process described in Subsection~\ref{ssec:beta_is_tiny} will be first be used to establish the
above inequality when $0 < \beta \le \frac{1}{53}$. It follows by part~(f) of
Lemma~\ref{lem:continuing_to_get_rid_of_pesky_term} that $(1-x)^{-(1-x)} \le x^{\gamma x}$ when
$0 < x \le \frac{1}{53}$ and $\gamma = -\frac{1}{4}$, so that this value may be used for~$\gamma$.

It also follows  by part~(k) of Lemma~\ref{lem:continuing_to_get_rid_of_power_of_q} that $46^x \le x^{\delta x}$
when $0 < x \le \frac{1}{53}$ and $\delta = -\frac{49}{50}$, so that this value can be chosen for~$\delta$. If
$c_4 = \frac{7}{3}$ then
\[
 c - 2 + {\textstyle{\frac{\delta}{q-1} + \frac{\gamma q}{q-1}}} ={\textstyle{\frac{181}{13800}}} > 0.013,
\]
so that the first process to verify the above relationship, described in Section~\ref{sec:modified_proof}, can be
applied with these values. Since $\frac{1}{53} < \frac{1}{35}$ it suffices to note that
$f_1\left(\frac{1}{53}\right) > 0.0002$ and $f_1\left(\frac{1}{35}\right) < -0.00001$ --- for it then follows
by Lemma~\ref{lem:correctness_of_first_approximation} that the inequality at
line~\eqref{eq:bound_for_47} is satisfied when $0 < \beta \le \frac{1}{53}$.

The process described in Subsection~\ref{ssec:beta_is_moderate} can now be used to establish
the above inequality when $\frac{1}{53} < \beta \le \frac{6}{7}$, as needed to establish the claimed result.
Since $F_1'\left(\frac{6}{7}\right) < -1.1$, the function $F_1$ is decreasing over this interval, for
every choice of~$\eta$. Since $\frac{6}{7} > \frac{1}{4} = 1 + \frac{1}{1- (7/3)}$ and
$F_2'\left(\frac{1}{4}\right) > 4.8$, the function~$F_2$ is increasing over this interval for every choice
of~$\eta$.

It now suffices to confirm that if $\eta = \frac{1}{9}$ then $F_1\left(\frac{1}{5}\right) > 0.06$ and
$F_2\left(\frac{1}{53}\right) > 0.007$, so that $F_1$ and~$F_2$ are both non-negative over the interval
$\frac{1}{53} \le \beta \le \frac{1}{5}$.

It next suffices to confirm that if $\eta = \frac{3}{5}$ then $F_1\left(\frac{3}{5}\right) > 0.06$ and
$F_2\left(\frac{1}{5}\right) > 0.06$, so that $F_1$ and~$F_2$ are both non-negative over the interval
$\frac{1}{5} \le \beta \le \frac{3}{5}$.

It next suffices to confirm that if $\eta = \frac{16}{17}$ then $F_1\left(\frac{4}{5}\right) > 0.04$ and
$F_2\left(\frac{3}{5}\right) > 0.01$, so that $F_1$ and~$F_2$ are both non-negative over the interval
$\frac{3}{5} \le \beta \le \frac{4}{5}$.

Finally, it suffices to confirm that if $\eta = \frac{44}{45}$ then $F_1\left(\frac{6}{7}\right) > 0.003$ and
$F_2\left(\frac{4}{5}\right) > 0.07$, so that $F_1$ and~$F_2$ are both non-negative over the interval
$\frac{4}{5} \le \beta \le \frac{6}{7}$, as needed to establish the claim when $47 \le q \le 59$.

\subsection{Analysis for the Case $61 \le q \le 71$}

Suppose, first, that $q=61$.
In this case it follows by the inequality at line~\eqref{eq:first_bound_for_rho0} that
\[
\rho_0 \le \sum_{1 \le j < k \beta_0} \frac{1}{\beta^{\beta} (1-\beta)^{1-\beta}} f(61)
  = \frac{60^{\beta}}{\beta^{\beta} (1-\beta)^{1-\beta}} \left( {\textstyle{\frac{1}{61}}}
    + \textstyle{\frac{60}{61}} (61n)^{-c_4 \beta} \right)
  \quad \text{when $q = 47$.}
\]

It will now be shown that
\begin{equation}
\label{eq:bound_for_61}
  \sum_{1 \le j < k \beta_0} \frac{60^{\beta}}{\beta^{\beta} (1-\beta)^{1-\beta}} \left( \frac{1}{61} +
   \frac{60}{61} (61n)^{-c_4 \beta} \right) \le \beta^{\beta}
\end{equation}
when $c_4 = \frac{9}{4}$ and $0 < x \le \beta_0 = \frac{8}{9}$. Since $c_4 = \frac{9}{4} \ge \frac{13}{6}$ and
$\beta_0 = \frac{8}{9} \le \frac{12}{13}$, it follows by Lemma~\ref{lem:f_is_decreasing} that
\[
 \rho_0 \le \sum_{1 \le j \le k \beta_0} \frac{1}{\beta^{\beta} (1-\beta)^{1-\beta}} f(q) \le \beta^{\beta}
\]
for $c_4 = \frac{9}{4}$ and $0 < x \le \beta_0 = \frac{8}{9}$ when $64 \le q \le 71$ as well.

The process described in Subsection~\ref{ssec:beta_is_tiny} will be first be used to establish the
above inequality when $0 < \beta \le \frac{1}{200}$. It follows by part~(g) of
Lemma~\ref{lem:continuing_to_get_rid_of_pesky_term} that $(1-x)^{-(1-x)} \le x^{\gamma x}$ when
$0 < x \le \frac{1}{200}$ and $\gamma = -\frac{19}{50}$, so that this value may be used for~$\gamma$.

It also follows  by part~(l) of Lemma~\ref{lem:continuing_to_get_rid_of_power_of_q} that $60^x \le x^{\delta x}$
when $0 < x \le \frac{1}{200}$ and $\delta = -\frac{4}{5}$, so that this value can be chosen for~$\delta$. If
$c_4 = \frac{9}{4}$ then
\[
 c - 2 + {\textstyle{\frac{\delta}{q-1} + \frac{\gamma q}{q-1}}} ={\textstyle{\frac{61}{6000}}} > 0.01,
\]
so that the first process to verify the above relationship, described in Section~\ref{sec:modified_proof}, can be
applied with these values. Since $\frac{1}{200} < \frac{1}{25}$ it suffices to note that
$f_1\left(\frac{1}{200}\right) > 0.0002$ and $f_1\left(\frac{1}{25}\right) < -0.0001$ --- for it then follows
by Lemma~\ref{lem:correctness_of_first_approximation} that the inequality at
line~\eqref{eq:bound_for_61} is satisfied when $0 < \beta \le \frac{1}{200}$.

The process described in Subsection~\ref{ssec:beta_is_moderate} can now be used to establish
the above inequality when $\frac{1}{200} < \beta \le \frac{8}{9}$, as needed to establish the claimed result.
Since $F_1'\left(\frac{8}{9}\right) < -1.1$, the function $F_1$ is decreasing over this interval, for
every choice of~$\eta$. Since $\frac{8}{9} > \frac{1}{5} = 1 + \frac{1}{1- (9/4)}$ and
$F_2'\left(\frac{1}{5}\right) > 4.5$, the function~$F_2$ is increasing over this interval for every choice
of~$\eta$.

It now suffices to confirm that if $\eta = \frac{1}{25}$ then $F_1\left(\frac{2}{25}\right) > 0.08$ and
$F_2\left(\frac{1}{200}\right) > 0.03$, so that $F_1$ and~$F_2$ are both non-negative over the interval
$\frac{1}{200} \le \beta \le \frac{2}{25}$.

It next suffices to confirm that if $\eta = \frac{1}{4}$ then $F_1\left(\frac{2}{5}\right) > 0.04$ and
$F_2\left(\frac{2}{25}\right) > 0.1$, so that $F_1$ and~$F_2$ are both non-negative over the interval
$\frac{2}{25} \le \beta \le \frac{2}{5}$.

It next suffices to confirm that if $\eta = \frac{5}{6}$ then $F_1\left(\frac{39}{50}\right) > 0.01$ and
$F_2\left(\frac{2}{5}\right) > 0.07$, so that $F_1$ and~$F_2$ are both non-negative over the interval
$\frac{2}{5} \le \beta \le \frac{39}{50}$.

It next suffices to confirm that if $\eta = \frac{41}{42}$ then $F_1\left(\frac{22}{25}\right) > 0.004$ and
$F_2\left(\frac{39}{50}\right) > 0.01$, so that $F_1$ and~$F_2$ are both non-negative over the interval
$\frac{39}{50} \le \beta \le \frac{22}{25}$.

Finally, it suffices to confirm that if $\eta = \frac{74}{75}$ then $F_1\left(\frac{8}{9}\right) > 0.004$ and
$F_2\left(\frac{22}{25}\right) > 0.009$, so that $F_1$ and~$F_2$ are both non-negative over the interval
$\frac{22}{25} \le \beta \le \frac{8}{9}$, as needed to establish the claim when $61 \le q \le 71$.

\subsection{Analysis for the Case $73 \le q \le 83$}

Suppose, first, that $q=73$.
In this case it follows by the inequality at line~\eqref{eq:first_bound_for_rho0} that
\[
\rho_0 \le \sum_{1 \le j < k \beta_0} \frac{1}{\beta^{\beta} (1-\beta)^{1-\beta}} f(73)
  = \frac{72^{\beta}}{\beta^{\beta} (1-\beta)^{1-\beta}} \left( {\textstyle{\frac{1}{73}}}
    + \textstyle{\frac{72}{73}} (73n)^{-c_4 \beta} \right)
  \quad \text{when $q = 73$.}
\]

It will now be shown that
\begin{equation}
\label{eq:bound_for_73}
  \sum_{1 \le j < k \beta_0} \frac{72^{\beta}}{\beta^{\beta} (1-\beta)^{1-\beta}} \left( \frac{1}{73} +
   \frac{72}{73} (73n)^{-c_4 \beta} \right) \le \beta^{\beta}
\end{equation}
when $c_4 = \frac{11}{5}$ and $0 < x \le \beta_0 = \frac{10}{11}$. Since $c_4 = \frac{11}{5} \ge \frac{13}{6}$ and
$\beta_0 = \frac{10}{11} \le \frac{12}{13}$, it follows by Lemma~\ref{lem:f_is_decreasing} that
\[
 \rho_0 \le \sum_{1 \le j \le k \beta_0} \frac{1}{\beta^{\beta} (1-\beta)^{1-\beta}} f(q) \le \beta^{\beta}
\]
for $c_4 = \frac{11}{5}$ and $0 < x \le \beta_0 = \frac{10}{11}$ when $79 \le q \le 83$ as well.

The process described in Subsection~\ref{ssec:beta_is_tiny} will be first be used to establish the
above inequality when $0 < \beta \le \frac{1}{750}$. It follows by part~(h) of
Lemma~\ref{lem:continuing_to_get_rid_of_pesky_term} that $(1-x)^{-(1-x)} \le x^{\gamma x}$ when
$0 < x \le \frac{1}{750}$ and $\gamma = -\frac{23}{150}$, so that this value may be used for~$\gamma$.

It also follows  by part~(m) of Lemma~\ref{lem:continuing_to_get_rid_of_power_of_q} that $72^x \le x^{\delta x}$
when $0 < x \le \frac{1}{750}$ and $\delta = -\frac{7}{10}$, so that this value can be chosen for~$\delta$. If
$c_4 = \frac{11}{5}$ then
\[
 c - 2 + {\textstyle{\frac{\delta}{q-1} + \frac{\gamma q}{q-1}}} ={\textstyle{\frac{19}{2700}}} > 0.07,
\]
so that the first process to verify the above relationship, described in Section~\ref{sec:modified_proof}, can be
applied with these values. Since $\frac{1}{750} < \frac{1}{30}$ it suffices to note that
$f_1\left(\frac{1}{750}\right) > 0.00005$ and $f_1\left(\frac{1}{30}\right) < -0.00002$ --- for it then follows
by Lemma~\ref{lem:correctness_of_first_approximation} that the inequality at
line~\eqref{eq:bound_for_73} is satisfied when $0 < \beta \le \frac{1}{750}$.

The process described in Subsection~\ref{ssec:beta_is_moderate} can now be used to establish
the above inequality when $\frac{1}{750} < \beta \le \frac{10}{11}$, as needed to establish the claimed result.
Since $F_1'\left(\frac{10}{11}\right) < -1$, the function $F_1$ is decreasing over this interval, for
every choice of~$\eta$. Since $\frac{10}{11} > \frac{1}{6} = 1 + \frac{1}{1- (11/5)}$ and
$F_2'\left(\frac{1}{6}\right) > 4.5$, the function~$F_2$ is increasing over this interval for every choice
of~$\eta$.

It now suffices to confirm that if $\eta = \frac{1}{50}$ then $F_1\left(\frac{1}{40}\right) > 0.06$ and
$F_2\left(\frac{1}{750}\right) > 0.0009$, so that $F_1$ and~$F_2$ are both non-negative over the interval
$\frac{1}{750} \le \beta \le \frac{1}{40}$.

It next suffices to confirm that if $\eta = \frac{1}{8}$ then $F_1\left(\frac{3}{11}\right) > 0.1$ and
$F_2\left(\frac{1}{40}\right) > 0.03$, so that $F_1$ and~$F_2$ are both non-negative over the interval
$\frac{1}{40} \le \beta \le \frac{3}{11}$.

It next suffices to confirm that if $\eta = \frac{2}{3}$ then $F_1\left(\frac{7}{10}\right) > 0.03$ and
$F_2\left(\frac{3}{11}\right) > 0.1$, so that $F_1$ and~$F_2$ are both non-negative over the interval
$\frac{2}{11} \le \beta \le \frac{7}{10}$.

It next suffices to confirm that if $\eta = \frac{26}{27}$ then $F_1\left(\frac{7}{8}\right) > 0.01$ and
$F_2\left(\frac{7}{10}\right) > 0.02$, so that $F_1$ and~$F_2$ are both non-negative over the interval
$\frac{7}{10} \le \beta \le \frac{7}{8}$.

It next suffices to confirm that if $\eta = \frac{72}{73}$ then $F_1\left(\frac{181}{200}\right) > 0.001$ and
$F_2\left(\frac{7}{8}\right) > 0.003$, so that $F_1$ and~$F_2$ are both non-negative over the interval
$\frac{7}{8} \le \beta \le \frac{181}{200}$.

It next suffices to confirm that if $\eta = \frac{87}{88}$ then $F_1\left(\frac{227}{250}\right) > 0.001$ and
$F_2\left(\frac{181}{200}\right) > 0.002$, so that $F_1$ and~$F_2$ are both non-negative over the interval
$\frac{181}{200} \le \beta \le \frac{227}{250}$.

It next suffices to confirm that if $\eta = \frac{88}{89}$ then $F_1\left(\frac{909}{1000}\right) > 0.00009$ and
$F_2\left(\frac{227}{250}\right) > 0.01$, so that $F_1$ and~$F_2$ are both non-negative over the interval
$\frac{227}{250} \le \beta \le \frac{909}{1000}$.

Finally, it suffices to confirm that if $\eta = \frac{89}{90}$ then $F_1\left(\frac{10}{11}\right) > 0.001$ and
$F_2\left(\frac{909}{1000}\right) > 0.005$, so that $F_1$ and~$F_2$ are both non-negative over the interval
$\frac{909}{1000} \le \beta \le \frac{10}{11}$, as needed to establish the claim when $73 \le q \le 83$.

\subsection{Analysis for the Case $q \ge 89$}

Suppose, first, that $q=89$.
In this case it follows by the inequality at line~\eqref{eq:first_bound_for_rho0} that
\[
\rho_0 \le \sum_{1 \le j < k \beta_0} \frac{1}{\beta^{\beta} (1-\beta)^{1-\beta}} f(89)
  = \frac{88^{\beta}}{\beta^{\beta} (1-\beta)^{1-\beta}} \left( {\textstyle{\frac{1}{89}}}
    + \textstyle{\frac{88}{89}} (89n)^{-c_4 \beta} \right)
  \quad \text{when $q = 89$.}
\]

It will now be shown that
\begin{equation}
\label{eq:bound_for_89}
  \sum_{1 \le j < k \beta_0} \frac{88^{\beta}}{\beta^{\beta} (1-\beta)^{1-\beta}} \left( \frac{1}{89} +
   \frac{88}{89} (89n)^{-c_4 \beta} \right) \le \beta^{\beta}
\end{equation}
when $c_4 = \frac{13}{6}$ and $0 < x \le \beta_0 = \frac{12}{13}$. It follows by Lemma~\ref{lem:f_is_decreasing}
that
\[
 \rho_0 \le \sum_{1 \le j \le k \beta_0} \frac{1}{\beta^{\beta} (1-\beta)^{1-\beta}} f(q) \le \beta^{\beta}
\]
for $c_4 = \frac{13}{6}$ and $0 < x \le \beta_0 = \frac{12}{13}$ when $q \ge 91$ as well.

The process described in Subsection~\ref{ssec:beta_is_tiny} will be first be used to establish the
above inequality when $0 < \beta \le \frac{1}{2000}$. It follows by part~(i) of
Lemma~\ref{lem:continuing_to_get_rid_of_pesky_term} that $(1-x)^{-(1-x)} \le x^{\gamma x}$ when
$0 < x \le \frac{1}{2000}$ and $\gamma = -\frac{33}{250}$, so that this value may be used for~$\gamma$.

It also follows  by part~(n) of Lemma~\ref{lem:continuing_to_get_rid_of_power_of_q} that $72^x \le x^{\delta x}$
when $0 < x \le \frac{1}{2000}$ and $\delta = -\frac{3}{5}$, so that this value can be chosen for~$\delta$. If
$c_4 = \frac{13}{6}$ then
\[
 c - 2 + {\textstyle{\frac{\delta}{q-1} + \frac{\gamma q}{q-1}}} ={\textstyle{\frac{239}{66000}}} > 0.003,
\]
so that the first process to verify the above relationship, described in Section~\ref{sec:modified_proof}, can be
applied with these values. Since $\frac{1}{2000} < \frac{1}{40}$ it suffices to note that
$f_1\left(\frac{1}{2000}\right) > 0.00001$ and $f_1\left(\frac{1}{40}\right) < -0.00005$ --- for it then follows
by Lemma~\ref{lem:correctness_of_first_approximation} that the inequality at
line~\eqref{eq:bound_for_89} is satisfied when $0 < \beta \le \frac{1}{2000}$.

The process described in Subsection~\ref{ssec:beta_is_moderate} can now be used to establish
the above inequality when $\frac{1}{2000} < \beta \le \frac{12}{13}$, as needed to establish the claimed result.
Since $F_1'\left(\frac{12}{13}\right) < -1$, the function $F_1$ is decreasing over this interval, for
every choice of~$\eta$. Since $\frac{12}{13} > \frac{1}{7} = 1 + \frac{1}{1- (13/6)}$ and
$F_2'\left(\frac{1}{7}\right) > 4.5$, the function~$F_2$ is increasing over this interval for every choice
of~$\eta$.

It now suffices to confirm that if $\eta = \frac{1}{80}$ then $F_1\left(\frac{1}{200}\right) > 0.02$ and
$F_2\left(\frac{1}{2000}\right) > 0.001$, so that $F_1$ and~$F_2$ are both non-negative over the interval
$\frac{1}{2000} \le \beta \le \frac{1}{200}$.

It next suffices to confirm that if $\eta = \frac{1}{28}$ then $F_1\left(\frac{11}{100}\right) > 0.07$ and
$F_2\left(\frac{1}{200}\right) > 0.0005$, so that $F_1$ and~$F_2$ are both non-negative over the interval
$\frac{1}{200} \le \beta \le \frac{11}{100}$.

It next suffices to confirm that if $\eta = \frac{2}{5}$ then $F_1\left(\frac{11}{20}\right) > 0.09$ and
$F_2\left(\frac{11}{100}\right) > 0.01$, so that $F_1$ and~$F_2$ are both non-negative over the interval
$\frac{11}{100} \le \beta \le \frac{11}{20}$.

It next suffices to confirm that if $\eta = \frac{12}{13}$ then $F_1\left(\frac{6}{7}\right) > 0.02$ and
$F_2\left(\frac{11}{20}\right) > 0.02$, so that $F_1$ and~$F_2$ are both non-negative over the interval
$\frac{11}{20} \le \beta \le \frac{6}{7}$.

It next suffices to confirm that if $\eta = \frac{69}{70}$ then $F_1\left(\frac{183}{200}\right) > 0.005$ and
$F_2\left(\frac{6}{7}\right) > 0.005$, so that $F_1$ and~$F_2$ are both non-negative over the interval
$\frac{6}{7} \le \beta \le \frac{183}{200}$.

Finally, it suffices to confirm that if $\eta = \frac{101}{100}$ then $F_1\left(\frac{12}{13}\right) > 0.0006$ and
$F_2\left(\frac{183}{200}\right) > 0.02$, so that $F_1$ and~$F_2$ are both non-negative over the interval
$\frac{183}{200} \le \beta \le \frac{12}{13}$, as needed to establish the claim when $q \ge 89$.

\end{document}